\documentclass[12pt,reqno]{amsart}
\usepackage[headings]{fullpage}
\usepackage{amssymb,amsmath,amscd,bbm,tikz}
\usepackage{graphicx}
\usepackage{texdraw}
\usepackage{pb-diagram}
\usepackage[all,cmtip]{xy}
\usepackage{url}
\usepackage[bookmarks=true,%
    colorlinks=true,%
    linkcolor=blue,%
    citecolor=blue,%
    filecolor=blue,%
    menucolor=blue,%
    urlcolor=blue,%
    breaklinks=true]{hyperref}
\usepackage{slashed}    
\usepackage{verbatim}

\newtheorem{theorem}{Theorem}[section]
\theoremstyle{definition}
\newtheorem{proposition}[theorem]{Proposition}
\newtheorem{lemma}[theorem]{Lemma}
\newtheorem{definition}[theorem]{Definition}
\newtheorem{remark}[theorem]{Remark}
\newtheorem{corollary}[theorem]{Corollary}
\newtheorem{conjecture}[theorem]{Conjecture}

\def\BZ{\mathbbm Z}
\def\BQ{\mathbbm Q}
\def\BR{\mathbbm R}
\def\BC{\mathbbm C}

\def\calF{\mathcal F}
\def\calW{\mathcal W}

\def\calI{\mathcal I}

\def\calC{\mathcal C}
\def\calE{\mathcal E}

\def\calM{\mathcal M}

\def\s{\sigma}

\def\ti{\widetilde}
\def\SL{\mathrm{SL}}

\def\tq{\tilde{q}}
\def\Res{\mathrm{Res}}

\def\a{\alpha}

\def\ve{\varepsilon}

\def\Re{\mathrm{Re}}
\def\Im{\mathrm{Im}}

\def\be{\begin{equation}}
\def\ee{\end{equation}}

\def\Bor{\mathcal{B}}
\def\Lap{\mathcal{L}}

\def\sfb{\mathsf{b}}

\def\GL{\mathrm{GL}}
\def\e{\bold e}

\def\diag{\mathrm{diag}}
\def\Av{\mathrm{Av}}
\def\Om{\Omega}
\def\rind{\rho}
\def\sma#1#2#3#4{\bigl(\smallmatrix#1&#2\\#3&#4\endsmallmatrix\bigr)} % small matrix
\def\mat#1#2#3#4{\begin{pmatrix}#1&#2\\#3&#4\end{pmatrix}} % 2x2 matrix

\def\comp{\mathrm{comp}}
\def\Sol{\mathrm{Sol}}
\def\Ker{\mathrm{Ker}}
\usepackage{accents}
\def\vM{\accentset{\longrightarrow}{M}}
\def\sfb{\mathsf{b}}

\makeatletter

\renewcommand\thepart{\@Roman\c@part}%
\renewcommand\part{%
   \if@noskipsec \leavevmode \fi
   \par
   \addvspace{6.7ex}%
   \@afterindentfalse
   \secdef\@part\@spart}
\def\@part[#1]#2{%
    \ifnum \c@secnumdepth >\m@ne
      \refstepcounter{part}%
      \addcontentsline{toc}{part}{Part~\thepart.\ #1}%
    \else
      \addcontentsline{toc}{part}{#1}%
    \fi
    {\parindent \z@ \raggedright
     \interlinepenalty \@M
     \normalfont
     \ifnum \c@secnumdepth >\m@ne
       \centering\large\scshape \partname~\thepart.%
       \hspace{1ex}%
     \fi%
     \large\scshape #2%
     \markboth{}{}\par}%
    \nobreak
    \vskip 4.7ex
    \@afterheading}
  \def\@spart#1{
  \refstepcounter{part}%
  \addcontentsline{toc}{part}{#1}%
    {\parindent \z@ \raggedright
     \interlinepenalty \@M
     \normalfont
     \centering\large\scshape #1\par}%
     \nobreak
     \vskip 4.7ex
     \@afterheading}
\renewcommand*\l@part[2]{%
  \ifnum \c@tocdepth >-2\relax
    \addpenalty\@secpenalty
    \addvspace{0.75em \@plus\p@}%
    \begingroup
      \parindent \z@ \rightskip \@pnumwidth
      \parfillskip -\@pnumwidth
      {\leavevmode
       \normalsize \bfseries #1\hfil \hb@xt@\@pnumwidth{\hss #2}}\par
       \nobreak
       \if@compatibility
         \global\@nobreaktrue
         \everypar{\global\@nobreakfalse\everypar{}}%
      \fi
    \endgroup
  \fi}

\def\l@subsection{\@tocline{2}{0pt}{2pc}{6pc}{}}
\makeatother

\begin{document}
\title[Modular $q$-holonomic modules]{
       Modular $q$-holonomic modules}
\author{Stavros Garoufalidis}
\address{% Max Planck Institute for Mathematics \\
         % Vivatsgasse 7, 53111 Bonn, GERMANY \newline
  International Center for Mathematics, Department of Mathematics \\
  Southern University of Science and Technology \\
  Shenzhen, China \newline
  {\tt \url{http://people.mpim-bonn.mpg.de/stavros}}}
\email{stavros@mpim-bonn.mpg.de}
\author{Campbell Wheeler}
\address{Max Planck Institute for Mathematics \\
         Vivatsgasse 7, 53111 Bonn, Germany \newline
         {\tt \url{http://guests.mpim-bonn.mpg.de/cjwh}}}
\email{cjwh@mpim-bonn.mpg.de}


\date{31 March 2022}%\today }
%\dedicatory{}

\begin{abstract}
  We introduce the notion of modular $q$-holonomic modules 
  whose fundamental matrices define a cocycle with improved analyticity properties
  and show that the generalised $q$-hypergeometric equation, as well as three
  key $q$-holonomic modules of complex Chern--Simons theory are modular. 
  This notion explains conceptually recent structural properties of quantum
  invariants of knots and 3-manifolds, and of exact and perturbative Chern--Simons
  theory~\cite{GGM:peacock,GGM,GGMW:trivial,GZ:kashaev}, and in addition provides
  an effective method to solve the corresponding linear $q$-difference equations.
  An alternative title of our paper, emphasising the equations rather than the
  modules, is: 
  \medskip
  \begin{center}
\fbox{
Modular linear $q$-difference equations
}
\end{center}
\end{abstract}

\maketitle

{\footnotesize
\tableofcontents
}

%%%%%%%%%%%%%%%%%%%%%%%%%%%%%%%%%%%%%%%%%%%%%%%%%%%%%%%%%%%%%%%%%%%%%%%%%%%%
%%%%%%%%%%%%%%%%%%%%%%%%%%%%%%%%%%%%%%%%%%%%%%%%%%%%%%%%%%%%%%%%%%%%%%%%%%%%

\section{Introduction}
\label{sec.intro}

\subsection{Summary}
We introduce a new class of linear $q$-difference equations and (corresponding
$q$-holonomic modules) which we call \emph{modular}, with several key features.

\begin{itemize}
  \item
 Their fundamental solutions at 0 and infinity have explicitly
  computable monodromy.
\item
  A natural $\SL_2(\BZ)$-cocycle constructed from a fundamental
  solution extends from $\BC\setminus\BR$ to the complex plane minus a ray in
  the real numbers.
\item
  Their fundamental solutions are meromorphic and their residues are
  expressed in terms of the solutions themselves (kind of 'resurgence').
\end{itemize}
Modular $q$-holonomic modules are abundant. We show that the generalised
$q$-hypergeometric equation ~\eqref{gener.qhyper} (and in particular
the $q$-hypergeometric equation) is modular and self-dual; see
Theorem~\ref{thm.heine} below. Among other things, this implies an improved
analyticity for the $q$-hypergeometric function
\be
\label{2phi1}
  {}_{2}\phi_{1}(a,b;c;q,t) =
  \sum_{k=0}^{\infty}\frac{(a;q)_{k}(b;q)_{k}}{(c;q)_{k}(q;q)_{k}}t^{k} \,,
\ee
namely, that the bilinear combination
\be
\begin{aligned}
  &\frac{(q;q)_{\infty}(c;q)_{\infty}}{(a;q)_{\infty}(b;q)_{\infty}}
  \frac{(\tq\tilde{a};\tq)_{\infty}(\tq\tilde{b};\tq)_{\infty}}{
    (\tq;\tq)_{\infty}(\tq\tilde{c};\tq)_{\infty}}{}_{2}\phi_{1}(a,b;c;q,t)
  {}_{2}\phi_{1}(\tilde{a}^{-1},\tilde{b}^{-1};\tilde{c}^{-1};\tq,\tq\tilde{a}
  \tilde{b}\tilde{c}^{-1}\tilde{t})\\
  &+\tau\frac{\theta(q^{-1}ct;q)(q;q)_{\infty}^{3}}{\theta(t;q)\theta(q^{-1}c;q)}
  \frac{(q^{2}c^{-1};q)_{\infty}(q;q)_{\infty}}{
    (qac^{-1};q)_{\infty}(qbc^{-1};q)_{\infty}}
  \frac{\theta(\tq\tilde{c};\tq)\theta(\tilde{t};\tq)}{
    \theta(\tq\tilde{c}\tilde{t};\tq)
    (\tq;\tq)_{\infty}^3}\frac{(\tilde{a}\tilde{c}^{-1};\tq)_{\infty}
    (\tilde{b}\tilde{c}^{-1};\tq)_{\infty}}{(\tq^{-1}\tilde{c}^{-1};\tq)_{\infty}
    (\tq;\tq)_{\infty}}\\
  &\quad\times{}_{2}\phi_{1}(qac^{-1},qbc^{-1};q^2c^{-1};q,t)
  {}_{2}\phi_{1}(\tq\tilde{a}^{-1}\tilde{c},
  \tq\tilde{b}^{-1}\tilde{c};\tq^2\tilde{c};
  \tq,\tq\tilde{a}\tilde{b}\tilde{c}^{-1}\tilde{t})
\end{aligned}
\ee
(which a priori is a meromorphic function of $\tau \in \BC\setminus\BR$)
extends to $\tau\in\BC\setminus(-\infty,0]$.
We also show that modular linear $q$-difference equations and the corresponding
$q$-holonomic modules appear naturally in complex Chern-Simons theory. We illustrate
with three examples (Theorems~\ref{thm.ex1}, ~\ref{thm.ex2} and ~\ref{thm.41} below)
whose corresponding cocycles are the Faddeev quantum dilogarithm, the Appell-Lerch
sums, and the Andersen--Kashaev state integrals of the $4_1$ knot.

We expect that all proper (i.e., basic) $q$-hypergeometric modules are modular, and
in particular the ones that appear in the quantum differential equation in Quantum
Cohomology, or the linear q-difference equation for the small J-function of
Quantum K-Theory.

\subsection{Motivation}
  
In a recent talk~\cite{Okounkov:talk,Okounkov:notes}, Okounkov asked the
question:

\begin{center}
\emph{What does it take to solve a $q$-difference equation?}
\end{center}

Solving a linear equation usually means giving a basis of solutions,
say at $t=0$ and $t=\infty$ (the only two canonical points, fixed under the
shift transformation
$t \mapsto q t$), and to compute the monodromy (i.e., the connection)
between $t=0$ and $t=\infty$ in
terms of ``known'' functions. This problem has a rich history
with an interesting balance between concrete special functions and the abstract,
which the reader may consult in Okounkov's talks and their references.

A key to the solution is to choose linear $q$-difference equations of natural
origin, for instance the ones appearing in quantum cohomology~\cite{OP},
in Kontsevich's talks on resurgence~\cite{Kontsevich:talks}, or in
quantum topology and Chern--Simons theory~\cite{Ga:CS,GL}.

Recently, it was observed that the solutions of some 
linear $q$-difference equations concerning quantum knot invariants
~\cite{GZ:qseries,GZ:kashaev} or resurgence and Borel resummation of
perturbative Chern--Simons theory~\cite{GGM:peacock,GGM,GGMW:trivial}, have
an improved analyticity property. Roughly speaking, this means that 
some holomorphic functions on $\BC\setminus\BR$ extend on the cut plane
$\BC'=\BC\setminus (-\infty,0]$. This extension property has been formalised recently
by Zagier as holomorphic quantum modular forms~\cite{Z:HQMF}.

Our attempt to understand the mechanism behind this property abstractly led
us to the notion of a modular $q$-holonomic module. Quite by accident, this
suggests an answer to the question posed above, namely:

\begin{center}
\emph{Modularity can solve effectively a $q$-difference equation}. 
\end{center}

\subsection{Definition and properties}
\label{sub.def}

To explain the new concept, consider the linear $q$-difference equation
\be
\label{qdiff}
\s X = A X
\ee
for a vector-valued function $X=X(t,q)$, where $(\s X)(t,q)=X(qt,q)$ denotes the
shift operator, $A(t,q) \in \GL_r(\BQ(t,q))$ and $q$ is a nonzero complex number
with $|q|\neq1$. 
Fix a fundamental matrix solution $U$ to~\eqref{qdiff} at $t=0$, which we assume
is filtration-preserving and of weight $\kappa_{U}=\diag(\tau^{\kappa_{U,1}},\dots,
\tau^{\kappa_{U,r}})$ and define 
\be
\label{Om}
\Om_{U,\gamma} =(U|_{\kappa_{U}}\gamma)U^{-1},
\qquad \gamma=\sma abcd \in \SL_2(\BZ)
\ee
where the slash operator $|_{\kappa_{U}}\gamma$ is defined in Section~\ref{sub.slash}
below. The map $\gamma \mapsto \Om_{U,\gamma}$ is a cocycle of $\SL_2(\BZ)$, i.e.,
it satisfies
\be
\label{Omcoc}
\Om_{U,\gamma \gamma'}(z,\tau)
= \Om_{U,\gamma}(\gamma'(z,\tau))  \Om_{U,\gamma'}(z,\tau) 
\ee
for all $\gamma, \gamma' \in \SL_2(\BZ)$.
Likewise, let $V$ denote a filtration-preserving fundamental solution at $t=\infty$,
and $\vM=V^{-1}U$ denote the monodromy matrix (often called ``connection matrix'').
The monodromy is always an elliptic function (i.e., $\s$-invariant).
If it satisfies the equation
\be
\label{Mmod}
\vM = \Delta_{\kappa_V,\gamma} (\vM|_{\kappa_U} \gamma),
\qquad \gamma=\sma abcd \in \SL_2(\BZ)
\ee
for some $\gamma=\sma abcd \in \SL_2(\BZ)$ then the cocycle matrices associated to
$U$ and $V$ are equal: 
\be
\label{OmUV}
\Om_{U,\gamma} = \Om_{V,\gamma} \qquad \gamma=\sma abcd \in \SL_2(\BZ) \,.
\ee
(The converse holds, too, and the equivalence of~\eqref{OmUV} and~\eqref{Mmod}
holds for each fixed $\gamma$). A priori, $\Om_{\gamma}=\Om_{U,\gamma}=\Om_{V,\gamma}$
is a holomorphic function of $\tau \in \BC\setminus\BR$ and a
meromorphic function of $z$ (with $t=\e(z)$) with poles in a finite union of
translates of the lattice $\BZ+\tau\BZ$. For an element
$\gamma=\sma abcd \in \SL_2(\BZ)$, we let $\BC_\gamma$ denote the cut
plane $\BC\setminus (-\infty,-d/c]$ (if $c>0$),
$\BC\setminus [-d/c,\infty)$ (if $c<0$) and $\BC\setminus\BR$ when $c=0$. 
The next definition concerns an improved analytic property of $\Om_{\gamma}$.
\begin{definition}
  \label{def.qmm}
  We say that a $q$-difference equation is \emph{modular} if for all $\gamma
  \in \SL_2(\BZ)$, its cocycle $\Om_{\gamma}$ extends to a meromorphic function
  of $(z,\tau)\in\BC\times\BC_\gamma$ with potential poles at
  $z \in S_\gamma + \BZ + \tau \BZ$ for a finite set $S_\gamma$. 
\end{definition}

\noindent
We make several comments concerning this definition. 

\noindent 1. %$\bullet$ 
For a modular $q$-difference equation, the two sides of $\tau \in \BC\setminus\BR$
communicate: the restriction of $\Om_{\gamma}$ on the one side of the plane
$\Im(\tau)>0$ uniquely determines (and is determined by) the function on the
other side $\Im(\tau)<0$.

\noindent 2. %$\bullet$ 
In a sense, our notion of modular linear $q$-difference equation is a $q$-analogue
of the second, third and fourth order modular linear differential 
equations studied by Kaneko, Nagatomo, Zagier and others~\cite{Arike:fourth-order, 
Kaneko:third-order,Kaneko-Zagier} in relation to Conformal Field Theory 
and Vertex Operator Algebras. The characters of a VOA (under a mild $C_2$-condition)
are solutions to a modular linear differential equation
~\cite{Dong:modular-invariance,Zhu:modular-invariance}. An alternative title of
our paper could be
``Modular linear $q$-difference equations'' in direct analogy with the
modular linear differential equations of CFTs and VOAs. On the other hand, the
  current title emphasises the ``module'' aspect as opposed to the ``equation
  aspect''. 

\noindent 3. %$\bullet$ 
Our notion explains conceptually the improved analyticity properties conjectured
in relation to the structure of exact and perturbative invariants of Chern--Simons
theory~\cite{GGM:peacock,GGM,GGMW:trivial,GZ:qseries,GZ:kashaev}, as well as the
work of Gukov et al relating logarithmic CFTs and VOAs to Chern--Simons
theory~\cite{Cheng:VOA}. 
  
\noindent 4. %$\bullet$
Modular linear differential equations seem rare. On the other hand,
modular linear $q$-difference equations appear abundant: we conjecture
(see the end of Section~\ref{sec.intro}) that the $\hat A$-polynomial of a knot
(i.e., the linear $q$-difference equation satisfied by the colored Jones polynomial
of a knot~\cite{GL}) is a modular linear $q$-difference equation. This brings
a new perspective to the Jones polynomial of a knot.

\noindent 5. %$\bullet$
The modularity of the monodromy for modular $q$-holonomic modules is a restricted
condition, which leads to the complete determination of the monodromy even when a
fundamental solution is defined by $q$-Borel resummation. This in turn
completely determines the so-called $q$-Stokes phenomenon coming from the
change of the ray of the $q$-Laplace transform.

\noindent 6. %$\bullet$
The cocycle $\Omega_{U,\gamma}$ as a function of a fundamental solution of a
linear $q$-difference equation appears new and different from the matrices
considered in Etingof~\cite{Etingof} that generate the Galois group of the equation.

\noindent 7. %$\bullet$
The above definition involves all elements of $\SL_2(\BZ)$, although it
imposes no improved extension when $\gamma_{2,1}=0$.  The third part of the next
theorem (under the hypothesis that $\Om_T=I$ which is satisfied in all of our
examples), rephrases the modularity of Definition~\ref{def.qmm}
in terms of $\Om_S$ alone, where $S=\sma 0{-1}10$ and $T=\sma 1101$ of
$\SL_2(\BZ)$ are the standard generators of $\SL_2(\BZ)$. Recall that
$\BC'=\BC_S=\BC\setminus(-\infty,0]$.  

\begin{theorem}
\label{thm.ST}
{\rm (a)} If Equation~\eqref{Mmod} for the monodromy holds for $\gamma=S$
and $\gamma=T$, then it holds for all $\gamma \in \SL_2(\BZ)$.
\newline
{\rm (b)} If $\Om$ is an $\SL_2(\BZ)$-cocycle with $\Om_T=I$, then $\Om_S$
satisfies the 4-term and the 3-term functional equations
%% see Mathematica file: wheeler/SL2Z.Matrices.Cocycle.nb
\begin{subequations}
\begin{align}
  \label{OmS1}
  1 &= 
  \Om_S\left(-\frac{z}{\tau},-\frac{1}{\tau}\right)
  \Om_S\left(-z,\tau\right)
  \Om_S\left(\frac{z}{\tau},-\frac{1}{\tau}\right)
  \Om_S(z,\tau) \\
  \label{OmS2}
  \Om_S(z,\tau) &= \Om_S\left(\frac{z}{\tau+1}, \frac{\tau}{\tau+1}\right)
  \Om_S(z,\tau+1) \,.
\end{align}
\end{subequations}
Conversely, given $\Om_T=I$ and $\Om_S$ that satisfies the above functional
equations, there is a $\SL_2(\BZ)$-cocycle with those values.
\newline
{\rm (c)} If $\Omega$ is a cocycle such that $\Om_T=I$ and $\Om_S$ extends as
a meromorphic function of $(z,\tau) \in \BC \times \BC'$, then $\Om_\gamma$
extends as a meromorphic function of $(z,\tau) \in \BC \times \BC_\gamma$
for all $\gamma \in \SL_2(\BZ)$. 
\end{theorem}

Note that with the assumptions of part (c) above, we have
\be
\label{Om-I}
\Om_{-I}(z,\tau) = \Om_S\left(\frac{z}{\tau},-\frac{1}{\tau}\right)
\Om_S(z,\tau)
\ee
which is meromorphic on $\BC\times (\BC\setminus\BR)$, in general does not extend
any further; see for instance the cocycle of Theorem~\ref{thm.ex1} below.

\noindent 8. %$\bullet$
Our final comment concerns the distinction between a linear $q$-difference equation
to a $q$-holonomic module. A linear $q$-difference equation leads to a pair
$(M,e)$ of a $q$-holonomic module $M$ over the $q$-Weyl algebra
$\calW=\BQ(q,t) \langle \s \rangle/(\s t-q t \s)$ and a cyclic vector $e$ of $M$. 
The category $\calM$ of $q$-holonomic modules is an abelian category,
which admits a multivariable extension closed under the usual operations on sheaves.
For a detailed introduction, see~\cite{Coutinho} and~\cite{Sabbah}, and
also~\cite{GL:survey}. Whereas $D$-modules over the Weyl algebra have
only one dual (see for example, ~\cite[Sec.2.2]{Singer:galois-differential}), modules
over the $q$-Weyl algebra have two duals $M^{\wedge}$ (the usual $\mathrm{Hom}$-dual
as in the case of $D$-modules) and a dual $M^{\vee}$ coming from the involution
$q \mapsto q^{-1}$, discussed in Section~\ref{sub.duality} below.
%, both preserving $q$-holonomicity.
The next lemma summarises some categorical properties of modular $q$-holonomic
modules.

\begin{lemma}
\label{lem.cat}
  \rm{(a)} Modularity of a $q$-holonomic module is independent of the choice
  of a cyclic vector.
  \newline
  \rm{(b)} If $0 \to M' \to M \to M'' \to 0$ is a short exact sequence of
  $q$-holonomic modules and $M$ is modular, then $M$ and $M'$ are also modular.
  \newline
  \rm{(c)} $M$ is a modular, then $M^\wedge$ and $M^\vee$ are modular with
  cocycles given by
  \be
\label{Omduals}
\Om^\wedge(z,\tau):=\Om(z,-\tau), \qquad \text{and}
\qquad \Om^\vee=(\Om^{-1})^t
\ee
where $\Om$ is a cocycle of $M$. 
\end{lemma}

\subsection{The generalised $q$-hypergeometric equation is modular}
\label{sub.qhyper}

Contrary to what one might expect, modular $q$-holonomic modules are abundant.
This section concerns the generalised $q$-hypergeometric equation

\be
\label{gener.qhyper}
\left(\prod_{j=0}^{r-1}(1-q^{-1}b_{j}\sigma)-t\prod_{j=1}^{r}(1-a_{j}\sigma)\right)f
=0
\ee
for a function $f=f(t,a,b,q)$ where $a=(a_1,\dots,a_r)$ and $b=(b_0,\dots,b_{r-1})$
with $b_0=q$, where $\s$ is the operator that shifts $t$ to $qt$. This equation
is a $q$-deformation of the hypergeometric equation (a regular singular linear
differential equations with singularities at $0,1,\infty$), itself a generalisation
of the Gauss hypergeometric equation. The hypergeometric equation has a rich history
related to, among other things, periods in algebraic geometry and the
in the Gauss-Manin connection of the middle cohomology of the Dwork family of
smooth hypersurfaces; see for example~\cite{Beukers,Katz:another} and references
therein. 

The generalised $q$-hypergeometric equation~\eqref{gener.qhyper} was
introduced and studied by Heine, who shows that this equation has a solution the
function, analytic at $t=0$, given by the $q$-hypergeometric series
\be
\label{phirr}
{}_{r}\phi_{r-1}(a;b;q,t) =
\sum_{k=0}^{\infty}
\frac{(a_{1};q)_{k}\dots(a_{r};q)_{k}}{(b_{1};q)_{k}\dots(b_{r-1};q)_{k}(q;q)_{k}}t^{k}
\,.
\ee
The generalised $q$-hypergeometric equation and the $q$-hypergeometric series
are classical examples of functions that appear in many areas of research, and
for a comprehensive treatment, we refer the reader to the book of Gasper--Rahman
~\cite{Gasper} and references therein. 

Consider the matrices 
\be
\label{UVheine}
  U=W(f^{(1)},f^{(qb_{1}^{-1})},\dots,f^{(qb_{r-1}^{-1})}), \qquad
  V=W(g^{(a_{1}^{-1})},\dots,f^{(a_{r}^{-1})})
\ee
where $W$ denotes the Wronskian of
$r$ functions $f_1,\dots,f_r$, which is an $r \times r$ matrix defined by
$W(f_1,\dots,f_r)_{ij}=(\s^{i-1} f_j)$ for $i,j=1,\dots,r$ and
\be
\label{heine.bjsols}
  f^{(q^{-1}b_{j})}(t,a,b,q)
  =
  B_{j}(a,b,q)\frac{\theta(q^{-1}b_{j}t;q)}{\theta(t;q)}
  {}_{r}\phi_{r-1}(qa/b_{j};qb/b_j;q,t)
\ee
for $j=0,\dots,r-1$ with $B_{0}=1$ and, for $j>0$,
\be
\label{Bjabq}
B_{j}(a,b,q) =
\left(\prod_{i=1}^{r}\frac{(a_{i};q)_{\infty}(qb_{i-1}/b_{j};q)_{\infty}}{
    (qa_{i}/b_{j};q)_{\infty}(b_{i-1};q)_{\infty}}\right)
\frac{(q;q)_{\infty}^{3}}{\theta(q^{-1}b_{j};q)} \,,
\ee
and 
\be
\label{heine.ajsols}
g^{(a_{j}^{-1})}(t,a,b,q) =
A_{j}(a,b,q)\frac{\theta(q^{-1}a_{j}t;q)}{\theta(q^{-1}t;q)}
{}_{r}\phi_{r-1}(qa_{i}/b;qa_{i}/a;q,qb_{1}\dots b_{r-1}a_{1}^{-1}
\dots a_{r}^{-1}t^{-1})
\ee
for $j=1,\dots,r$ with
\be
A_{j}(a,b,q) =
\frac{(q;q)_{\infty}^2}{\theta(q^{-1}a_{j};q)
  \prod_{i=1,i\neq j}^{r}(a_{i}/a_{j};q)_{\infty}}
\prod_{i=1}^{r}
\frac{(a_{i};q)_{\infty}(b_{i-1}/a_{j};q)_{\infty}}{(b_{i-1};q)_{\infty}}\,.
\ee
The calculation of the monodromy is a classical result and follows from
~\cite[Eq.(4.5.2)]{Gasper}. The monodromy
$\vM(t,a,b,q) = V(t,a,b,q)^{-1}U(t,a,b,q)$ is given explicitly by
\be
\label{Mheine}
\vM(t,a,b,q) =
\begin{pmatrix}
  1 & \frac{(q;q)_{\infty}^{3}\theta(b_{1}t;q)\theta(a_{1}/b_{1}t;q)
  \theta(a_{1};q)}{\theta(t;q)\theta(a_{1}t;q)\theta(b_{1};q)\theta(a_{1}/b_{1};q)}
  & \dots & \frac{(q;q)_{\infty}^{3}\theta(b_{r-1}t;q)\theta(a_{1}/b_{r-1}t;q)
  \theta(a_{1};q)}{\theta(t;q)\theta(a_{1}t;q)\theta(b_{r-1};q)\theta(a_{1}/b_{r-1};q)}
  \\
1 & \frac{(q;q)_{\infty}^{3}\theta(b_{1}t;q)\theta(a_{2}/b_{1}t;q)
  \theta(a_{2};q)}{\theta(t;q)\theta(a_{2}t;q)\theta(b_{1};q)\theta(a_{2}/b_{1};q)}
& \dots & \frac{(q;q)_{\infty}^{3}\theta(b_{r-1}t;q)\theta(a_{2}/b_{r-1}t;q)
  \theta(a_{2};q)}{\theta(t;q)\theta(a_{2}t;q)\theta(b_{r-1};q)\theta(a_{2}/b_{r-1};q)}
\\
: & : & : & \dots
\\
1 & \frac{(q;q)_{\infty}^{3}\theta(b_{1}t;q)\theta(a_{r}/b_{1}t;q)
  \theta(a_{r};q)}{\theta(t;q)\theta(a_{r}t;q)\theta(b_{1};q)\theta(a_{r}/b_{1};q)}
& \dots & \frac{(q;q)_{\infty}^{3}\theta(b_{r-1}t;q)\theta(a_{r}/b_{r-1}t;q)
  \theta(a_{r};q)}{\theta(t;q)\theta(a_{r}t;q)\theta(b_{r-1};q)\theta(a_{r}/b_{r-1};q)}
\end{pmatrix} \,.
\ee
Then from the modularity of the Dedekind $\eta$-function and the Jacobi
$\theta$-function the matrix $\vM$ satisfies the modular transformation~\eqref{Mmod}
with weights $\kappa_{U}=(0,1,\dots,1)$ and $\kappa_{V}=(0,\dots,0)$.

When $a,b$ are specialised at points where there are singularities, one can take
coefficients of the expansion around these points to define $U,V$.

Theorem~\ref{thm.heine} states that the generalised $q$-hypergeometric equation in
modular. The theorem will also hold when we specialise the values of $a,b$. This is
because the integral representations of the cocycle will be regular at these
specialisations and the factorisation of the integral will similarly involve the
expansion around these points as it will involve computation of residues of higher
order.

\begin{theorem}
\label{thm.heine}
The two cocycles~\eqref{Om} $\Om_{U}$ and $\Om_{V}$ are equal and modular.
% \be
% \begin{tiny}
% \begin{aligned}
% \int_{-i\sqrt{\tau}(\BR+i\ve)}
% \frac{(-q^{1/2}\e(x+\beta_{0});q)_{\infty}
% \dots(-q^{1/2}\e(x+\beta_{r-1});q)_{\infty}(-\tq^{1/2}
% \e(\frac{x+\alpha_{1}}{\tau});\tq)_{\infty}\dots(-\tq^{1/2}
% \e(\frac{x+\alpha_{r}}{\tau});\tq)_{\infty}}
% {(-\tq^{1/2}\e(\frac{x+\beta_{0}}{\tau});\tq)_{\infty}\dots(-\tq^{1/2}
% \e(\frac{x+\beta_{r-1}}{\tau});\tq)_{\infty}
% (-q^{1/2}\e(x+\alpha_{1});q)_{\infty}(-q^{1/2}\e(x+\alpha_{r});q)_{\infty}}
% \e\left(\frac{zx}{\tau}\right)
% dx \,.
% \end{aligned}
% \end{tiny}
% \ee
\end{theorem}

The proof of the above theorem is given in Section~\ref{sec.heine} and uses an
integral representation of the solutions of Equation~\eqref{gener.qhyper} in
terms of a special function, the Faddeev quantum dilogarithm~\cite{Faddeev,FK-QDL},
the factorisation of the corresponding integrals (so-called state integrals)
as a bilinear combination functions of $z,\tau$ and $(z/\tau,-1/\tau)$ along the
lines of~\cite{GK:qseries}, combined with an explicit description of the
self-duality of the corresponding $q$-holonomic module. What is more, the $U$ and the
$V$ cocycles are obtained by moving the contour of integration of the state integral
upwards or downwards, respectively, and this is one explanation of their
equality. We also remark that the above theorem is valid for $r=1$,
where the corresponding cocycle is none other than a ratio of two Faddeev
quantum dilogarithms.

\subsection{Modular $q$-holonomic modules in Chern--Simons theory}

In this section we present three $q$-holonomic modules whose cocycles play a key
role in complex Chern--Simons theory and prove their modularity. We will not need
a detailed knowledge of Chern--Simons theory, but focus on the fact that it is a
gauge theory with gauge group a complex Lie group, whose partition function
can be identified by a finite dimensional integrals of products of the Faddeev
quantum dilogarithm. A detailed exposition of Chern--Simons theory is given in the
work of Andersen--Kashaev and Beem, Dimofte and Pasquetti~\cite{AK,Dimofte:complex,
  Beem,Pasquetti}. The quasi-periodicity of the Faddeev quantum dilogarithm implies
that the partition function of complex Chern--Simons theory satisfies a
$q$-holonomic module of linear $q$-difference equations.

In the modular $q$-holonomic modules discussed below, we prove their modularity
\emph{and} at the same time define and compute fundamental matrices $U$ and $V$
of solutions algorithmically, and give explicit formula for their monodromy.
This is possible because of two key features, namely:
\begin{itemize}
\item[(a)]
  The fundamental matrices $U$ and $V$ are meromorphic functions with explicit
  poles and with residues expressed linearly in terms of $U$ and $V$.
  This is some kind of resurgence property and it is ultimately responsible for
  determining the monodromy.
\item[(b)]
  The monodromy is uniquely determined by its elliptic property, the explicit
  poles and principle parts, and its limiting values at $t=0$ or $t=\infty$. 
\end{itemize}

Our first module, the linear $q$-difference equation for the infinite
Pochhammer symbol
\be
\label{ex1}
%\big((1-qt)\sigma-1\big)(f)(t,q)=
(1-qt)f(qt,q)-f(t,q)=0 \,,
\ee
has trivial monodromy but its cocycle is a very interesting function, the Faddeev
quantum dilogarithm function which is the building block of the partition function
of complex Chern--Simons theory. With the definition of $f^{(0)}$ and $g^{(1)}$ given
in Equations~\eqref{ex1.f0} and~\eqref{ex1.g1} below, we define the fundamental
matrices $U(t,q)$ and $V(t,q)$ of solutions at $t=0$ and $t=\infty$ by 
\be
\label{UVex1}
U(t,q) = W(f^{(0)}(t,q)), \qquad
V(t,q) = W(g^{(1)}(t,q)) \,,
\ee
with equal weights $\kappa = 0$. In our theorems below,
we will use the terms ``a $q$-holonomic module is modular'' and ``its cocycle
in modular'' interchangeably. 

\begin{theorem}
\label{thm.ex1}
The monodromy matrix is $\vM(t,q)=(1)$. The two cocycles~\eqref{Om} $\Om_U$
and $\Om_V$ are equal, modular and when $\gamma=S$, they are given explicitly by
\be
\Om_{S}(z,\tau) = %\frac{(\ti q\ti t;\ti q)_{\infty}}{(qt;q)_{\infty}} =
\Phi_{\sfb}\left(\frac{iz}{\sfb}
  +\frac{i\sfb}{2}+\frac{1}{2i\sfb}\right)^{-1}
\ee
where $\Phi$ is the Faddeev quantum dilogarithm function~\cite{Faddeev, FK-QDL}.
\end{theorem}
  
In forthcoming work~\cite{GKZ:cocycle}, the cocycle $\Om_\gamma$ will be
identified with an $\SL_2(\BZ)$-extension of the Faddeev quantum dilogarithm,
and its modularity will be deduced independently, using properties of the
odd Eisenstein series. 

Our second module is the $q$-difference equation of the Appell-Lerch sums
\be
\label{ex2}
%(\sigma-1)(\sigma+t)(f)(t,q)=
f(q^2t,q)+(qt-1)f(qt,q)-tf(t,q)=0 \,.
\ee
The cocycle of~\eqref{ex2} is the Appell-Lerch sum, which is the building block
for the extension of the partition function of complex Chern--Simons theory
that detects the trivial flat connection, see~\cite{GGMW:trivial}.
In the above equation, some formal power series solutions are divergent, and their 
$q$-Laplace resummation is expressed in terms of an Appell-Lerch sum that depends
on an additional elliptic variable. The monodromy of this equation is an explicit
product of theta functions that depends on two elliptic variables, but the cocycle
is independent of them, and is expressed explicitly in terms of the Mordell
integral. Our results give a new interpretation of the Mordell integral and
Appell-Lerch sums emphasising the role of the linear $q$-difference equations
as opposed to the aspects of modular forms advocated by Zwegers in his
thesis~\cite{Zwegersthesis} and by Dabholkar--Murthy--Zagier~\cite{DMZ:blackholes}
in the study of mock modular forms and their incarnations in the mathematical
physics of black holes.

With the definition of $f^{(-1)}$, $f^{(0)}$, $g^{(0)}$ and $g^{(1)}$ given in
Equations~\eqref{ex2.f-1}, \eqref{ex2.f0}, \eqref{ex2.g0} and~\eqref{ex2.g1}
below, we define the fundamental matrices $U(t,\lambda,q)$ and $V(t,\mu,q)$
of solutions at $t=0$ and $t=\infty$ by 
\be
\label{UVex2}
U(t,\lambda,q) = W(f^{(0)}(t,\lambda,q), f^{(-1)}(t,q)), \qquad
V(t,\mu,q) = W(g^{(0)}(t,\mu,q), g^{(-1)}(t,q)) \,,
\ee
% and the two cocycles
% \be
% \label{WUVex2}
% \Om_{U,\gamma}(z,u,\tau) = (U|_{\kappa}\gamma )U^{-1},
% \qquad
% \Om_{V,\gamma}(z,v,\tau) = (V|_{\kappa}\gamma )V^{-1} 
% \ee
with equal weights $\kappa=(0,1)$.

\begin{theorem}
\label{thm.ex2}
{\rm (a)}
The monodromy matrix $\vM(t,\lambda,\mu,q)=V(t,\mu,q)^{-1}U(t,\lambda,q)$ is
given by
\be
\label{Mex2b}
\vM(t,\lambda,\mu,q) =
\begin{pmatrix}
1 & 0\\
* & 1\\
\end{pmatrix}
\ee
where 
\be
\label{M12ex2}
\vM_{2,1}(t,\lambda,\mu,q) =
-\frac{(q;q)_{\infty}^{3}\theta(q^{-1}t;q)\theta(\lambda ^{-1}\mu ;q)
  \theta(\lambda ^{-1}\mu ^{-1}t^{-1};q)}{\theta(\lambda ^{-1};q)\theta(\mu ;q)
  \theta(\lambda ^{-1}t^{-1};q)\theta(\mu ^{-1}t^{-1};q)} \,.
\ee
The matrix $\vM$ satisfies the modular transformation~\eqref{Mmod}.
\newline
{\rm (b)}
The two cocycles~\eqref{Om} $\Om_U$ and $\Om_V$ are equal, modular, 
independent of $\lambda$ and $\mu$ and when $\gamma=S$, are given explicitly
by
\be
\Om_{S}(z,\tau)
=
\begin{pmatrix}
0 & 1\\
1 & -\tilde{t}
\end{pmatrix}
\begin{pmatrix}
\tau^{-1} & 0\\
\frac{\e\left(\frac{1}{8}\right)}{\sqrt{\tau}}\tilde{t}^{-\frac{1}{2}}
\tq^{\frac{1}{8}}\e\left(\frac{z^2}{2\tau}\right)h(z,\tau) &
\frac{\e\left(\frac{1}{8}\right)}{\sqrt{\tau}}\e
\left(-\frac{(z+1/2-\tau/2)^{2}}{2\tau}\right)
\end{pmatrix}
\begin{pmatrix}
0 & 1\\
1 & -t
\end{pmatrix}^{-1}
\ee
where
\be
\label{mordell}
h(z,\tau)=\int_{\BR}\frac{e^{\pi i (\tau x^{2} + 2 i z x)}}{2\cosh(\pi x)}dx.
\ee
is the Mordell integral. 
\end{theorem}

Our third module is a linear $q$-difference equation associated to quantum
invariants of the simplest hyperbolic  knot $4_1$,
\be
\label{41x1}
tqf(qt;q) +(1-2t)f(t;q)+ tq^{-1}f(q^{-1}t;q) = 0 \,.
\ee
This equation, which is also related to the $q$-Hahn Bessel function, appeared
in homogeneous form in~\cite[Eqn.(11)]{GGM}. 
With the definition of $f^{(-1)}$, $f^{(1)}$ and $g^{(0,0)}$ and $g^{(0,1)}$ given in
Equations~\eqref{ex41.f-1}, \eqref{ex41.f1} and \eqref{ex41.g00}
below, we define the fundamental matrices $U(t,\lambda,q)$ and $V(t,q)$
of solutions at $t=0$ and $t=\infty$ by 
\be
\label{UVex41}
U(t,\lambda,q) = W(f^{(1)}(t,\lambda,q), f^{(-1)}(t,q)), \qquad
V(t,q) = W(g^{(0,0)}(t,q), g^{(0,1)}(t,q)) \,,
\ee
% with two cocycles
% \be
% \label{WUVex3}
% \Om_{U,\gamma}(z,u,\tau) = (U|_{\kappa}\gamma )U^{-1},
% \qquad
% \Om_{V,\gamma}(z,\tau) = (V|_{\kappa}\gamma )V^{-1} 
% \ee
and equal weights $\kappa=(1,0)$.
\begin{theorem}
  \label{thm.41}
{\rm (a)}
The monodromy matrix $\vM(t,\lambda,q)=V^{-1}(t,q)U(t,\lambda,q)$ is given by
\be
\label{M41b}
\vM(t,\lambda,q) =
\begin{pmatrix}
- 1 & 0\\
\vM_{2,1}(t,\lambda,q) & 1\\
\end{pmatrix}
\ee
where 
\be
\label{M412}
\vM_{2,1}(t,\lambda,q) =
\frac{\theta(qt;q)\theta(t\lambda;q)\theta(t\lambda q^{-1/2};q)
\theta(t\lambda^{2}q^{-1/2};q)}{\theta(t\lambda q^{1/4};q)
\theta(-t\lambda q^{1/4};q)\theta(t\lambda q^{-1/4};q)
\theta(-t\lambda q^{-1/4};q)\theta(q^{-1}\lambda;q)\theta(q^{-3/2}\lambda;q)} \,.
\ee
The matrix $M$ satisfies the modular transformations
\be
\label{M41c}
\Delta_{T^{2},\kappa}\vM|_{\kappa}T^2=\vM,\quad
\Delta_{STS,\kappa}\vM|_{\kappa}STS=\vM,\quad
\Delta_{TST,\kappa}\vM|_{\kappa}TST=\vM\,.
\ee
It follows that the $\SL_2(\BZ)$-orbit of $M$ consists
of three functions $M,\Delta_{T,\kappa}\vM|_{\kappa}T,\Delta_{S,\kappa}\vM|_{\kappa}S$.
\newline
{\rm (b)}
The two cocycles~\eqref{Om} $\Om_{\Av(U)}$ and $\Om_{V}$, defined by the averaging 
\be
\Av(U)=V\Av(\vM), \qquad \Av(\vM)=
\frac{1}{3}\left(\vM+\Delta_{T,\kappa}\vM|_{\kappa}T
  +\Delta_{S,\kappa}\vM|_{\kappa}S\right) \,,
\ee
are equal, modular, independent of $\lambda$ and when $\gamma=S$, they
are given by elementary functions times the state integrals
% \be
% %\mathcal{I}_{1,2}(z,\tau)=
% \int_{-i\sqrt{\tau}(\BR+i\ve)}
% \frac{(-q^{1/2}\e(x);q)_{\infty}^{2}}{(-\tq^{1/2}\e(x/\tau);\tq)_{\infty}^{2}}
% %\Phi_{\sqrt{\tau}}\left(\frac{ix}{\sqrt{\tau}}\right)
% \e\left(\frac{1}{2\tau}x^{2}+\frac{zx}{\tau}\right)dx.
% \ee
\be
\int_{\BR+i\ve}
\Phi_{\sfb}(x)^2
\exp\left(-\pi i x^{2}-2\pi\sfb^{-1}zx\right)dx.
\ee
\end{theorem}

The above can be phrased by saying that $M$ is modular on an index three subgroup
$\Gamma'$ of $\SL_2(\BZ)$ (conjugate to the $\theta$ subgroup and to the congruence
subgroup $\Gamma_0(2)$ of $\SL_2(\BZ)$) 
% \SG{which one, the theta subgroup of $\Gamma_0(2)$?}
% \CW{Unfortunately the other one. I'm not sure what it's name is.
%  There is a presentation $STS=1,T^2=1$}
and if $\gamma \in \Gamma'$ then $\Om_{U,\gamma}=\Om_{V,\gamma}$ is independent of
$\lambda$ and modular. The corresponding state integrals are defined in forthcoming
work~\cite{GKZ:cocycle}. However, the modularity of the module follows from
Theorem~\ref{thm.ST}. 

The equation~\eqref{41x1} has two important extensions, each containing important
information about the knot. The first extension is an inhomogeneous version, which
for the $4_1$ knot, takes the form (see~\cite[Eqn.(98)]{GZ:kashaev}
and~\cite[Sec.2.2]{GGMW:trivial})
\be
\label{41x1inhom}
tqf(qt;q) +(1-2t)f(t;q)+ tq^{-1}f(q^{-1}t;q) = 1 \,.
\ee
The solution $f^{(0)}$ at $t=0$ given in Equation~\eqref{g2} below is a resummation
of the Kashaev invariant of the $4_1$ knot, whereas the solution $g^{(0,2)}$ 
given in Equation~\eqref{f0x1} below is a $q$-series that appeared recently
in~\cite{GGMW:trivial} in relation to the asymptotic expansion of the Kashaev
invariant at the trivial representation. With those two solutions, we define
the fundamental matrices
\be
\label{UVex41b}
\begin{aligned}
  U(t,\lambda_{1},\lambda_{2},q)
  &=W(f^{(1)}(t,\lambda_{1},q),f^{(-1)}(t,q),f^{(0)}(t,\lambda_{2},q)) \\
V(t,q) &= W(g^{(0,0)}(t,q), g^{(0,1)}(t,q), g^{(0,2)}(t,q)) \,,
\end{aligned}
\ee
with equal weights $\kappa=(2,1,0)$. The monodromy matrix involves the Weierstrass
elliptic function $\wp$, a well-known function discussed in detail for example
in~\cite{Akhiezer}.

\begin{theorem}
\label{thm.41b} 
{\rm (a)} The monodromy matrix $\vM(t,\lambda_{1},\lambda_{2},q)
= V(t,q)^{-1}U(t,\lambda_{1},\lambda_{2},q)$ is given by
\be
\label{M413}
\vM(t,\lambda_{1},\lambda_{2},q) =
\begin{pmatrix}
- 1 & 0 & \wp(t,q)\\
\vM_{2,1}(t,\lambda_{1},q) & 1 &
\frac{1}{2}\frac{\wp'(t,q)-\wp'(\lambda_{2},q)}{\wp(t,q)-\wp(\lambda_{2},q)}\\
0 & 0 & 1
\end{pmatrix}
\ee
where $\vM_{2,1}$ is given by~\eqref{M412}. The matrix $M$ satisfies the
modular transformations~\eqref{M41c} with weight $\kappa=(2,1,0)$. 
\newline
{\rm (b)}
The two cocycles~\eqref{Om} $\Om_{\Av(U)}$ and $\Om_{V}$ are equal,
modular, independent of $\lambda_{1},\lambda_{2}$ and when $\gamma=S$,
they are given by combinations of elementary functions times the state integrals
% \be
% \int_{-i\sqrt{\tau}(\BR+i\ve)}
% \frac{(-q^{1/2}\e(x);q)_{\infty}^{2}}{(-\tq^{1/2}\e(x/\tau);\tq)_{\infty}^{2}}
% %\Phi_{\sqrt{\tau}}\left(\frac{ix}{\sqrt{\tau}}\right)^2
% \frac{\e\left(\frac{1}{2\tau}x^{2}+\frac{zx}{\tau}\right)}{1-\e(x/\tau-1/2-1/2\tau)}
% dx \,.
% \ee
\be
\int_{\BR+i\ve}
\Phi_{\sfb}(x)^2
\frac{\exp\left(-\pi i x^{2}-2\pi\sfb^{-1}zx\right)}{1+\tq^{1/2}
  \exp\left(-2\pi\sfb^{-1}x\right)}dx\,.
\ee
\end{theorem}

%% see Section 9 of frobqdiff.tex

The second and last extension of equation~\eqref{41x1inhom} is the addition of
an $x$-variable in $\BC^\times$, which topologically measures the holonomy of the
meridian of the knot, or the color of the colored Jones polynomial, and behaves
like a Jacobi variable. The new equations are now a two variable holonomic
system and take the form 
{\small
\begin{subequations}
\be
\label{41x.inhom}
\begin{aligned}
tqf(qt,x,q)+(1-(x^{-1}+x)t)f(t,x,q) + tq^{-1}f(q^{-1}t,x,q)
=
1 
\end{aligned}
\ee
\be
\label{41x2.inhom}
\begin{aligned}
(1-qx)(1-q^{-1}x^{2})&f(t,qx,q)\\
-(x-1)^2&(x+1)(x^2t-x-(q^{-1}+q)t-x^{-1}+x^{-2}t)f(t,x,q)\\
+&(1-qx^2)(1-q^{-1}x)f(t,q^{-1}x,q)
=
(1+x^{-1})(1-qx^2)(1-q^{-1}x^{2}) \,,
\end{aligned}
\ee
\be
\label{41x3.inhom}
(1-xq)(f(t,qx,q)-x^{-1}f(qt,qx,q))=(1-x^{-1})(f(t,x,q)-qxf(qt,x,q)) \,.
\ee
\end{subequations}
}
This is not a random system of equations, instead they are the defining equations
of the descendant colored Jones polynomial of the $4_1$ knot, and appeared
explicitly in~\cite[Eqn.(97)]{GGMW:trivial}. This is a $q$-holonomic module of
rank 3, with fundamental solutions $f^{(j)}(t,x,q)$ for $j=-1,0,1$ at $t=0$
and $g^{(0,x^{j})}(t,x,q)$ for $j=-1,0,1$ at $t=\infty$ defined
in Equations~\eqref{sols.41.x}, \eqref{f0x} and~\eqref{gGM}. 
With these solutions we can define the fundamental matrices with respect
to the shift in $t$ as
\be
\label{UVex41.x}
\begin{aligned}
  U(t,x,\lambda_{1},\lambda_{2},q)
  &=W(f^{(1)}(t,x,\lambda_{1},q),f^{(-1)}(t,x,q),f^{(0)}(t,x,\lambda_{2},q))\,, \\
V(t,x,q) &= W(g^{(0,x^{-1})}(t,x,q), g^{(0,x)}(t,x,q), g^{(0,1)}(t,x,q)) \,,
\end{aligned}
\ee
with weights $\kappa_{U}=(0,1,0)$ and $\kappa_{V}=(1,1,0)$.
The next theorem gives the properties of this monodromy.

\begin{theorem}
\label{thm.41c}
{\rm (a)} The monodromy matrix is given by
  \be
  {\small
\begin{aligned}
  \label{m3x3x}
&\vM(t,x,\lambda_{1},\lambda_{2},q)
=
\begin{pmatrix}
1 & -1 & 0\\
-1 & -1 & 0\\
0 & 0 & 1
\end{pmatrix} \times \\ 
&\begin{pmatrix}
  \frac{\theta(x^{-2};q)
    \theta(t;q)^2(q;q)_{\infty}^3}{2\theta(q^{-1}x;q)^2\theta(tx;q)\theta(tx^{-1};q)}
  & 0 & \frac{\theta(x^{-2};q)
  \theta(t;q)^2(q;q)_{\infty}^{3}}{2\theta(q^{-1}x;q)^2\theta(tx;q)\theta(tx^{-1};q)}\\
  m_{2,1}(t,x,\lambda_{1},q) & 1 &
  \left(\frac{\theta'(t\lambda_{2})}{\theta(t\lambda_{2})}
-\frac{\theta'(tx)}{2\theta(tx)}-\frac{\theta'(tx^{-1})}{2\theta(tx^{-1})}
-\frac{\theta'(\lambda_{2})}{\theta(\lambda_{2})}-\frac{1}{2}\right)\\
  0 & 0 & 1
\end{pmatrix}
\end{aligned}
  }
  \ee
  where $m_{2,1}(t,x,\lambda_{1},q)$ is the unique elliptic function in $t$
  satisfying $m_{2,1}(1,x,\lambda_{1},q)=0$ with simple poles at $t_0\in
\{x^{\pm}q^{\BZ},\pm\lambda^{-1}q^{\frac{1}{2}\BZ}\}$ and residues
$\rho_{t_0}:=\Res_{t=t_0}m_{2,1}(t,x,\lambda_{1},q)\frac{dt}{2\pi i t}$ given by
  \be
  \label{m21resx}
\begin{aligned}
  \rho_{x^{\pm}q^{\BZ}} &=
  \frac{-1}{2} \\
  \rho_{\pm\lambda_{1}^{-1}q^{-1/4-\BZ}} &=
  \frac{\theta(\pm q^{3/4}\lambda_{1}^{-1};q)\theta(\pm q^{-1/4}x;q)
\theta(\pm q^{-3/4}x;q)
\theta(\pm q^{-3/4}\lambda_{1};q)}{2\theta(x;q)\theta(x^{-1};q)
\theta(q^{-1}\lambda_{1};q)\theta(q^{-3/2}\lambda_{1};q)} \\
  \rho_{\pm\lambda_{1}^{-1}q^{-3/4-\BZ}} &=
  \frac{\theta(\pm q^{1/4}\lambda_{1}^{-1};q)\theta(\pm q^{-1/4}x;q)
  \theta(\pm q^{-3/4}x;q)
  \theta(\pm q^{-1/4}\lambda_{1};q)}{2\theta(x;q)\theta(x^{-1};q)
  \theta(q^{-1/2}\lambda_{1};q)\theta(q^{-1}\lambda_{1};q)}\,.
  \end{aligned}
  \ee
The monodromy satisfies the following modular transformations
\be
\label{M41x}
\Delta_{T^{2},\kappa_{V}}\vM|_{\kappa_{U}}T^2=\vM,\quad
\Delta_{STS,\kappa_{V}}\vM|_{\kappa_{U}}STS=\vM,\quad
\Delta_{TST,\kappa_{V}}\vM|_{\kappa_{U}}TST=\vM\,.
\ee
{\rm (b)}
The two cocycles~\eqref{Om} $\Om_{\Av(U)}$ and $\Om_{V}$ are equal,
modular, independent of $\lambda_{1},\lambda_{2}$ and when $\gamma=S$, they
are given by combinations of elementary functions times the state integrals
% \be
% \int_{-i\sqrt{\tau}(\BR+i\ve)}
% \frac{(-q^{1/2}\e(x+u);q)_{\infty}(-q^{1/2}\e(x-u);q)_{\infty}}{
%   (-\tq^{1/2}\e(\frac{x+u}{\tau});\tq)_{\infty}(-\tq^{1/2}
%   \e(\frac{x-u}{\tau});\tq)_{\infty}}
% \frac{\e\left(\frac{1}{2\tau}x^{2}+\frac{zx}{\tau}\right)}{1-\e(x/\tau-1/2-1/2\tau)}
% dx \,.
% \ee
\be
\int_{\calC}
\Phi_{\sfb}(x+i\sfb u)\Phi_{\sfb}(x-i\sfb u)
\frac{\exp\left(-\pi i x^{2}-2\pi\sfb^{-1} zx\right)}{1+\tq^{1/2}
  \exp\left(-2\pi\sfb^{-1}x\right)}dx\,.
\ee
\end{theorem}

An explicit formula for the function $m_{2,1}$ is given in Equation
~\eqref{m21explicit} below.

The next theorem identifies the function 
$f^{(0)}(t,x,\lambda_{1},q)$ with the $q$-Borel resummation of the descendant of
the colored Jones polynomial, and provides a lift of the colored Jones polynomial
(as a function of $N$ and $q$) to an analytic function of $x$ and $q$.  
Its extension to all knots will be discussed in forthcoming work. 

\begin{theorem}
  \label{thm.CJ41lift}
  The $N$-th colored Jones polynomial $J_{N}(q)$ of $4_{1}$ for $|q|<1$ is given by
  \be
  \label{J41f0}
J_{N}(q)=f^{(0)}(1,q^{N},\lambda_{2},q) \,.
  \ee
\end{theorem}
Note that there is no $\lambda_{2}$ dependence on the left hand side of
~\eqref{J41f0} and that $\lambda_{2}$ parametrises a family of analytic
continuations of the colored Jones polynomial as we vary $x$ away from $q^{N}$.

The cocycles given in the previous theorems reveal the relation of the
following three special functions (all being entries of a matrix-valued
$S$-cocycle of a modular $q$-difference equation)
  \begin{itemize}
\item[(a)] the Fadeev quantum dilogarithm
\item[(b)] the Mordell integral
\item[(c)] the Andersen-Kashaev state integral. 
\end{itemize}

We end this section with some comments regarding modular $q$-holonomic modules.
It is not obvious that $q$-holonomic modules exist or occur naturally. Yet,
they are abundant, for instance all proper $q$-hypergeometric multidimensional
sums are $q$-holonomic; see  Zeilberger--Wilf~\cite{WZ}. A detailed introduction
of $q$-holonomic modules and their functional and closure properties can be found
in~\cite{GL:survey,Koutschan:HF,PWZ}. 

Regarding the occurence of $q$-holonomic modules, they are often given in the form
$M_f=\calW f$ where $f: \BZ^r \to \BQ(q)$ is a quantum invariant. For instance,
if $f$ denotes the colored Jones polynomial of a link (colored with an arbitrary
representation of a fixed simple Lie group), or the colored HOMFLY-PT polynomial
of a link (colored by arbitrary partitions with a fixed number of rows or columns),
the corresponding modules $M_f$ are $q$-holonomic; see~\cite{GL} and~\cite{GLL}.
In addition, the special $q$-hypergeometric sums (so-called Nahm sums) studied
in~\cite[Sec.4.7]{GZ:kashaev} that depend on the upper-half of a symplectic matrix,
are $q$-holonomic. Moreover, the linear $q$-difference equations in quantum cohomology
or in quantum $K$-theory are often (and perhaps always?) specialisations of
$q$-hypergeometric series; see~\cite{OP,GS:quintic,Iritani:rec,Wen}.
Finally, the $q$-GKZ modules of Gelfand--Kapranov--Zelevinsky~\cite{Gelfand3,Gelfand2,
  Gelfand1} which are constructed by combinatorial data (a matrix of integers) together
with some ``charge vectors'' are $q$-holonomic.

The above discussion leads naturally to the following conjecture.

\begin{conjecture}
\label{conj.2}
Every $q$-holonomic module associated to a proper $q$-hypergeometric multi-dimensional
sum is modular.
\end{conjecture}

%%%%%%%%%%%%%%%%%%%%%%%%%%%%%%%%%%%%%%%%%%%%%%%%%%%%%%%%%%%%%%%%%%%%%%%%%%%% 
%%%%%%%%%%%%%%%%%%%%%%%%%%%%%%%%%%%%%%%%%%%%%%%%%%%%%%%%%%%%%%%%%%%%%%%%%%%%

\section{A review of linear $q$-difference equations}
\label{sec.prelim}

\subsection{Preliminaries}
\label{sub.prelim}

It is well-known that formal power series solutions to linear difference equations
with a small parameter are typically factorially divergent series,
which lead to analytic solutions to the difference equation after applying the
process of a Borel transformation, followed by a Laplace transformation. This
subject is classical and well-studied, see for example~\cite{Bender,
  Costin:asymptotics,Miller:applied,Sauzin:divergent}.

A corresponding theory for linear $q$-difference equations was developed recently
by di Vizio, Sauloy, Ramis and others~\cite{Dreyfus:building,Hardouin,Ramis},
with particular emphasis given on the arithmetic and the Galois theoretic aspects
of the theory. A $q$-holonomic module has two canonical filtrations, one from $t=0$
and another from $t=\infty$. These filtrations can be computed concretely choosing
a cyclic vector which converts the $q$-holonomic module into a linear
$q$-difference equation.

To explain the solutions of linear $q$-difference equations, let us recall
 the Jacobi theta function which is given by a one dimensional
lattice sum, and by an infinite product (known as the Jacobi triple product identity)
by
\be
\label{theta}
\theta(t;q)
=
\sum_{k\in\BZ}(-1)^{k}q^{k(k+1)/2}t^{k}
=
(qt;q)_{\infty}(t^{-1};q)_{\infty}(q;q)_{\infty} \,.
\ee
From the above representation, it is easy to see that it satisfies the functional
equations
\be
\label{theta.fun}
\theta(t^{-1};q)
=
%(qt^{-1};q)_{\infty}(t;q)_{\infty}(q;q)_{\infty} =
\theta(q^{-1}t;q)
=
-t\theta(t;q)
\ee
which imply that
\be
\label{theta.fun2}
\theta(q^{\ell}t;q)
=
(-1)^{\ell}q^{-\ell(\ell+1)/2}t^{-\ell}\theta(t;q),
\qquad (\ell \in \BZ) \,.
\ee
The Jacobi theta function is modular. Its transformation under the element
$S$ of $\SL_2(\BZ)$ is given by
\be
\label{thetaS}
\theta(\tilde{t};\tilde{q})
= \e(-3/8)
\e\left(\frac{z^{2}}{2\tau}\right)t^{\frac{1}{2}}
\tilde{t}^{-\frac{1}{2}}q^{\frac{1}{8}}\tq^{-\frac{1}{8}}\sqrt{\tau} \,
\theta(t;q)\,.
\ee
The derivative of $\theta$ with respect to $t d/dt=1/(2 \pi  i) d/dz$,
\be
\theta'(t;q)
=
\sum_{k}(-1)^{k}kq^{k(k+1)/2}t^{k}
\ee
satisfies the $q$-difference equation
\be
\label{thetader}
\frac{\theta'(q^{\ell}t;q)}{\theta(q^{\ell}t;q)}
=
\frac{\theta'(t;q)}{\theta(t;q)}-\ell \,.
\ee
% Noting that
% \be
%   \frac{\theta'(t;q)}{\theta(t;q)}=\frac{d}{2\pi idz}\log(\theta(t;q))
% \ee
% we see that
It has transformation under the element $S$ of $\SL_2(\BZ)$ given by
\be
\label{S.mod.thd}
\begin{aligned}
  \frac{\theta'(t;q)}{\theta(t;q)} \Big|_{1} S
  :=\frac{1}{\tau}\frac{\theta'(\tilde{t};\tq)}{\theta(\tilde{t};\tq)}
% &=
% \frac{d}{2\pi id(z/\tau)}\log(\theta(\tilde{t};\tq))\\
% &=
% \tau\frac{d}{2\pi idz}\log(\theta(t;q)
%e\left(\frac{z^{2}}{2\tau}\right)t^{\frac{1}{2}}
% \tilde{t}^{-\frac{1}{2}}q^{\frac{1}{8}}\tq^{-\frac{1}{8}}\sqrt{\tau}e(-3/8))\\
% &=
% \tau\frac{d}{2\pi idz}\log(\theta(t;q))
% +\tau\frac{d}{dz}\left(\left(\frac{z^{2}}{2\tau}\right)
%+z/2-z/2\tau+\log(q^{\frac{1}{8}}\tq^{-\frac{1}{8}}\sqrt{\tau}e(-3/8))\right)\\
&=
\frac{\theta'(t;q)}{\theta(t;q)}+\frac{z}{\tau}+\frac{1}{2}-\frac{1}{2\tau}\,.
\end{aligned}
\ee

\subsection{An algorithm for a fundamental matrix}
\label{sub.algorithm}

In this section we review an algorithm to obtain a fundamental matrix solution
to a $q$-holonomic module given by Dreyfus~\cite{Dreyfus:building} using the
$q$-Borel and the $q$-Laplace transform. Let us describe the main steps here.

\noindent $\bullet$
Choose a cyclic vector to present a $q$-holonomic module in the form of a linear
$q$-difference equation

\be
\label{qdiffeqn}
\sum_{j=0}^r a_j \s^j f =0
\ee
where $f=f(t,q)$, $(\s f)(t,q)=f(qt,q)$ and $a_j(t,q) \in \BQ(q)[t^{\pm 1}]$,
with $a_r \neq 0$. The rest of the algorithm depends on~\eqref{qdiffeqn} alone.

\noindent $\bullet$
The lower Newton polygon of~\eqref{qdiffeqn} is the lower convex hull
of the points $(j, v_0(a_j))$ for $j=0,\dots,r$, where $v_0(p)$ denotes the
minimum $t$-exponent of a Laurent polynomial $p(t)$ (with the convention that
$v_0(0)=\infty$). The boundary of the lower Newton
polygon is a finite sequence of edges with increasing slopes. 

\noindent $\bullet$
For each edge of the lower Newton polygon, we replace our function by a ratio
of theta functions times a new function, so that the corresponding edge is now
horizontal, and then apply the Frobenius method to get $t$-formal power
series solutions, once for each root of the indicial polynomial of the edge.
If $\kappa$ is the slope of the edge and an integer, the solutions have the form
\be
\theta(t;q)^{\kappa}
\sum_{k=0}^{\infty}\alpha_{k}(q)t^{k}\frac{\theta(\rind^{-1}t;q)}{\theta(t;q)}
\ee
where $\rind$ is determined by the roots of the indicial polynomial of the edge.
A special feature in the equations that we study is that the roots of the indicial
polynomials are roots of unity times a fractional power of $q$ times a monomial
in any additional variables.

\noindent $\bullet$
If a solution is not convergent at $t=0$, apply a $q$-Borel transform, followed by
an iterated $q$-Laplace transform (defined below), to construct a fundamental matrix
$U(t,q)$ of analytic solutions at $t=0$.
The $q$-Borel and $q$-Laplace transforms preserve linear $q$-difference equations
and change their Newton polygons by an affine linear transformation.

\noindent $\bullet$
Repeat the above steps using the upper Newton polygon of~\eqref{qdiffeqn}, that
is the upper convex hull of the points $(j, v_\infty(a_j))$ for $j=0,\dots,r$,
where $v_\infty(p)$ denotes the maximum $t$-exponent of a Laurent polynomial $p(t)$
(with the convention that $v_\infty(0)=-\infty$). Call the corresponding fundamental
matrix $V(t,q)$.

The ratio $M = V^{-1} U$ is a matrix of elliptic functions. These functions depend
on additional elliptic parameters that come from the $q$-Laplace transform,
a feature of $q$-difference equations which is absent in the world of linear
differential equations. The connection problem, i.e., the determination of this
matrix, is largely unsolved, with partial success for the case of many
$q$-hypergeometric difference equations; see Ohyama, Morita
~\cite{Morita:stokes-qdiff, Morita:stokes-rama,Ohyama:talk,Ohyama:connection}. 

\subsection{The $q$-Borel and the $q$-Laplace transforms}
\label{sub.qborel}

We now recall the $q$-Borel transform $\Bor_\kappa$ (for a rational number $\kappa$)
defined by
\be
\label{qborel}
\Bor_{\kappa}\left(\sum_{\ell=0}^{\infty}a_{\ell}t^{\ell}\right)(\xi,q)
=
\sum_{\ell=0}^{\infty}(-1)^{\ell}q^{\kappa \ell(\ell+1)/2}a_{\ell}\xi^{\ell}\,.
\ee
Its role is to convert divergent series, e.g., of the form
$\sum_{\ell=0}^\infty q^{-\kappa \ell(\ell+1)/2}a_{\ell}t^{\ell}$ (where $a_\ell$
is bounded and $\kappa>0$) to convergent ones at $\xi=0$.

An inverse of the $q$-Borel transform is the $q$-Laplace transform, whose role
is to construct analytic solutions to the linear $q$-difference equation
with prescribed asymptotics. It is defined for $\kappa>0$ by
\be
\label{qlap}
\Lap_{\kappa}(f)(t,\lambda,q)
=
\frac{1}{\theta(\lambda;q^\kappa)}\sum_{\ell\in\BZ}(-1)^{\ell}
q^{\kappa\ell(\ell+1)/2}\lambda^{\ell}f(q^{\kappa\ell}\lambda t,q)
=
\sum_{\ell\in\BZ}
\frac{f(q^{\kappa\ell}\lambda t,q)}{\theta(q^{\kappa\ell}\lambda;q^\kappa)}\, ,
\ee
and for $\kappa<0$, by
\be
\Lap_{\kappa}(f)(t,q)=\oint f(\xi,q)\theta(t/\xi;q^{-\kappa})
\frac{d\xi}{2\pi i\xi} \,.
\ee
The $q$-Laplace transform for positive $\kappa$ is more common, however, in
the computation of the monodromy of the $4_1$ knot, we will use $\kappa <0$,
following analogous computations of Morita \cite{Morita:conn-mon}.

In some sense, the two tranformations are inverse to each other. More precisely,
for all $\kappa$ and all natural numbers $n$, we have
\be
\label{qlapbort}
\Lap_{\kappa}\Bor_{\kappa}(t^{n})(t,\lambda,q) = t^n \,.
\ee
(Interestingly, the right hand side is independent of $\lambda$.) Indeed,
for $\kappa >0$, we have 
\be
\begin{aligned}
\Lap_{\kappa}\Bor_{\kappa}(t^{n})
&=
\frac{1}{\theta(\lambda;q^\kappa)}\sum_{\ell\in\BZ}(-1)^{\ell}q^{\kappa\ell(\ell+1)/2}
\lambda^{\ell}(-1)^{n}q^{\kappa n(n+1)/2}q^{\kappa\ell n}\lambda^{n}t^{n}\\
&=
\frac{1}{\theta(\lambda;q^\kappa)}\sum_{\ell\in\BZ}(-1)^{\ell+n}
q^{\kappa(\ell+n)(\ell+n+1)/2}\lambda^{\ell+n}t^{n}\\
&=
\frac{1}{\theta(\lambda;q^\kappa)}\sum_{\ell\in\BZ}(-1)^{\ell}
q^{\kappa\ell(\ell+1)/2}\lambda^{\ell}t^{n} = t^n \,.
\end{aligned}
\ee
A similar calculation using the residue theorem shows~\eqref{qlapbort}
for $\kappa <0$.
More generally, the fact solutions constructed by Borel transforms followed
by iterated Laplace transforms are asymptotic to the original formal power series is
known in the literature as Watson's lemma, %~\cite{Watson}
a modern proof of which may be found for instance in
Miller~\cite[p.~53, Prop.~2.1]{Miller:applied}. An analogous lemma holds for
the $q$-case, see~\cite[Prop.2.9]{Dreyfus:building}.

By its very definition, $\Lap_\kappa$ for $\kappa>0$ depends on the
variable $\lambda$ in an elliptic way
\be
\label{lapelliptic}
\Lap_\kappa(f)(t,q^{\kappa}\lambda,q)=\Lap_\kappa(f)(t,\lambda,q) 
\ee
whereas $\Lap_\kappa$ for $\kappa <0$ does not involve a variable $\lambda$. 
The next lemma, whose proof follows from an elementary application of the
residue theorem, concerns the dependence of the $q$-Laplace transform on the
auxiliary variable $\lambda$, and may be of independent interest.

\begin{lemma}
\label{lem.fund}
Assuming that
\be
\lim_{\ve\rightarrow 0}\oint_{|\xi|=\ve^{\pm}}
\frac{f(\xi t,q)\theta(\mu^{-1}\lambda^{-1}\xi;q)}{\theta(\xi\mu^{-1};q)
  \theta(\xi\lambda^{-1};q)}\frac{d\xi}{2\pi i\xi}=0
\ee
where $\ve$ avoids the poles of the integrand we have  
\be
\begin{aligned}
\Lap_{1}(f)(t,\lambda &,q)-\Lap_{1}(f)(t,\mu,q)\\
&=
\frac{\theta(\lambda ^{-1}\mu ;q)(q;q)_{\infty}^{3}}{\theta(\lambda ^{-1};q)
  \theta(\mu ;q)}\sum_{x\in\mathrm{poles\;of}\;f}\Res_{\xi=x}\frac{f(\xi,q)
  \theta(\lambda ^{-1}\mu ^{-1}t^{-1}\xi;q)}{\theta(\xi\lambda ^{-1}t^{-1};q)
  \theta(\xi\mu ^{-1}t^{-1};q)} \,.
\end{aligned}
\ee
\end{lemma}

Note that the assumption on $f$ is mild since as $t$ approaches $0$ or $\infty$
bounded away from the poles, the quotient of $\theta$'s approaches $0$. Note
also that the lemma can be extended to the case of $\mathcal{L}_{\kappa}$ for
$\kappa >0$ by substituting $q\mapsto q^{\kappa}$ except in the argument of $f$. 

\begin{proof}
We compute  
\be
\begin{aligned}
0&=
\sum_{\xi\in \mathrm{poles}}\Res_{\xi}\frac{f(\xi t,q)\theta(\mu^{-1}
  \lambda^{-1}\xi;q)}{\theta(\xi\mu^{-1};q)\theta(\xi\lambda^{-1};q)}
\frac{d\xi}{2\pi i\xi}\\
&=
\sum_{k}\Res_{\xi=q^{k}\lambda}\frac{f(\xi t,q)
  \theta(\mu^{-1}\lambda^{-1}\xi;q)}{\theta(\xi\mu^{-1};q)
  \theta(\xi\lambda^{-1};q)}\frac{d\xi}{2\pi i\xi}
+\sum_{k}\Res_{\xi=q^{k}\mu}\frac{f(\xi t,q)
  \theta(\mu^{-1}\lambda^{-1}\xi;q)}{\theta(\xi\mu^{-1};q)
  \theta(\xi\lambda^{-1};q)}\frac{d\xi}{2\pi i\xi}\\
&\qquad\qquad+\sum_{x\in \mathrm{poles\;of}\;f}\Res_{\xi=t^{-1}x}
\frac{f(\xi t,q)\theta(\mu^{-1}\lambda^{-1}\xi;q)}{\theta(\xi\mu^{-1};q)
  \theta(\xi\lambda^{-1};q)}\frac{d\xi}{2\pi i\xi}\\
% &=
% \sum_{k}\Res_{\xi=1}\frac{f(q^{k}\lambda\xi t,q)
% \theta(\mu^{-1}q^{k}\xi;q)}{\theta(q^{k}\mu^{-1}\lambda\xi;q)
% \theta(q^{k}\xi;q)}\frac{d\xi}{2\pi i\xi}
% +\sum_{k}\Res_{\xi=1}\frac{f(q^{k}\mu\xi t,q)\theta(\lambda^{-1}q^{k}\xi;q)}{
% \theta(q^{k}\mu\xi\lambda^{-1};q)\theta(q^{k}\xi;q)}\frac{d\xi}{2\pi i\xi}\\
% &\qquad\qquad+\sum_{x\in \mathrm{poles\;of}\;f}\Res_{\xi=t^{-1}x}
% \frac{f(\xi t,q)\theta(\mu^{-1}\lambda^{-1}\xi;q)}{\theta(\xi\mu^{-1};q)
%\theta(\xi\lambda^{-1};q)}\frac{d\xi}{2\pi i\xi}\\
&=
\sum_{k}(-1)^{k}q^{k(k+1)/2}\lambda^{k}
\frac{f(q^{k}\lambda t,q)
  \theta(\mu^{-1};q)}{\theta(\mu^{-1}\lambda;q)(q;q)_{\infty}^{3}}
+\sum_{k}(-1)^{k}q^{k(k+1)/2}\mu^{k}
\frac{f(q^{k}\mu t,q)\theta(\lambda^{-1};q)}{\theta(\mu\lambda^{-1};q)
  (q;q)_{\infty}^{3}}\\
&\qquad\qquad+\sum_{x\in \mathrm{poles\;of}\;f}\Res_{\xi=t^{-1}x}
\frac{f(\xi t,q)\theta(\mu^{-1}\lambda^{-1}\xi;q)}{\theta(\xi\mu^{-1};q)
  \theta(\xi\lambda^{-1};q)}\frac{d\xi}{2\pi i\xi}\\
&=
\frac{\theta(\mu^{-1};q)\theta(\lambda;q)}{\theta(\mu^{-1}\lambda;q)
  (q;q)_{\infty}^{3}}\Lap_{1}(f)(t,\lambda,q)
+\frac{\theta(\mu;q)\theta(\lambda^{-1};q)}{\theta(\mu\lambda^{-1};q)
  (q;q)_{\infty}^{3}}\Lap_{1}(f)(t,\mu,q)\\
&\qquad\qquad+\sum_{x\in \mathrm{poles\,of}\,f}\Res_{\xi=x}
\frac{f(\xi,q)\theta(\mu^{-1}\lambda^{-1}t^{-1}\xi;q)}{\theta(\xi\mu^{-1}t^{-1};q)
  \theta(\xi\lambda^{-1}t^{-1};q)}\frac{d\xi}{2\pi i\xi} \,.
\end{aligned}
\ee
\end{proof}

This type of residue formula for the Laplace transform is similar to the definition
of the Laplace transform for $\kappa<0$. We can find a similar expression for a single
Laplace transform using a special function. The Appell-Lerch sum will be studied later
in Section~\ref{sub.ex2} but we will define it here.
\be
L(t,\lambda,q)
=
\frac{1}{\theta(\lambda;q)}\sum_{k\in\BZ}(-1)^{k}
\frac{q^{k(k+1)/2}\lambda^{k}}{1-q^{k}\lambda t}.
\ee
Using this we have the following integral expression for the Laplace transform for
$\kappa>0$.

\begin{lemma}
For $\kappa>0$, we have  
\be
(\Lap_{\kappa}f)(t,\lambda,q) =
\sum_{k\in\BZ}\Res_{\xi=q^{\kappa k}}L(\xi^{-1}\lambda^{-1},\lambda,q^{\kappa})
f(\xi\lambda t,q)\frac{d\xi}{2\pi i\xi} \,.
\ee
\end{lemma}

Deforming the contour and the residue theorem give the following lemma.

\begin{lemma}
  Assuming that
  \be
\lim_{\ve\rightarrow 0}
\oint_{|\xi|=\ve^{\pm}}
L(\xi^{-1}\lambda^{-1},\lambda,q^{\kappa})
f(\xi\lambda t,q)\frac{d\xi}{2\pi i\xi}
=0
  \ee
  where $\ve$ avoids the poles of the integrand we have
  \be
(\Lap_{\kappa}f)(t,\lambda,q)
=
-\sum_{x\in\mathrm{poles\,of}\,f}
L(x^{-1}t,\lambda,q)
\Res_{\xi=x}f(\xi,q)\frac{d\xi}{2\pi i\xi}
  \ee
\end{lemma}
Notice that all the dependence on $t$ and $\lambda$ is now in the arguments of
the Appell-Lerch sums. This illustrates the important role the residues of the
Borel transform play in the resummation. An application of this lemma leads to
the following remarkable formula.

\begin{corollary}
  \be
f^{(1)}(t,\lambda,q)
=
-\theta(t;q)\sum_{k=0}^{\infty}
\left(L(q^{1/4+k/2}t,\lambda,q^{1/2})R_{+}(k,q)
  +L(-q^{1/4+k/2}t,\lambda,q^{1/2})R_{-}(k,q)\right)
  \ee
  where $f^{(1)}$ is given in Equation~\eqref{ex41.f1} and $R_{\pm}$ is given
  in Equation~\eqref{Rpm}.
\end{corollary}

\subsection{The slash operator and proof of Theorem~\ref{thm.ST}}
\label{sub.slash}

In this section we recall the slash operator, an important ingredient to
express modularity. We will use the usual conventions for the modular
$q=\e(\tau)$ and the Jacobi $t=\e(z)$, $\lambda_{i}=\e(u_{i})$, $\mu_{i}=\e(v_{i})$,
$x=\e(w)$ variables, i.e., $\ti q=\e(-1/\tau)$, $\ti t=\e(z/\tau)$ and
$\e(x)=e^{2\pi i x}$.
Now recall the slash operator $f|_\kappa \gamma$ (see, eg.~\cite[p.13]{123}
and~\cite{EZ}) for $\gamma=\sma abcd \in\SL_{2}(\BZ)$ acting on a function
$f(z,\tau)$ by 
\be
\label{fslash}
(f|_\kappa \gamma)(z,\tau) = (c \tau+d)^{-\kappa} f(\gamma(z,\tau)),
\qquad \gamma(z,\tau) = \left(\frac{z}{c\tau+d},\frac{a\tau+b}{c\tau+d}\right) \,.
%f\left(\frac{z}{c\tau+d},\frac{a\tau+b}{c\tau+d}\right) \,.
\ee
This action can be extended to a matrix-valued
function $F(z,\tau)$ of weight $\kappa=\diag(\kappa_1,\dots,\kappa_r)$ by 
\be
\label{Fslash}
  (F|_\kappa\gamma)(z,\tau)
  =
  F(\gamma(z,\tau))%\left(\frac{z}{c\tau+d},\frac{a\tau+b}{c\tau+d}\right)
 \diag((c \tau+d)^{-\kappa}) \,.
 \ee
 This extension satisfies
 \be
 \label{slashprops}
 \begin{aligned}
   (F|_\kappa\gamma)|_\kappa\gamma' &= F|_\kappa\gamma\gamma' \\
   (F G)|_\kappa\gamma &= (F|_\kappa\gamma) \Delta_\gamma
   (G|_\kappa\gamma) 
 \end{aligned}
 \ee
 for matrix-valued functions $F$ and $G$ and for $\gamma, \gamma' \in \SL_2(\BZ)$
 where $\Delta_\gamma(\tau)=\diag((c \tau+d)^{\kappa})$. We can extend these
 definitions to include half integral weight using a multiplier system as
 done for the Dedekind $\eta$-function. In all our examples the relative weights
 $\kappa_{i}-\kappa_{j}$ are integers and the absolute weights are either always
 integers, or always half-integers. Our choice of absolute weight depends on the
 normalisation of our solutions, and multiplying them by $\eta$-functions leads to
 a shift of the absolute weights by half-integers, but has no effect on the modularity
 of the linear $q$-difference equation. 

Recall the cocycle $\Om_{U,\gamma}$ from Equation~\eqref{Om}, and the corresponding
cocycle $\Om_{V,\gamma}$. If the monodromy matrix $M=V^{-1}U$ satisfies Equation
~\eqref{Mmod}, it follows that the cocycle matrices associated to $U$ and $V$
are equal:
\be
\label{WU=WV}
\begin{aligned}
  \Om_{U,\gamma} &=  (V\vM|_{\kappa_U}\gamma)(V\vM)^{-1}
  = (V|_{\kappa_{V}}\gamma) \Delta_{\kappa_V,\gamma}(\vM|_{\kappa_U}\gamma)(V\vM)^{-1}
  \\ &= (V|_{\kappa_V}\gamma)
  \Delta_{\kappa_V,\gamma}(\vM|_{\kappa_U}\gamma)\vM^{-1}V^{-1}
  = (V|_{\kappa_V}\gamma)V^{-1} = \Om_{V,\gamma} \,.
\end{aligned}
\ee

We remark that sometimes the monodromy matrix of the fundamental bases $U$
and $V$ that come from the algorithm of Section~\ref{sub.algorithm}
satisfies Equation~\eqref{Mmod} on a finite index subgroup of $\SL_2(\BZ)$
(this happens, e.g., in Theorem~\ref{thm.41}) which contains a conjugate
of a congruence subgroup of $\SL_2(\BZ)$. In this case, an averaging of
the monodromy leads to fundamental solutions whose cocycle extends for all
$\gamma \in \SL_2(\BZ)$.

In the rest of this subsection, we give a proof of Theorem~\ref{thm.ST}

\begin{proof}[Proof of Theorem~\ref{thm.ST}]
  Using the behaviour of the slash operator on the product of two
  matrices~\eqref{slashprops}, it is easy to see that if Equation~\eqref{Mmod}
  holds for two elements of $\SL_2(\BZ)$, it also holds for their product. Since
  $S$ and $T$ generate $\SL_2(\BZ)$, part (a) follows.

  Part (b) follows from the presentation of $\SL_2(\BZ)$ given by
  \be
  \label{SL2Z}
  \SL_2(\BZ) = \langle S,T \,\, | \,\, S^4=I, TSTST=S \rangle
  \ee
  and from the cocycle property, which implies that
  if $\gamma=T^{a_0} S T^{a_1} S T^{a_2} \dots S T^{a_r} S T^{a_{r+1}}$, then
  \be
  \label{OmST}
  \Om_\gamma(z,\tau) = \prod_{j=1}^r
  \Om_S( (T^{a_j} S T^{a_{j+1}} \dots T^{a_r} S T^{a_{r+1}})(z,\tau))\,.
  \ee

  The idea for part (c) is to use reduction theory, with attention paid
  to the domain of extension. Fix a cocycle $\Om$ that satisfies $\Om_T=I$ and
  $\Om_S$ extends as a meromorphic function to $\BC\times\BC'$. Below, when
  we say that $\Om_\gamma$ extends, we will mean that it extends to
  $\BC \times \BC_\gamma$. We will give the proof in several steps.

  \noindent {\bf Step 1.}
  The cocycle property implies that
  \be
  \Om_{\gamma^{-1}}(z,\tau) = \left(\Om_\gamma(\gamma^{-1}(z,\tau))\right)^{-1}\,.
  \ee
  This, together with the fact that $\gamma^{-1}(\tau)-\gamma^{-1}(\infty)=
  1/(c(-c\tau+a))$ imply that $\Om_\gamma$ extends if and only if $\Om_{\gamma^{-1}}$
  extends. In particular, applying it to $\gamma=S$, where $\gamma^{-1}=-S$, we
  obtain that
  \be
  \label{Om-S}
  \Om_{-S}(z,\tau) = (\Om_S(-z/\tau,-1/\tau))^{-1} \,.
  \ee
  Using our assumptions on $\Om_S$, it follows that $\Om_{-S}$ extends. 

  \noindent {\bf Step 2.}
  Suppose now that $\gamma=\sma abcd \in \SL_2(\BZ)$ with $a>0$ and $c>0$.
  It follows by Gauss reduction theory that we can write
  \be
  \label{gammaST}
  \gamma = T^{a_0} S T^{a_1} S \dots T^{a_r} S T^{a_{r+1}}
  \ee
  for integers $a_i$ where $a_0 \geq 0$ and $a_1, \dots, a_r >0$. Moreover, 
  $a_0, \dots, a_r$ can be obtained from the negative continued fraction
  expansion (using nearest integers from above, rather than from below)
  \be
  \label{contfrac}
  \frac{a}{c} = [a_0,a_1,\dots,a_r]:= a_0 -1/(a_1 -1/(a_2 - \dots))\,.
  \ee
  The continued fraction expansion shows that the first column of $\gamma$
  agrees with that of the product $T^{a_0} S T^{a_1} S \dots T^{a_r} S$, and
  the last integer $a_{r+1}$ is chosen so that Equation~\eqref{gammaST} holds.
  Equation~\eqref{OmST} implies that $\Om_\gamma(z,\tau)$ is a product of $r$
  matrices $\Om_S$ matrices evaluated suitably, and $\Om_\gamma$ extends to
  real $\tau$ that satisfy $(T^{a_j} S T^{a_{j+1}} \dots T^{a_r} S T^{a_{r+1}})(\tau)>0$
  for all $j=1,\dots,r$. The key property is that the first
  column of the matrices $T^{a_j} S T^{a_{j+1}} \dots T^{a_r} S T^{a_{r+1}}$
  for $j=1,\dots,r$  consists of positive integers (with the possible exception
  of $j=r+1$ where the $(2,1)$ entry may be zero). It follows that the system
  of inequalities cascades, and becomes equivalent to the single inequality
  $\tau > -d/c$. It follows that $\Om_\gamma$ extends in the case when
  $\gamma=\sma +*+* \in \SL_2(\BZ)$. 

  The above argument is best explained by an example. Consider the matrix
  $\gamma=\sma {17}{29}{7}{12}$. We expand the rational number $17/7$ of its
  first column in negative continued fractions 
  $$
  \frac{17}{7} = [3,2,4] = 3-\frac{4}{7} = 3-\frac{1}{7/4} = 3-\frac{1}{2-1/4}
  $$
  and obtain that the matrix
  $T^3 S T^2 S T^4 S=\begin{pmatrix} 17 & -5 \\ 7 & -2 \end{pmatrix}$,
%% see: mathematica/wheeler/SL2Z.Matrices.Cocycle.nb
  which further adjusting it by multiplying it on the right by $T^2$, gives
  $$
  \gamma =
  \begin{pmatrix} 17 & -5 \\ 7 & -2 \end{pmatrix} T^2 =
  T^3 S T^2 S T^4 S T^2 \,.
  $$
  The cocycle property and the fact that 
  $$
  T^2 S T^4 S T^2=\begin{pmatrix}
    7 & 12 \\ 4 & 7 \end{pmatrix}, \qquad T^4 S T^2=\begin{pmatrix}
    4 & 7 \\ 1 & 2 \end{pmatrix}, \qquad T^2= \begin{pmatrix}
    1 & 2 \\ 0 & 1 \end{pmatrix}
  $$
  implies that
  \be
  \Om_\gamma(z,\tau) =
  \Om_S\left(\frac{z}{4\tau+7}, \frac{7\tau+12}{4\tau+7}\right)
  \Om_S\left(\frac{z}{\tau+2}, \frac{4\tau+7}{\tau+2}\right)
  \Om_S(z, \tau+2)\,.
  \ee
  The right hand side extends when $\tau$ is real that satisfies
  $$
  \frac{7\tau+12}{4\tau+7} >0, \qquad \frac{4\tau+7}{\tau+2} >0, \qquad
  \tau+2 >0
  $$
  which (when reading the inequalities from last to first and simplifying)
  is equivalent to the system $\tau+2>0, 4\tau+7>0, 7\tau+12>0$, which is
  equivalent to $\tau> -12/7=\gamma^{-1}(\infty)$.

  \noindent {\bf Step 3.}
  We will now use the element $\ve=\sma {-1}001 \in \GL_2(\BZ)$ of order
  2, and observe that $\ve \sma abcd \ve = \sma a{-b}{-c}d \in \SL_2(\BZ)$
  for all $\sma abcd \in \SL_2(\BZ)$. In particular, $\ve T \ve = T^{-1}$
  and $\ve S \ve = -S$. It follows that if $\gamma=\sma abcd \in \SL_2(\BZ)$
  with $a>0$ and $c>0$ is given by~\eqref{gammaST}, then
  \be
  \label{egammaST}
  \ve\gamma\ve = \begin{pmatrix} a & -b \\ -c & d \end{pmatrix}=
    T^{-a_0} (-S) T^{-a_1} (-S) \dots T^{-a_r} (-S) T^{-a_{r+1}} \,.
  \ee
  The cocycle property implies that
  \be
  \label{OmSTe}
  \Om_{\ve\gamma\ve}(z,\tau) = \prod_{j=1}^r
  \Om_{-S}( (T^{-a_j} (-S) T^{-a_{j+1}} \dots T^{-a_r} (-S) T^{-a_{r+1}})(z,\tau))
  \ee
  where $\Om_{-S}(z,\tau)$ extends for $\tau<0$ by Step 1. Thus
  $\Om_{\ve\gamma\ve}$ extends when the inequalities 
  $(T^{-a_j} (-S) T^{-a_{j+1}} \dots T^{-a_r} (-S) T^{-a_{r+1}})(\tau)<0$
  for $j=1,\dots,r$. The key point now is that these inequalities cascade to
  a single inequality, namely, $\tau < d/c=(\ve\gamma\ve)^{-1}(\infty)$.
  It follows that $\Om_\gamma$ extends when $\gamma = \sma +*{-}* \in \SL_2(\BZ)$.
  In our running example above $\gamma=\sma {17}{29}{7}{12}$, we have
  \be
  \begin{pmatrix} 17 & -20 \\ -7 & 12 \end{pmatrix} =
  \ve \gamma \ve = T^{-3} (-S) T^{-2} (-S) T^{-4} (-S) T^{-2}
  \ee
  and the cocycle property and the fact that 
  $$
  T^{-2} (-S) T^{-4} (-S) T^{-2}=\mat 7{-12}{-4}7, \qquad
 T^{-4} (-S) T^{-2}=\mat 4{-7}{-1}2, \quad
 T^{-2}=\mat 1{-2}{0}1
 $$
 implies that
  \be
  \Om_{\ve \gamma \ve} =
  \Om_{-S}\left(\frac{z}{-4\tau+7}, \frac{7\tau-12}{-4\tau+7}\right)
  \Om_{-S}\left(\frac{z}{-\tau+2}, \frac{4\tau-7}{-\tau+2}\right)
  \Om_{-S}(z, \tau-2) \,.
  \ee
  The right hand side extends when $\tau$ is real that satisfies
  $$
  \frac{7\tau-12}{-4\tau+7} <0, \qquad \frac{4\tau-7}{-\tau+2} <0, \qquad
  \tau-2 <0
$$
which is equivalent to the system $7\tau-12<0, 4\tau-7<0, \tau-2<0$ which
is equivalent to $\tau < 12/7=(\ve\gamma\ve)^{-1}(\infty)$.

  \noindent {\bf Step 4.}
  The cocycle property and the triviality of $\Om_T=I$ implies that 
$\Om_{T \gamma}=\Om_\gamma$. Since $T \sma abcd =\sma {a+c}{b+d}cd$, it follows
that $\Om_\gamma$ (and hence its extension) depends only on the bottom row
$(c,d)$ of $\gamma$. When $\gamma_{2,1} \neq 0$, the extension follows from
either Step 2 or Step 3, depending on the sign of $\gamma_{2,1}$.

This concludes the proof of the theorem.  
\end{proof}

\subsection{Duality}
\label{sub.duality}

% notes added from qselfdual.tex.

% In this section we discuss duality properties of $q$-holonomic modules.
%Suppose $U=U(t,q)$ is a fundamental matrix solution to~\eqref{qdiff}. Then
%\be
%\s U = A U \leftrightarrow ((\s U)^{-1})^t = (A^{-1})^t ((U)^{-1})^t \leftrightarrow
%\s \check U = \check A \check U
%\ee
%where $\check M=(M^{-1})^t$. If $A=A_L$ is the companion matrix of a monic linear
%$q$-difference operator, unfortunately $\check {A_L}$ is not a companion matrix,
%but instead it needs to be gauged transformed into one.

%% see Mathematica file: wheeler/Enhanced.Companion.Matrices.pdf

In this section we review some elementary facts about duality of $q$-holonomic modules.
Recall that we can write the linear $q$-difference equation
\be
\label{af}
a_r(t,q) f(q^r t, q) + a_{r-1}(t,q) f(q^{r-1}t,q) + \dots + a_0(t,q) f(t,q) =0 
\ee
for a function $f(t,q)$ in matrix form $\s X = A X$ where
$X=X_f=(f, \s f, \dots, \s^{r-1} f)^t$ is a column vector and
$A=\comp(-a_0/a_r, \dots, -a_{r-1}/a_r)$ is the companion matrix where
{\small
\be
\label{comp}
%X_f = \begin{pmatrix} f \\ \s f \\ \vdots \\ \s^{r-1} f \end{pmatrix},
%\qquad
\comp(c_0,c_1,\dots,c_{r-1}) =
\begin{pmatrix}
0 & 1 & 0 & \dots & 0\\
0 & 0 & 1 & \dots & 0\\
: & : & : & \dots & :\\
0 & 0 & 0 & \dots & 1\\
c_0 & c_1 & c_2 & \dots
& c_{r-1} 
\end{pmatrix} \,.
\ee
}
We will also write~\eqref{af} in operator form $L f=0$ where
$L=\sum_{j=0}^r a_j \s^{r-j} \in \calW$ and denote by $M_f=\calW f$ the corresponding
module over the $q$-Weyl algebra $\calW$. The module $M_f$ has two duals.
The first one is defined by
\be
\label{Mdual1}
M_{f}^{\wedge}=M_{f^{\wedge}}, \qquad f^{\wedge}(t,q)=f(t,q^{-1}) 
\ee
which in matrix form is given by
\be
\label{Awedge}
\s X_{f^{\wedge}} = A^{\wedge} X_{f^{\wedge}}, \qquad
A^{\wedge}(t,q) = A(qt,q^{-1})^{-1} \,.
\ee
Indeed, we have
\begin{equation*}
a_r(t,q^{-1}) f^{\wedge}(q^{-r} t, q) + a_{r-1}(t,q^{-1})
f^{\wedge}(q^{-r+1}t,q) + \dots + a_0(t,q^{-1}) f^{\wedge}(t,q) =0 
\end{equation*}
and inverting $a_{r}$, we obtain that
\begin{equation*}
f^{\wedge}(qt,q)
=
-\frac{ a_{r}(qt,q^{-1})}{ a_{0}(qt,q^{-1})}f^{\wedge}(q^{1-r}t,q)
- \frac{ a_{r-1}(qt,q^{-1})}{ a_{0}(qt,q^{-1})}f^{\wedge}(t,q)
-\dots
-\frac{ a_{1}(qt,q^{-1})}{ a_{0}(qt,q^{-1})}f^{\wedge}(t,q)
\end{equation*}
which implies~\eqref{Awedge}. The second dual module is defined by
\be
\label{Mdual2}
M_{f}^{\vee}
=
\mathrm{Hom}_{\BQ(q)[t^{\pm 1}]}(M_{f},\BQ(q)[t^{\pm 1}]) 
\ee
with a basis $f^{\vee}_{i}$ for $i=0,\dots r-1$ such that for $j=0,\dots,r-1$ we have
$f^{\vee}_{i}(\sigma^{j}f)=\delta_{i,j}$. We claim that in matrix form, this
dual module is given by
\be
\label{Avee}
\s X_{f^{\vee}} = A^{\vee} X_{f^{\vee}}, \qquad
A^{\vee}(t,q) = (A(t,q)^{-1})^t \,.
\ee
Indeed, using the basis $f^{\vee}$, we can define an action of $\calW$ on
$M_{f}^{\vee}$ via conjugation with $\sigma$. By definition, the action satisfies
$\sigma(\lambda(v)) = (\sigma\cdot\lambda)(\sigma v)$ for $\lambda\in M_{f}^{\vee}$
and $v\in M_{f}$. In particular we have
$\sigma\cdot f_{i}^{\vee} = \sigma f_{i}^{\vee}\sigma^{-1}$. 
Notice that for $i,j=0,\dots,r-1$
\begin{equation*}
\begin{aligned}
(\sigma\cdot f_{i}^{\vee})(\sigma^{j}f)
&=
(\sigma f_{i}^{\vee})(\sigma^{j-1}f)
=
(\sigma f_{i}^{\vee})\left(\sum_{k=0}^{r-1}(A(q^{-1}t,q)^{-1})_{jk}\sigma^{k}f\right)
\\ &=
\sigma\left(\sum_{k=0}^{r-1}(A(q^{-1}t,q)^{-1})_{jk}\delta_{ik}\right)
%\\ &
= (A(t,q)^{-1})_{ji}
= \left(\sum_{k=0}^{r-1}(A(t,q)^{-t})_{ik}f^{\vee}_{k}\right)(\sigma^{j}f)
\end{aligned}
\end{equation*}
which implies~\eqref{Avee}.

Recall that if two modules with companion matrices $A$ and $B$ are isomorphic,
there exists a change of basis $P$ such that
\be
B
=
(\s P) A P^{-1} \,.
\ee
Then taking
\be
\begin{tiny}
\begin{aligned}
  P^{\vee}(t,q)
  =
  \begin{pmatrix}
    1 & 0 & 0 & \dots & 0\\
    a_{1}/a_0 & 1 & 0 & \dots & 0\\
    a_{2}/a_0 & \s(a_{1}/a_0) & 1 & \dots & 0\\
    : & : & : & \dots & :\\
    a_{r-1}/a_0 & \s(a_{r-2}/a_0) & \s(a_{r-3}/a_0) & \dots & 1\\
  \end{pmatrix}, \qquad
  P^{\wedge}(t,q)
  =
  \begin{pmatrix}
    0 & 0 & \dots & 0 & 1\\
    0 & 0 & \dots & 1 & 0\\
    : & : & \dots & : & :\\
    0 & 1 & \dots & 0 & 0\\
    1 & 0 & \dots & 0 & 0\\
  \end{pmatrix}
\end{aligned}
\end{tiny}
\ee
we find that
\be
\begin{aligned}
  (\s P^{\vee})^{-1}A^{\vee}P^{\vee}
  &=\comp(-a_r/a_0,\s(a_{r-1}/a_0), \dots, \s^{r}(a_{1}/a_0))\\
  (\s P^{\vee})^{-1}A^{\vee}P^{\vee}
  &=\comp\left(-\frac{a_r(qt,q^{-1})}{a_0(qt,q^{-1})},
    \frac{a_{r-1}(qt,q^{-1})}{a_0(qt,q^{-1})}, \dots,
    \frac{a_{1}(qt,q^{-1})}{a_0(qt,q^{-1})}\right)\,.
\end{aligned}
\ee

We now remark an elementary relation between fundamental solutions of
inhomogeneous linear $q$-difference equations and their corresponding
homogeneous ones. Consider the inhomogeneous equation
\be
 a_{0}(t,q)f(t,q)+\dots+ a_{r-1}(t,q)f(q^{r-1}t,q)+f(q^{r}t,q)=c_{0}(q).
\ee
We can write it either in the form $\s X_{f^{(in)}} = A^{(in)} X_{f^{(in)}}$
with 
{\tiny
  \begin{equation*}
   X_{f^{(in)}}
=
\begin{pmatrix}
1\\
f \\
:\\
\s^{r-1} f
\end{pmatrix},
\qquad
A^{(in)}
=
\begin{pmatrix}
1 & 0 & 0 & 0 & \dots & 0\\
0 & 0 & 1 & 0 & \dots & 0\\
0 & 0 & 0 & 1 & \dots & 0\\
: & : & : & : & \dots & :\\
0 & 0 & 0 & 0 & \dots & 1\\
c_{0} & - a_{0} & - a_{1} & - a_{2} & \dots & - a_{r-1}
\end{pmatrix}
\end{equation*}
}

\noindent
or in the form $\s X = A X$ with
{\small
$$
X_f=(f, \s f, \dots \s^r f)^t, \qquad 
A=\comp(- a_{0} , \s a_{0}- a_{1} , \s a_{1}- a_{2} , \dots ,
\s a_{r-2}- a_{r-1} , \s a_{r-1}-1) \,.
$$
}
%{\small
%\begin{equation*}
%X_{f}
%=
%\begin{pmatrix}
%f \\
%:\\
%\s^{r-1} f\\
%\s^r f
%\end{pmatrix}, \qquad
%A =
%\begin{pmatrix}
%0 & 1 & 0 & \dots & 0 & 0\\
%0 & 0 & 1 & \dots & 0 & 0\\
%: & : & : & \dots & : & :\\
%0 & 0 & 0 & \dots & 0 & 1\\
%- a_{0} & \s a_{0}- a_{1} & \s a_{1}- a_{2} & \dots &
%\s a_{r-2}- a_{r-1} & \s a_{r-1}-1
%\end{pmatrix} \,.
%\end{equation*} 
%}
The two equations are related by 
$X_f = P X_{f^{(in)}}, \,\,\, A=\s P A^{(in)} P^{-1}$
where
$$
P=\comp(c_{0} , - a_{0} , - a_{1} , - a_{2} , \dots , - a_{r-1}) \,.
$$
%{\small
%\begin{equation*}
%P=
%\begin{pmatrix}
%0 & 1 & 0 & 0 & \dots & 0\\
%0 & 0 & 1 & 0 & \dots & 0\\
%0 & 0 & 0 & 1 & \dots & 0\\
%: & : & : & : & \dots & :\\
%0 & 0 & 0 & 0 & \dots & 1\\
%c_{0} & - a_{0} & - a_{1} & - a_{2} & \dots & - a_{r-1}\\
%\end{pmatrix} \,.
%\end{equation*}
%}
We end this subsection by discussing the duality
\be
M^{\vee} \cong M^{\wedge}
\ee
which follows from (but is not equivalent to) the existence of matrices
$P^{\vee},P^{\wedge} \in GL_r(\BQ(t,q))$ such that
\be
\s P^{\wedge} A^{\wedge} (P^{\wedge})^{-1}
=
\s P^{\vee} A^{\vee} (P^{\vee})^{-1}
\ee
is a companion matrix. For the $q$-holonomic modules that come from Chern-Simons
theory, the duality $M \mapsto M^{\wedge}$ corresponds to orientation-reversal of
the ambient 3-manifold. On the other hand, the factorisation of the Andersen--Kashaev
state integrals into elements of $M$ and $M^{\wedge}$ suggests that in those
examples, we have $M^{\wedge}\cong M^{\vee}$. The following proposition
confirms this for the case of the $4_1$ knot. 

\begin{proposition}
\label{prop.M41duals}
The $q$-difference module $M$ associated to Equation~\eqref{41x1} satisfies
\be
  M
  \cong
  M^{\wedge}
  \cong
  M^{\vee} \,.
  \ee
The fundamental matrices satisfy
\be
\begin{aligned}
  U(t,\lambda,q)
  &=
  P^{\wedge}(t,q)
  U^{\wedge}(t,\lambda,q)
  =
  P^{\vee}(t,q)
  U^{\vee}(t,\lambda,q)
  \begin{pmatrix}
    0 & -1\\
    1 & 0
  \end{pmatrix}\,,\\
  V(t,q)
  &=
  P^{\wedge}(t,q)
  V^{\wedge}(t,q)
  =
  P^{\vee}(t,q)
  V^{\vee}(t,q)
  \begin{pmatrix}
    0 & -1\\
    1 & 0
  \end{pmatrix}
\end{aligned}
\ee
with
\be
\label{Pvw}
P^{\wedge}(t,q)
  =
  \begin{pmatrix}
1 & 0\\
2q^{-1}-t^{-1}q^{-1} & -q^{-2}
  \end{pmatrix}
  ,\quad
  P^{\vee}(t,q)
  =
  \begin{pmatrix}
0 & -q^{-1}t^{-2}\\
q^{-1}t^{-2} & 0
  \end{pmatrix} 
\ee
and cocycle $\Om$ of $M$ satisfies
  \be
  \label{quadrel}
  \Om = (P^{\wedge}|_{\kappa}\gamma) \Om^{\wedge}(P^{\wedge})^{-1}
  = (P^{\vee}|_{\kappa}\gamma) \Om^{\vee} (P^{\vee})^{-1} \,.
  \ee
\end{proposition}

Equation~\eqref{quadrel} was called ``quadratic relations'' in Section 3.3 and
equations (68)-(70) of~\cite{GZ:kashaev}. An example of a self-dual module is the
generalised $q$-hypergeometric equation (see Equation~\eqref{heine.selfdual.Ueq}
below). 

\begin{proof}
With $A$, given in Equation~\eqref{comp.mat.41.x1}, and the gauge transformations
~\eqref{Pvw}, we have
\be
A(t,q)
=
P^{\wedge}(qt,q)A^{\wedge}(t,q)P^{\wedge}(t,q)^{-1}
=
P^{\vee}(qt,q)A^{\vee}(t,q)P^{\vee}(t,q)^{-1}.
\ee
\end{proof}
The extra symmetry with $M \cong M^{\wedge}$ comes from the fact the $4_{1}$ knot
is amphichiral. However, this symmetry does not persist to the
module associated to the inhomogeneous equation. This can again be seen from
the state integrals introduced in \cite{GGMW:trivial} whose integrand lacks
the symmetry the Andersen-Kashaev state integrals have.

\begin{proposition}
\label{41inhom.not.sd}
The $q$-difference module $M^{\vee}$ associated to Equation~\eqref{41x1inhom}
is not isomorphic to $M^{\wedge}$. 
\end{proposition}

\begin{proof}
The companion matrices of $M$ and $M^{\wedge}$ are given by
\be
  A^{\wedge}(t,q)
  =
  \begin{pmatrix}
1 & 0 & 0\\
q^{-1}t^{-1} & 2q^{-1}-t^{-1}q^{-1} & -q^{-2}\\
0 & 1 & 0
  \end{pmatrix},
  \quad
  A^{\vee}(t,q)
  =
  \begin{pmatrix}
1 & t^{-1} & 0\\
0 & 2q-t^{-1} & 1\\
0 & -q^{2} & 0
  \end{pmatrix} \,.
\ee
If there was an isomorphism there would exist $P(t,q) \in \GL_3(\BQ(t,q))$
such that
\be
  P(qt,q)A^{\vee}(t,q)=A^{\wedge}(t,q)P(t,q) \,.
\ee
It follows that   
\be
\label{P1111}
  P_{1,1}(qt,q)=P_{1,1}(t,q), \qquad
  P_{1,2}(qt,q)=P_{1,3}(t,q),
\ee
which then implies
\be
  tP_{1,2}(t,q)+(1-2tq)P_{1,2}(qt,q)+q^2tP_{1,2}(q^2t,q)=P_{1,1}(qt,q).
\ee
Since $P_{1,1}\in\BQ(t,q)$ satisfies~\eqref{P1111}, it is independent of $t$,
i.e., $P_{1,1}(t,q)=P_{1,1}(q)$. Therefore, $P_{1,2}$ would be a $P_{1,1}(q)$
multiple of a $\BQ(t,q)$-valued solution to Equation~\eqref{41x1inhom}. The
only such solution is zero, thus $P_{1,1}=P_{1,2}=0$ which, together with~\eqref{P1111}
gives also $P_{1,3}=0$, which violates the fact that $P$ is invertible.
\end{proof}

\subsection{Categorical aspects}
\label{sub.cat}

In this section we briefly recall some categorical aspects of modules over
the $q$-Weyl algebra and give a proof of Lemma~\ref{lem.cat}.

To begin with, a gauge transformation $X=P^{-1}Y$ changes~\eqref{qdiff} to
$\s Y = B Y$ where $B=\s P A P^{-1}$, changes a fundamental solution $U$ of
~\eqref{qdiff} to $P^{-1}U$, and consequently changes $\Om_\gamma$ to
$(P|_{\kappa}\gamma)^{-1} \Om P$. Hence, if $P \in \GL_r(\BQ(t,q))$, then modularity
is a property of the gauge equivalence class of a linear $q$-difference equation,
i.e., a property of a $q$-holonomic module, independent of a choice of a cyclic
vector. This concludes part (a) of Lemma~\ref{lem.cat}.

For part (b), we use the convariant function
$M \mapsto \Sol(M):=\Ker(\s, \calF \otimes_{\BQ(t,q)} M)$ where $\calF$ denotes a
universal $q$-difference field; see for
example~\cite[Sec.2.2]{Singer:galois-differential} for the case of modules over the
Weyl algebra and~\cite{Hardouin} for its extension for the $q$-Weyl algebra $\calW$.
This functor by definition satisfies~\cite[Lem.2.16]{Singer:galois-differential}
\be
\Sol(M) \cong \{y \in \calF^r \,\, | \,\, \s y=A y\}
\ee
where $A$ is the matrix obtained by a choice of a cyclic vector of $M$. Moreover,
if
$$
0 \to M' \to M \to M'' \to 0
$$
is a short exact sequence of finitely generated $\calW$-modules, then
\be
\label{solexact}
0 \to \Sol(M') \to \Sol(M) \to \Sol(M'') \to 0
\ee
is a short exact sequence of vector spaces over $\BQ(t,q)$;
see~\cite[Sec.2.2]{Singer:galois-differential}. In addition, $M$ has a canonical
filtration at $t=0$ (and also at $t=\infty$) independent of the choice of cyclic
vector~\cite{Sauloy:filtration} compatible with submodules and quotient modules.
Converting the above into matrices, it follows that the filtration preserving
fundamental matrices $U$, $U'$ and $U''$ of $M$, $M'$ and $M''$ (and likewise,
$V$, $V'$ and $V''$) and their corresponding cocycles are related by
\be
U = \begin{pmatrix} U' & * \\ 0 & U'' \end{pmatrix}, \qquad
\Om_U = \begin{pmatrix} \Om_{U'} & * \\ 0 & \Om_{U''} \end{pmatrix} \,.
\ee
It follows that if $\Om_U$ extends to the cut plane, so does $\Om_{U'}$
and $\Om_{U''}$, concluding part (b). It is unlikely that the converse to
part (b) holds, namely if extensions of modular $q$-holonomic modules are
$q$-holonomic, but not necessarily modular.

For part (c), observe that if $U$ is a fundamental matrix for $M$, then
$U^\wedge(t,q):=U(t,q^{-1})$ and $U^\vee=(U^{-1})^t$ are fundamental matrices for
$M^\wedge$ and $M^\vee$. It follows that if $\Om$ is a cocycle of $M$ then
$\Om^\wedge(z,\tau):=\Om(z,-\tau)$ and $\Om^\vee=(\Om^{-1})^t$ are cocycles for
$M^\wedge$ and $M^\vee$. Part (c) follows. This concludes the proof of
Lemma~\ref{lem.cat}.  
\qed

%%%%%%%%%%%%%%%%%%%%%%%%%%%%%%%%%%%%%%%%%%%%%%%%%%%%%%%%%%%%%%%%%%%%%%%%%%%%
%%%%%%%%%%%%%%%%%%%%%%%%%%%%%%%%%%%%%%%%%%%%%%%%%%%%%%%%%%%%%%%%%%%%%%%%%%%% 

\section{Heine's $q$-hypergeometric functions}
\label{sec.heine}

This section is devoted to the proof of Theorem~\ref{thm.heine}.

\subsection{Solutions}

In this section we describe the solutions of the generalised $q$-hypergeometric
equation~\eqref{gener.qhyper}. Since that equation depends on parameters, it will
be convenient to consider the following system of equations
\be
\label{heine.equs}
\begin{aligned}
  \left(\prod_{j=0}^{r-1}(1-q^{-1}b_{j}\sigma_{t})-t\prod_{j=1}^{r}(1-a_{j}\sigma_{t})
  \right)f
  &=0\\
  \left(\sigma_{a_{i}}^{-1}-q^{-1}a_{i}\sigma_{t}\sigma_{a_{i}}^{-1}-(1-q^{-1}a_{i})
  \right)f
  &=0\\
  \left(1-q^{-1}b_{i}\sigma_{t}-(1-q^{-1}b_{i})\sigma_{b_{i}}^{-1}\right)f
  &=0\,.
\end{aligned}
\ee
The first equation describes the $q$-difference equation in $t$, namely Equation
~\eqref{gener.qhyper} whose Newton polygon shown in Figure~\ref{f.heine}.

\begin{figure}[!htpb]
\begin{center}
\begin{tikzpicture}[scale=0.8,baseline=-3]
\draw[<-,thick] (-1.5,0) -- (3.5,0);
\draw[->,thick] (5.5,0) -- (10.5,0);
\draw[<->,thick] (0,-1.5) -- (0,3.5);
\filldraw (0,0) circle (2pt);
\filldraw (0,2) circle (2pt);
\filldraw (2,0) circle (2pt);
\filldraw (2,2) circle (2pt);
\filldraw (7,0) circle (2pt);
\filldraw (7,2) circle (2pt);
\filldraw (9,0) circle (2pt);
\filldraw (9,2) circle (2pt);
\filldraw (4,1) circle (1pt);
\filldraw (4.5,1) circle (1pt);
\filldraw (5,1) circle (1pt);
\node at (1,2.4) {$g^{(a_{1}^{-1})}$};
\node at (8,2.4) {$g^{(a_{r}^{-1})}$};
\node at (1,-0.4) {$f^{(1)}$};
\node at (8,-0.4) {$f^{(qb_{r-1}^{-1})}$};
\draw (0,2) -- (3.5,2);
\draw (5.5,2) -- (9,2) -- (9,0);
\fill[blue,opacity=0.2] (0,0) -- (0,2) -- (3.5,2) -- (3.5,0) -- cycle;
\fill[blue,opacity=0.1] (3.5,0) -- (5.5,0) -- (5.5,2) -- (3.5,2) -- cycle;
\fill[blue,opacity=0.2] (5.5,0) -- (5.5,2) -- (9,2) -- (9,0) -- cycle;
\end{tikzpicture}
\caption{The Newton polygon of the first Equation~\eqref{heine.equs}.}
\label{f.heine}
\end{center}
\end{figure}
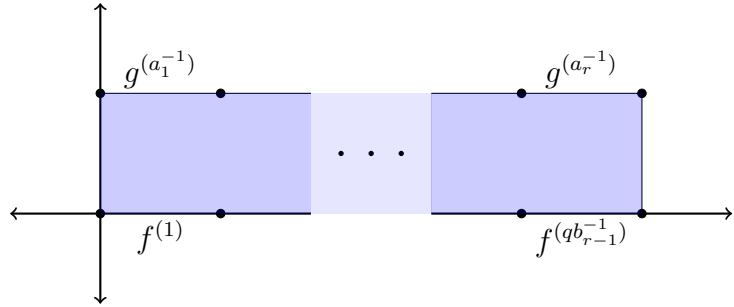

We see that there are no slopes of the Newton polygon and therefore all solutions
are determined by the top and bottom edges and their indicial polynomials. We will
normalise the solutions coming from the Frobenius algorithm so that they satisfy
the full system of Equations~\eqref{heine.equs}. The bottom edge of the Newton
polygon in Figure~\ref{f.heine} has indicial polynomial
\be
  (1-\rho)(1-q^{-1}b_{1}\rho)\dots(1-q^{-1}b_{r-1}\rho)=0\,.
\ee
The solutions corresponding to the roots are given by $f^{(q^{-1}b_{j})}$ in
Equation~\eqref{heine.bjsols}, with the convention that $b_0=q$. The top
edge of the Newton polygon in Figure~\ref{f.heine} has indicial polynomial
\be
  (1-a_{1}\rho)\dots(1-a_{r}\rho)=0\,.
\ee
The solutions corresponding to the roots are given by $g^{(a_{j}^{-1})}$ in
Equation~\eqref{heine.ajsols}. The companion matrix of Equation~\eqref{gener.qhyper}
is given by
\be
  A(t,a,b,q)
  =
  \begin{pmatrix}
    0 & 1  & \dots & 0\\
    : & : & \dots & :\\
    0 & 0 & \dots & 1\\
    -(-1)^{r}\frac{1-t}{e_{r}(b/q)-te_{r}(a)}
    & (-1)^{r}\frac{e_{1}(b/q)-te_{1}(a)}{e_{r}(b/q)-te_{r}(a)} & \dots
    & \frac{e_{r-1}(b/q)-te_{r-1}(a)}{e_{r}(b/q)-te_{r}(a)}
  \end{pmatrix}
\ee
where
\be
  e_{k}(x)
  =
  \sum_{1\leq j_{1}<\dots<j_{k}\leq r}x_{j_{1}}\dots x_{j_{k}}
  \ee are the elementary symmetric polynomials and so with $U$ and $V$ in
  Equation~\eqref{UVheine},
\be
  U(qt,a,b,q)
  =
  A(t,a,b,q)U(t,a,b,q)
  \quad\text{and}\quad
  V(qt,a,b,q)
  =
  A(t,a,b,q)V(t,a,b,q)\,.
\ee

\subsection{Self duality}

We will introduce a state integral in Section~\ref{subsec.heine.stateint} which
factorises as a finite sum of products of solutions of the module $M$ associated
to Equations~\eqref{heine.equs} and its dual $M^{\wedge}$ (see Section~\ref{sub.duality}
for the definitions of the two duals). To prove modularity we must factorise the state
integral as a finite sum of products of solutions of $M$ and $M^{\vee}$. To do
so, we need to give an explicit isomorphism between $M^{\vee}$ and $M^{\wedge}$.
This is the content of the following proposition, which after some change of
variables, is equivalent to Beukers--Jouhet~\cite[Thm.1.3]{Beukers:duality}.
For completeness, we will give an independent proof using the methods of our paper.

\begin{proposition}
  \label{prop.heine.selfdual}\cite[Thm.1.3]{Beukers:duality}
If $M$ is the module associated to Equations~\eqref{heine.equs} then
\be
\label{heine.selfdual.MV=MA}
  M^{\vee}\cong M^{\wedge} \,.
\ee
Explicitly, there exists $Q(t,a,b,q)\in GL_{r}(\BQ(t,a,b,q))$ such that
\be\label{heine.selfdual.Peq}
  Q(qt,a,b,q)A(qt,a,b,q^{-1})^{-1}
  =
  A(t,a,b,q)^{-T}Q(t,a,b,q)\,.
\ee
In addition, $Q$ satisfies 
\be
\label{heine.selfdual.Ueq}
  U(t,a,b,q)^{-T}
  =
  Q(t,a,b,q)
  U(t,a,b,q^{-1})
  \mathrm{diag}(1,-1,-1,\dots,-1)\,.
\ee
\end{proposition}

\begin{proof}
To prove Equation~\eqref{heine.selfdual.Peq} let
\be
\begin{tiny}
\begin{aligned}
  &(1-q^{-1}b_{1})\dots(1-q^{-1}b_{r-1})P(t,a,b,q)\\
  &=\begin{pmatrix}
    1-t & 0 & 0 & \dots & 0\\
    0 & -(e_{2}(b/q)-q^{-1}te_{2}(a)) & e_{3}(b/q)-q^{-2}te_{3}(a) & \dots
    & (-1)^{r-1}(e_{r}(b/q)-q^{-r+1}te_{r}(a))\\
    0 & (e_{3}(b/q)-q^{-1}te_{3}(a)) & -(e_{4}(b/q)-q^{-2}te_{4}(a)) & \dots & 0\\
    : & : & : & \dots & :\\
    0 & (-1)^{r-1}(e_{r}(b/q)-q^{-1}te_{r}(a)) & 0 & \dots & 0
  \end{pmatrix}\,.
\end{aligned}
\end{tiny}
\ee
Using the fact that $e_{k}(\lambda x)=\lambda^{k}e_{k}(x)$, one can then see that
\be
\begin{tiny}
\begin{aligned}
  &(1-q^{-1}b_{1})\dots(1-q^{-1}b_{r-1})P(qt,a,b,q)A(qt,q^{-1}a,q^{-2}b,q^{-1})^{-1}\\
  &=(1-q^{-1}b_{1})\dots(1-q^{-1}b_{r-1})A(t,a,b,q)^{-T}P(t,a,b,q)\\
  &=
  \begin{pmatrix}
    e_{1}(b/q)-te_{1}(a) & -(e_{2}(b/q)-q^{-1}te_{2}(a)) & \dots
    & (-1)^{r-1}(e_{r}(b/q)-q^{-r+1}te_{r}(a))\\
    -(e_{2}(b/q)-te_{2}(a)) & e_{3}(b/q)-q^{-1}te_{3}(a) & \dots & 0\\
    : & : & \dots & : \\
    (-1)^{r-1}(e_{r}(b/q)-te_{r}(a)) & 0 & \dots & 0\\
  \end{pmatrix}\,.
\end{aligned}
\end{tiny}
\ee
Now notice that the second and third Equations~\eqref{heine.equs} give gauge
equivalences between the modules in $t$ thinking of $a,b$ as constants when we
shift $a_{j}\mapsto qa_{j}$ and $b_{i}\mapsto qb_{i}$. Therefore, multiplying $P$
by these gauge equivalences gives the desired $Q$. Now to prove
Equation~\eqref{heine.selfdual.Ueq} we note that, from Equation
~\eqref{heine.selfdual.Peq}
\be
  E(t,a,b,q)
  :=
  U(t,q^{-1}a,q^{-2}b,q^{-1})^{-1}P(t,a,b,q)^{-1}U(t,a,b,q)^{-T}
\ee
is an elliptic function. However,
\be
  H(t,a,b,q)
  :=
  \mathrm{diag}\left(\frac{\theta(q^{-2}b_{i}t;q)}{\theta(q^{-1}t;q)}\right)^{-1}
  E(t,a,b,q)\,\,\,
  \mathrm{diag}\left(\frac{\theta(q^{-1}b_{i}t;q)}{\theta(t;q)}\right)
\ee
is holomorphic at $t=0$. Therefore, we see that
\be
H_{i,j}(t,a,b,q)=
-t\frac{\theta(q^{-1}b_{j-1}t;q)}{\theta(q^{-2}tb_{i-1};q)}E_{i,j}(t,a,b,q)
\ee
is holomorphic. This implies that if $i\neq j$ that $E_{i,j}(t,a,b,q)$ has at most
has simple poles at $q\in b_{j-1}^{-1}q^{\BZ}$ and must have zeros at
$t\in b_{i-1}^{-1}q^{\BZ}$ and there is no such non-zero elliptic function and
therefore $E_{i,j}(t,a,b,q)=0$. Then notice that
\be
  H_{ii}(t,a,b,q)
  =
  -t\frac{\theta(q^{-1}b_{i-1}t;q)}{\theta(q^{-2}tb_{i-1};q)}E_{i,i}(t,a,b,q)
  =
  qb_{i-1}^{-1}E_{i,i}(t,a,b,q)
\ee
is holomorphic at $t=0$ and so as $E_{i,i}(t,a,b,q)$ is also elliptic in $t$ it
is constant in $t$. Now notice that
\be
\begin{tiny}
\begin{aligned}
  &\lim_{t\rightarrow 0}U(t,q^{-1}a,q^{-2}b,q^{-1})
  \mathrm{diag}\left(\frac{\theta(q^{-2}b_{i}t;q)}{\theta(q^{-1}t;q)}\right)\\
  &=
  \begin{pmatrix}
    1 & 1 & \dots & 1\\
    1 & qb_{1}^{-1} & \dots & qb_{r-1}^{-1}\\
    : & : & \dots & :\\
    1 & (qb_{1}^{-1})^{r-1} & \dots & (qb_{r-1}^{-1})^{r-1}\\
  \end{pmatrix}
  \begin{pmatrix}
    1 & 0 & \dots & 0\\
    0 & B_{1}(q^{-1}a,q^{-2}b,q^{-1}) & \dots & 0\\
    : & : & \dots & :\\
    0 & 0 & \dots & B_{r-1}(q^{-1}a,q^{-2}b,q^{-1})\\
  \end{pmatrix}
\end{aligned}
\end{tiny}
\ee
and
\be
\begin{small}
\begin{aligned}
  &\lim_{t\rightarrow 0}\mathrm{diag}
  \left(\frac{\theta(q^{-1}b_{i}t;q)}{\theta(t;q)}\right)^{-1}U(t,a,b,q)^{T}\\
  &=
  \begin{pmatrix}
    1 & 0 & \dots & 0\\
    0 & B_{1}(a,b,q) & \dots & 0\\
    : & : & \dots & :\\
    0 & 0 & \dots & B_{r-1}(a,b,q)\\
  \end{pmatrix}
  \begin{pmatrix}
    1 & 1 & \dots & 1\\
    1 & qb_{1}^{-1} & \dots & (qb_{1}^{-1})^{r-1}\\
    : & : & \dots & :\\
    1 & qb_{r-1}^{-1} & \dots & (qb_{r-1}^{-1})^{r-1}\\
  \end{pmatrix}\,.
\end{aligned}
\end{small}
\ee
Then using the fact that
\be
  \sum_{j=0}^{r}(-1)^{j}e_{j}(b/q)\rho^{j}
  =
  (1-\rho)(1-q^{-1}b_{1}\rho)\dots(1-q^{-1}b_{r-1}\rho)
\ee
vanishes at $\rho\in\{qb_{i}^{-1}\}$ we can show the matrix
\be
\begin{aligned}
  \prod_{k=1}^{r}(1-q^{-1}b_{k})
  \begin{pmatrix}
    1 & 1 & \dots & 1\\
    1 & qb_{1}^{-1} & \dots & (qb_{1}^{-1})^{r-1}\\
    : & : & \dots & :\\
    1 & qb_{r-1}^{-1} & \dots & (qb_{r-1}^{-1})^{r-1}\\
  \end{pmatrix}
  P(t,a,b,q)
  \begin{pmatrix}
    1 & 1 & \dots & 1\\
    1 & qb_{1}^{-1} & \dots & qb_{r-1}^{-1}\\
    : & : & \dots & :\\
    1 & (qb_{1}^{-1})^{r-1} & \dots & (qb_{r-1}^{-1})^{r-1}\\
  \end{pmatrix}
\end{aligned}
\ee
has $(i+1,j+1)$-th entry
\be
\begin{aligned}
  &\sum_{k=0}^{r-1}(-1)^{j}e_{j}(b/q)(qb_{j}^{-1})^{k}
  \sum_{\ell=0}^{r-1-k}(b_{i}b_{j}^{-1})^{\ell}\\
  &=
  \delta_{i,j}(1-b_{0}/b_{j})\dots(1-b_{j-1}/b_{j})(1-b_{j+1}/b_{j})
  \dots(1-b_{r-1}/b_{j})\
\end{aligned}
\ee
where we have used the fact that
\be
\begin{small}
\begin{aligned}
    &\sum_{k=0}^{r-1}(-1)^{j}e_{j}(b/q)x^{k}\sum_{\ell=0}^{r-1-k}(xy^{-1})^{\ell}
    =
    \sum_{k=0}^{r-1}(-1)^{j}e_{j}(b/q)x^{k}\frac{1-(xy^{-1})^{r-k}}{1-xy^{-1}}\\
    &=
    \frac{1}{1-xy^{-1}}\left((1-x)(1-q^{-1}b_{1}x)
 \dots(1-q^{-1}b_{2}x)-x^{r}y^{-r}(1-y)(1-q^{-1}b_{1}y)\dots(1-q^{-1}b_{2}y)\right)\,.
\end{aligned}
\end{small}
\ee
Therefore, with the convention that $(x;q^{-1})_{\infty}=(qx;q)_{\infty}^{-1}$ and
$(q^{-1};q^{-1})_{\infty}=(q;q)_{\infty}^{-1}$ when $|q|<1$, for $j>0$
\be
\begin{aligned}
  &B_{j}(a,b,q)B_{j}(q^{-1}a,q^{-2}b,q^{-1})\\
  &=-q^{-1}b_{j}\frac{(1-q^{-1}b_{1})\dots(1-q^{-1}b_{r-1})}{(1-q/b_{j})(1-b_{1}/b_{j})
    \dots(1-b_{j-1}/b_{j})(1-b_{j+1}/b_{j})\dots(1-b_{r-1}/b_{j})}
\end{aligned}
\ee
and we see that
\be
\begin{aligned}
  &\begin{pmatrix}
    E_{1,1}(t,a,b,q) & 0 & 0 & \dots & 0\\
    0 & qb_{1}^{-1}E_{2,2}(t,a,b,q) & 0 & \dots & 0\\
    : & : & : & \dots & :\\
    0 & 0 & 0 & \dots & qb_{r-1}^{-1}E_{r,r}(t,a,b,q)
  \end{pmatrix}\\
  &=
  \lim_{t\rightarrow 0}H(t,a,b,q)
  =
  \begin{pmatrix}
    1 & 0 & 0 & \dots & 0\\
    0 & -qb_{1}^{-1} & 0 & \dots & 0\\
    : & : & : & \dots & :\\
    0 & 0 & 0 & \dots & -qb_{r-1}^{-1}
  \end{pmatrix}
\end{aligned}
\ee
Therefore, again using the gauge equivalence in the second and third Equations
~\eqref{heine.equs} complete the proof.
\end{proof}

\begin{remark}
  \label{rem.r=2}
When $r=2$, Proposition~\ref{prop.heine.selfdual} is equivalent to the identity
~\cite[Eqn.(1.4.3)]{Gasper}
\be
  {}_{2}\phi_{1}(a,b;c;q,t)
  =
  \frac{(abc^{-1}t;q)_{\infty}}{(t;q)_{\infty}}
  {}_{2}\phi_{1}(ca^{-1},cb^{-1};c^{-1};q,abc^{-1}t)
\ee
where we note that the ratio of Pochhammers is related to the determinant of $U$.
This is the $q$-analogue of Euler's transformation formula for ${}_2F_{1}$
~\cite[Eqn.(1.4.2)]{Gasper}
\be
  {}_{2}F_{1}(a,b;c;t)=(1-t)^{c-a-b}{}_{2}F_{1}(c-a,c-b;c;t)
\ee
where
\be
{}_{r}F_{r-1}(a;b;t) =
\sum_{k=0}^{\infty}
\frac{(a_{1})_{k}\dots(a_{r})_{k}}{(b_{1})_{k}\dots(b_{r-1})_{k}}\frac{t^{k}}{k!} \,.
\ee
\end{remark}

\begin{remark}
\label{rem.q=1}
The $q\rightarrow 1$ limit of Proposition~\ref{prop.heine.selfdual} is discussed
in full generality in \cite[Thm.1.1]{Beukers:duality}, where it is shown that the
dual of the module associated to ${}_{r}F_{r-1}(a;b;t)$ is the module associated to
${}_{r}F_{r-1}(1-a;2-b;t)$.
\end{remark}

\subsection{State integral}
\label{subsec.heine.stateint}

Consider the following state integral

\be
\label{heine.stateint}
I(z,\alpha,\beta,\tau) = 
\int_{\calC} \frac{\Phi_{\sfb}(x)
  \prod_{j=1}^{r-1} \Phi_{\sfb}(x+i\sfb^{-1}\beta_{j}+i\sfb^{-1}-i\sfb)}{
\prod_{j=1}^{r} \Phi_{\sfb}(x+i\sfb^{-1}\alpha_{j}+i\sfb^{-1}-i\sfb)}
%\frac{\Phi_{\sfb}(x)\Phi_{\sfb}(x+i\sfb^{-1}\beta_{1}-i\sfb^{-1}-i\sfb)\dots
%  \Phi_{\sfb}(x+i\sfb^{-1}\beta_{r-1}-i\sfb^{-1}-i\sfb)}
%{\Phi_{\sfb}(x+i\sfb^{-1}\alpha_{1}-i\sfb^{-1}-i\sfb)\dots
%  \Phi_{\sfb}(x+i\sfb^{-1}\alpha_{r}-i\sfb^{-1}-i\sfb)}
\exp\left(-2\pi\frac{zx}{\sfb}\right)
dx \,.
%\int_{\calC}
%\frac{\Phi_{\sfb}(x)\Phi_{\sfb}(x+i\sfb^{-1}\beta_{1}-i\sfb^{-1}-i\sfb)\dots
%  \Phi_{\sfb}(x+i\sfb^{-1}\beta_{r-1}-i\sfb^{-1}-i\sfb)}
%{\Phi_{\sfb}(x+i\sfb^{-1}\alpha_{1}-i\sfb^{-1}-i\sfb)\dots
%  \Phi_{\sfb}(x+i\sfb^{-1}\alpha_{r}-i\sfb^{-1}-i\sfb)}
%\exp\left(-2\pi\frac{zx}{\sfb}\right)
%dx \,.
\ee
where $\calC$ is a contour in the complex plane asymptotic to
$\BR+i\ve$ that separates the poles of the numerator from the zeros of the
denominator of the integrand, $a_{i}=\e(\alpha_{i})$,
$\tilde{a}_{i}=\e(\alpha_{i}/\tau)$, $b_{i}=\e(\beta_{i})$ and
$\tilde{b}_{i}=\e(\beta_{i}/\tau)$, and $\Phi_\sfb$ is the Faddeev quantum
dilogarithm function ~\cite{Faddeev, FK-QDL} and $\sfb=\sqrt{\tau}$.

% \CW{Need to potentially deal with a wiggly contour on a compact}
% \be
% \begin{tiny}
% \begin{aligned}
% &\calI(z,\alpha,\beta,\tau)=\\
% &\int_{-i\sqrt{\tau}(\BR+i\ve)}
% \frac{(-q^{1/2}\e(x+\beta_{0});q)_{\infty}
% \dots(-q^{1/2}\e(x+\beta_{r-1});q)_{\infty}(-\tq^{1/2}
% \e(\frac{x+\alpha_{1}}{\tau});\tq)_{\infty}\dots(-\tq^{1/2}
% \e(\frac{x+\alpha_{r}}{\tau});\tq)_{\infty}}
% {(-\tq^{1/2}\e(\frac{x+\beta_{0}}{\tau});\tq)_{\infty}\dots
% (-\tq^{1/2}\e(\frac{x+\beta_{r-1}}{\tau});\tq)_{\infty}
% (-q^{1/2}\e(x+\alpha_{1});q)_{\infty}(-q^{1/2}\e(x+\alpha_{r});q)_{\infty}}
% \e\left(\frac{zx}{\tau}\right)
% dx \,.
% \end{aligned}
% \end{tiny}
% \ee
We first discuss convergence of the above integral for
$\tau=\sfb^2$ in the upper half-plane. Using the asymptotic behavior
$\Phi_{\sfb}(x) \sim e^{2 \pi i x^2}$ (resp., $1$) when $\Re(x) \gg 0$
(resp., $\Re(x) \ll 0$) (see for example, ~\cite[Eqn.(46)]{AK}), it follows that
when $\Re(x)\gg 0$, the integrand of~\eqref{heine.stateint} is given by a constant
times $e^{-2\pi x \sfb^{-1} (\gamma+z)}$ where
$\gamma=\sum_{j=0}^{r-1}\beta_j-\sum_{j=1}^r \alpha_j$, and setting $x=x_0+i t$
with $t \gg 0$, it follows that the absolute value of the integrand is a constant
times $e^{-2\pi t \Im(\sfb^{-1}(\gamma+z))}$, which is exponentially decaying when
$\Im( \sfb^{-1}(\gamma+z))<0$. Likewise, when $\Re(x) \ll 0$, the integrand is
exponentially decaying when $\Im(\sfb^{-1} z)>0$. 

Finally, the state integral satisfies difference equations when we shift $\alpha,\beta,z$ by either $1$ or $\tau$. This can be used to analytically extend to a meromorphic function for $\tau\in\BC'$ and $\alpha\in\BC^{r}$, $\beta\in\BC^{r-1}$, and $z\in\BC$.

% \CW{I think we can suppress this next calculation}
% \be
% \begin{aligned}
%   &\int_{\BR+i\epsilon}^{*}\frac{\Phi_{\textsf{b}}(x)
%   \Phi_{\textsf{b}}(x+i\mathsf{b}^{-1}\beta_{1})\dots
%   \Phi_{\textsf{b}}(x+i\mathsf{b}^{-1}\beta_{r-1})}{
%   \Phi_{\textsf{b}}(x+i\mathsf{b}^{-1}\alpha_{1})\dots
% \Phi_{\textsf{b}}(x+i\mathsf{b}^{-1}\alpha_{r})}
%   \exp(2\pi\mathsf{b}^{-1}xz)dx\\
%   &\propto
%   \sum_{m,n\in\BZ_{\geq0}}
%   \left(\sum_{i=0}^{r-1}\Res_{w=-\beta_{i}}\right)
%   \Bigg[\frac{(q\e(w);q)_{\infty}(q\e(w)b_{1};q)_{\infty}\dots
% (q\e(w)b_{r-1};q)_{\infty}}{(q\e(w)a_{1};q)_{\infty}\dots(q\e(w)a_{r};q)_{\infty}}\\
% &\quad\times\frac{(\tq\e(w/\tau)\tilde{a}_{1};\tq)_{\infty}\dots
% (\tq\e(w/\tau)\tilde{a}_{r};\tq)_{\infty}}{(\tq\e(w/\tau);\tq)_{\infty}
% (\tq\e(w/\tau)\tilde{b}_{1};\tq)_{\infty}\dots
% (\tq\e(w/\tau)\tilde{b}_{r-1};\tq)_{\infty}}\\
% &\quad\times\frac{(\e(-w/2\tau)-\e(w/2\tau)\tilde{a}_{1})\dots
% (\e(-w/2\tau)-\e(w/2\tau)\tilde{a}_{r})}{(\e(-w/2\tau)-\e(w/2\tau))
% (\e(-w/2\tau)\tilde{b}_{1}-\e(w/2\tau)\tilde{b}_{1})\dots(\e(-w/2\tau)-\e(w/2\tau)
% \tilde{b}_{r-1})}\e(zw/\tau)\\
% &\quad\times\frac{(q\e(w)a_{1};q)_{m}\dots(q\e(w)a_{r};q)_{m}}{
% (q\e(w);q)_{m}(q\e(w)b_{1};q)_{m}\dots(q\e(w)b_{r-1};q)_{m}}t^{m}\\
% &\quad\times\frac{(\tq\e(-w/\tau)\tilde{a}_{1}^{-1};\tq)_{n}\dots
% (\tq\e(-w/\tau)\tilde{a}_{r}^{-1};\tq)_{n}}{(\tq\e(-w/\tau);\tq)_{n}
% (\tq\e(-w/\tau)\tilde{b}_{1}^{-1};\tq)_{n}\dots
% (\tq\e(-w/\tau)\tilde{b}_{r-1}^{-1};\tq)_{n}}(\tilde{a}_{1}\dots
% \tilde{a}_{r}\tilde{b_{1}}^{-1}\dots\tilde{b_{r-1}}^{-1}\tilde{t})^{n}\Bigg]\\
% \end{aligned}
% \ee
From its very definition, the state integral is a well-defined holomorphic
function of $\tau \in \BC'$. Moreover, after moving the contour of integration
upwards and using the residue theorem (see eg.~\cite{GK:qseries}), the
state integral in Equation~\eqref{heine.stateint} can be written in the
factorised form
\be
\label{heine.II}
I(z,\alpha,\beta,\tau)=-t^{1/2}\tilde{t}^{1/2}\frac{\tau}{2\pi i}
\calI(z,\alpha,\beta,\tau)
\ee
where
\be
\label{heine.stateint.fac}
\calI(z,\alpha,\beta,\tau)
=f^{(1)}(t,a,b,q)f^{(1)}(\tilde{t},\tilde{a},\tilde{b},\tq^{-1})
-\tau\sum_{j=1}^{r}f^{(qb_{j}^{-1})}(t,a,b,q)
f^{(qb_{j}^{-1})}(\tilde{t},\tilde{a},\tilde{b},
\tq^{-1})\,.
\ee
Now $\calI(z+k+j\tau,\alpha,\beta,\tau)$ is nothing but the $(j+1,k+1)$ entry
of the matrix
\be
  U(t,a,b,q)
  \begin{pmatrix}
    1 & 0 & \dots & 0\\
    0 & -\tau & \dots & 0\\
    : & : & \dots & :\\
    0 & 0 & \dots & -\tau\\
  \end{pmatrix}
  U(\tilde{t},\tilde{a},\tilde{b},\tq^{-1})^{T}\,,
\ee
and therefore, by Equation~\eqref{heine.selfdual.Ueq}, the $(j+1,k+1)$ entry of
\be
  U(t,a,b,q)
  \begin{pmatrix}
    1 & 0 & \dots & 0\\
    0 & \tau & \dots & 0\\
    : & : & \dots & :\\
    0 & 0 & \dots & \tau\\
  \end{pmatrix}
  U(\tilde{t},\tilde{a},\tilde{b},\tq)^{-1}Q(\tilde{t},\tilde{a},\tilde{b},\tq)^{-T}
  \,.
\ee
Therefore, inverting $Q$, we see that the cocycle
\be
  \Om(z,\alpha,\beta,\tau)
  =
  U(\tilde{t},\tilde{a},\tilde{b},\tq)
  \begin{pmatrix}
    1 & 0 & \dots & 0\\
    0 & \tau & \dots & 0\\
    : & : & \dots & :\\
    0 & 0 & \dots & \tau\\
  \end{pmatrix}^{-1}
  U(t,a,b,q)^{-1}
\ee
extends to a meromorphic function for $\tau\in\BC'$. This complete the proof of
Theorem~\ref{thm.heine}.

\begin{remark}
\label{rem.updown}
Note that the state integral ~\eqref{heine.stateint} is absolutely convergent
and its contour of integration can be pushed either upwards or downwards. Doing
so, the integral factorises in two different ways, one giving the $U$-cocycle
and another giving the $V$-cocycle. This explains the equality of the two
cocycles from first principles.
\end{remark}

\subsection{Resonance}

In this section we discuss in detail the resonant generalised $q$-hypergeometric
equation~\eqref{gener.qhyper}, i.e., the case where at least one of the ratios
$a_i/a_j$ (for $i \neq j$), $b_i/b_j$ (for $i \neq j$) or $a_i/b_j$ (for some $i$
and $j$) is an integer power of $q$. For simplicity, we will 
consider only the case of $r=2$, although our arguments remain valid for all
$r$. When $r=2$, the system of equations ~\eqref{heine.equs} is given by
{\small
\be
\label{2phi1equs}
\begin{aligned}
  (1-t)f(t,a,b,c,q)-(q^{-1}c+1-t(a+b))f(qt,a,b,c,q)+(q^{-1}c-tab)f(q^2t,a,b,c,q)
  &=0\\
  f(t,q^{-1}a,b,c,q)-q^{-1}af(qt,q^{-1}a,b,c,q)-(1-q^{-1}a)f(t,a,b,c,q)
  &=0\\
  f(t,a,q^{-1}b,c,q)-q^{-1}bf(qt,a,q^{-1}b,c,q)-(1-q^{-1}b)f(t,a,b,c,q)
  &=0\\
  f(t,a,b,c,q)-q^{-1}cf(qt,a,b,c,q)-(1-q^{-1}c)f(t,a,b,q^{-1}c,q)
  &=0\,,
\end{aligned}
\ee
}

\noindent
with Equation~\eqref{2phi1} being one such solution. We can specialise $a,b,c$ so
that some of $a$, $b$, $c$, $a/b$, $a/c$ or $b/c$ lies in $q^{\BZ}$. All of these
conditions can be seen to be special
points of the monodromy matrix~\eqref{Mheine} and various special properties of the
equations appear like, for example, submodules. 

We now present two examples of these special points in the simplest case or
$r=2$. These can all be deduced from Theorem~\ref{thm.heine}.

\noindent
$\bullet \,\, b=c$. 
The first Equation~\eqref{2phi1equs} now takes the form
\be
\begin{aligned}
  &(1-q^{-1}b\s_{t})\left((1-\s_{t})-t(1-a\s_{t})\right)\\
  &=\left((1-\s_{t})(1-q^{-1}b\s_{t})-t(1-a\s_{t})(1-b\s_{t})\right)f\\
  &=0\,.
\end{aligned}
\ee
This has normalised solutions
\be
\begin{aligned}
  f^{(1)}(t,a,b,q)
  &=
  \sum_{k=0}^{\infty}\frac{(a;q)_{k}}{(q;q)_{k}}t^{k}\\
  f^{(qb^{-1})}(t,a,b,q)
  &=
  \frac{(a;q)_{\infty}(q^{2}b^{-1};q)_{\infty}
  \theta(q^{-1}bt;q)(q;q)_{\infty}^{2}}{(qab^{-1};q)_{\infty}\theta(q^{-1}b;q)
  \theta(t;q)}\sum_{k=0}^{\infty}\frac{(qab^{-1};q)_{k}}{(q^{2}b^{-1};q)_{k}}t^{k}\,.
\end{aligned}
\ee
Note that the first generates a submodule and indeed satisfies
\be
\left(\left((1-\s_{t})-t(1-a\s_{t})\right)f^{(1)}\right)(t,q)
=(1-t)f^{(1)}(t,a,b,q)-(1-at)f^{(1)}(qt,a,b,q)=0\,.
\ee
We note that for the value at $t=0$ and the $q$-difference equation we see that
\be
  f^{(1)}(t,a,b,q)
  =
  \frac{(at;q)_{\infty}}{(t;q)_{\infty}}
\ee
a classical result known as the $q$-binomial theorem. From Theorem~\ref{thm.ex1} or
even Theorem~\ref{thm.heine} with $r=1$, we then see that this is a modular
$q$-holonomic submodule. Now the second solution satisfies the inhomogeneous equation
\be
(1-t)f^{(qb^{-1})}(t,a,b,q)-(1-at)f^{(qb^{-1})}(qt,a,b,q)
=\frac{(a;q)_{\infty}(q^{2}b^{-1};q)_{\infty}
  \theta(q^{-1}bt;q)(q;q)_{\infty}^{2}}{(qab^{-1};q)_{\infty}
  \theta(q^{-1}b;q)\theta(t;q)}
\ee
where we note that the RHS is of course annihilated by $(1-q^{-1}b\s_{t})$. The full
module can then be shown to be modular using elementary functions holomorphic for
$\tau\in\BC'$ times the state integral
% \be
% \begin{aligned}
% \int_{-i\sqrt{\tau}(\BR+i\ve)}
% \frac{(-q^{1/2}\e(x);q)_{\infty}(-\tq^{1/2}\e(\frac{x+\alpha}{\tau});\tq)_{\infty}}
% {(-\tq^{1/2}\e(\frac{x}{\tau});\tq)_{\infty}(-q^{1/2}\e(x+\alpha);q)_{\infty}}
% \frac{\e\left(\frac{zx}{\tau}\right)}{1-\e(x/\tau+\beta/\tau-1/2-1/2\tau)}
% dx \,.
% \end{aligned}
% \ee
\be
\begin{aligned}
\int_{\calC}
\frac{\Phi_{\sfb}(x)}
{\Phi_{\sfb}(x+i\sfb^{-1}\alpha)}
\frac{\exp\left(-2\pi\frac{zx}{\sfb}\right)}{1+\tq^{1/2}
  \exp\left(-2\pi\frac{x}{\sfb}\right)\e(\beta/\tau)}
dx \,.
\end{aligned}
\ee
Note that this state integral is of course the same as the one in
Equation~\eqref{heine.stateint} where $\beta_{1}=\gamma=\beta+1=\alpha_{2}+1$.

\noindent
$\bullet \,\, c=q$. 
The first Equation~\eqref{2phi1equs} now takes the form
\be
    (1-t)f(t,a,b,q)-(2-t(a+b))f(qt,a,b,q)+(1-tab)f(q^2t,a,b,q)=0\,.
\ee
Notice that this now has indicial polynomial $(1-\rho)^2$. Therefore, we expand
using the Frobenius method to find solutions which are the coefficients of $\ve$
in the expansion to order $O(\ve^2)$ of 
\be
\begin{aligned}
  f^{(1,0)}(t,a,b,q)&=\frac{(qe^{\ve};q)_{\infty}^{2}}{
    (ae^{\ve};q)_{\infty}(be^{\ve};q)_{\infty}}\sum_{k=0}^{\infty}
  \frac{(ae^{\ve};q)_{k}(be^{\ve};q)_{k}}{(qe^{\ve};q)_{k}^{2}}t^{k}
  \frac{\theta(e^{-\ve}t;q)}{\theta(t;q)}\\
  &=f^{(1,0)}(t,a,b,q)+f^{(1,1)}(t,a,b,q)\ve+O(\ve^2)\,.
\end{aligned}
\ee
Then considering the state integral
% \be
% \begin{aligned}
% \int_{-i\sqrt{\tau}(\BR+i\ve)}
% \frac{(-q^{1/2}\e(x);q)_{\infty}^{2}(-\tq^{1/2}
% \e(\frac{x+\alpha}{\tau});\tq)_{\infty}
% (-\tq^{1/2}\e(\frac{x+\beta}{\tau});\tq)_{\infty}}
% {(-\tq^{1/2}\e(\frac{x}{\tau});\tq)_{\infty}^{2}(-q^{1/2}
% \e(x+\alpha);q)_{\infty}(-q^{1/2}\e(x+\beta);q)_{\infty}}
% \e\left(\frac{zx}{\tau}\right)
% dx \,.
% \end{aligned}
% \ee
\be
\begin{aligned}
\int_{\calC}
\frac{\Phi_{\sfb}(x)^2}
{\Phi_{\sfb}(x+i\sfb^{-1}\alpha)\Phi_{\sfb}(x+i\sfb^{-1}\beta)}
\exp\left(-2\pi\frac{zx}{\sfb}\right)
dx \,.
\end{aligned}
\ee
along with Equation~\eqref{S.mod.thd} we can show that the module is modular with
this special value.

%%%%%%%%%%%%%%%%%%%%%%%%%%%%%%%%%%%%%%%%%%%%%%%%%%%%%%%%%%%%%%%%%%%%%%%%%%%%
%%%%%%%%%%%%%%%%%%%%%%%%%%%%%%%%%%%%%%%%%%%%%%%%%%%%%%%%%%%%%%%%%%%%%%%%%%%% 

\section{Proof of Theorems~\ref{thm.ex1} and~\ref{thm.ex2}}
\label{sec.thm12}

\subsection{The $q$-Pochhammer symbol}
\label{sub.ex1}

This section is devoted to the proof of Theorem~\ref{thm.ex1}. 
The $q$-difference equation~\eqref{ex1} can be written in operator form
as $((1-qt)\sigma-1)f=0$, with the Newton polygon shown in Figure~\ref{f.ex1}.

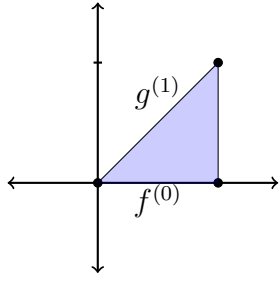
\begin{figure}[!htpb]
\begin{center}
\begin{tikzpicture}[scale=0.8,baseline=-3]
\draw[thick,<->] (0,-1.5) -- (0,3);
\draw[thick,<->] (-1.5,0) -- (3,0);
\foreach \y in {2}
\draw[thick] (2pt,\y) -- (-2pt,\y);
\filldraw (0,0) circle (2pt);
\filldraw (2,2) circle (2pt);
\filldraw (2,0) circle (2pt);
\draw (0,0) -- (2,2);
\draw (2,2) -- (2,0);
\fill[blue,opacity=0.2] (0,0) -- (2,2) -- (2,0) -- cycle;
\node at (1,1.5) {$g^{(1)}$};
\node at (1,-0.3) {$f^{(0)}$};
\end{tikzpicture}
\caption{The Newton polygon of Equation~\eqref{ex1}.}
\label{f.ex1}
\end{center}
\end{figure}

\noindent
The lower Newton polygon has one edge of slope zero. Applying the
Frobenius method, we seek a formal power series solution of the form
\be
f^{(0)}(t,q)
=
\sum_{k=0}^{\infty}\alpha_{k}(q)t^{k}\frac{\theta(\rind^{-1}t;q)}{\theta(t;q)}
\quad\text{ where }\quad
q^{k}\rind\alpha_{k}(q)-q^{k}\rind\alpha_{k-1}(q)-\alpha_{k}(q)=0.
\ee
Since $\alpha_{k}(q)=0$ for $k<0$, setting $k=0$ in the above equation
implies the vanishing of the indicial polynomial
\be
(\rind-1)\alpha_{0}(q)=0,
\ee
giving $\rind=1$. Then we have
\be
\frac{\alpha_{k}(q)}{\alpha_{k-1}(q)}
=
-\frac{q^{k}}{1-q^{k}}.
\ee
Therefore, normalising so that $\alpha_{0}(q)=1$,
\be
\label{ex1.f0}
f^{(0)}(t,q)
=
\sum_{k=0}^{\infty}(-1)^{k}\frac{q^{k(k+1)/2}}{(q;q)_{k}}t^{k}
=
(qt;q)_{\infty} \,.
\ee
This solution is convergent at $t=0$, in fact it is an entire function of $t$.

The upper Newton polygon has one edge of slope of one. Therefore, we must multiply
by a $\theta$-function to get a slope zero, \emph{i.e.} a power series solution.
The new Newton polygon is as follows.
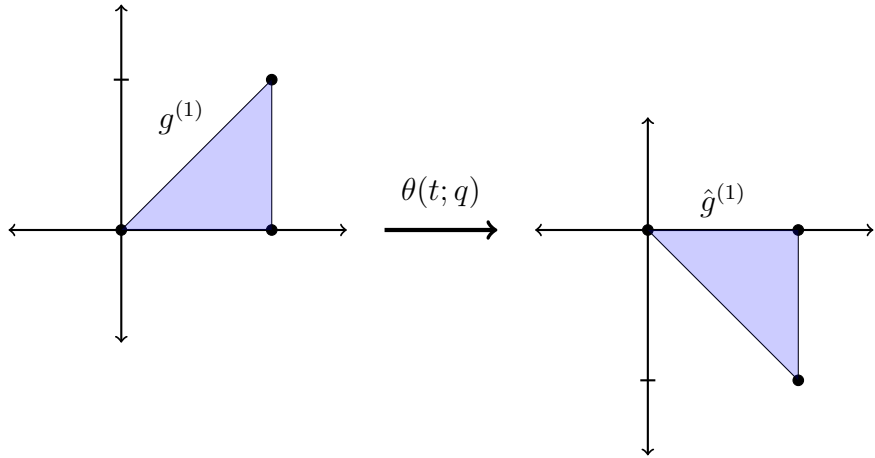
\begin{center}
\begin{tikzpicture}[scale=1]
\draw[thick,<->] (0,-1.5) -- (0,3);
\draw[thick,<->] (-1.5,0) -- (3,0);
\foreach \y in {2}
\draw[thick] (0.1,\y) -- (-0.1,\y);
\filldraw (0,0) circle (2pt);
\filldraw (2,2) circle (2pt);
\filldraw (2,0) circle (2pt);
\draw (0,0) -- (2,2);
\draw (2,2) -- (2,0);
\fill[blue,opacity=0.2] (0,0) -- (2,2) -- (2,0) -- cycle;
\node at (0.8,1.5) {$g^{(1)}$};
\draw[ultra thick, ->] (3.5,0) -- (5,0);
\node at (4.25,0.5) {$\theta(t;q)$};
\draw[thick,<->] (0+7,-3) -- (0+7,1.5);
\draw[thick,<->] (-1.5+7,0) -- (3+7,0);
\draw[thick] {(0.1+7,-2)} -- (-0.1+7,-2);
\filldraw (0+7,0) circle (2pt);
\filldraw (2+7,-2) circle (2pt);
\filldraw (2+7,0) circle (2pt);
\draw (0+7,0) -- (2+7,-2);
\draw (2+7,-2) -- (2+7,0);
\fill[blue,opacity=0.2] (0+7,0) -- (2+7,-2) -- (2+7,0) -- cycle;
\node at (1+7,0.4) {$\hat{g}^{(1)}$};
\end{tikzpicture}
\end{center}
Therefore, the top edge has solution of the form
\be
g^{(1)}(t,q)
=
\theta(t;q)\hat{g}^{(1)}(t,q), \qquad \hat{g}^{(1)}(t,q)
=
\theta(t;q)
\sum_{k=0}^{\infty}\beta_{k}(q)t^{-k}\frac{\theta(\rind^{-1}t;q)}{\theta(t;q)}
\ee
where
\be
-q^{-k}\rind\beta_{k-1}(q)+q^{-k}\rind\beta_{k}(q)-\beta_{k}(q)=0.
\ee
As $\beta_{k}(q)=0$ for $k<0$, we get indicial polynomial
\be
(\rind-1)\beta_{0}(q)=0
\ee
and so $\rind=1$. Then we have
\be
\frac{\beta_{k}}{\beta_{k-1}}
=
\frac{1}{1-q^{k}}.
\ee
Therefore, normalising so that $\beta_{0}(q)=\frac{1}{(q;q)_{\infty}}$,
we obtain that
\be
\label{ex1.g1}
g^{(1)}(t,q)
=
\frac{\theta(t;q)}{(q;q)_{\infty}}\sum_{k=0}^{\infty}\frac{t^{-k}}{(q;q)_{k}}
=
\frac{\theta(t;q)}{(t^{-1};q)_{\infty}(q;q)_{\infty}} \,.
\ee
It follows that the monodromy is given by
\be
\vM(t,q)
=
V(t,q)^{-1}U(t,q)
=
\frac{(qt;q)_{\infty}(t^{-1};q)_{\infty}(q;q)_{\infty}}{\theta(t;q)}
=
1 
\ee
and the cocycles are equal and given by 
\be
\Om_{U}(z,\tau)
=
(U|_{0}S)(z,\tau) U(z,\tau)^{-1}
=
\Om_{U}(z,\tau)
=
(V|_{0}S)(z,\tau) V(z,\tau)^{-1}
=
\frac{(\ti q\ti t;\ti q)_{\infty}}{(qt;q)_{\infty}} \,.
\ee
Notice that
\be
\frac{(\ti q\ti t;\ti q)_{\infty}}{(qt;q)_{\infty}}
=
\Phi_{\sfb}\left(\frac{iz}{\sfb}
  +\frac{i\sfb}{2}+\frac{1}{2i\sfb}\right)^{-1}
\ee
where $\Phi$ is the Faddeev quantum dilogarithm function~\cite{Faddeev,
FK-QDL}. This function extends to a meromorphic function of $(z,\tau)\in
\BC \times \BC'$, with poles at $z\in \BZ_{\geq0}+\BZ_{\geq0}\tau$. Noting that $\Omega_{T}=1$,
it then follows from Theorem~\ref{thm.ST} that
$\Om_{U,\gamma}$ also extends. This is discussed in detail in upcoming
work~\cite{GKZ:cocycle}. This proves that this is a modular
$q$-holonomic module.

\subsection{The Appell-Lerch sums}
\label{sub.ex2}

This section is devoted to the proof of Theorem~\ref{thm.ex2}.
The $q$-difference equation~\eqref{ex2} has Newton polygon shown in
Figure~\ref{f.ex2}.

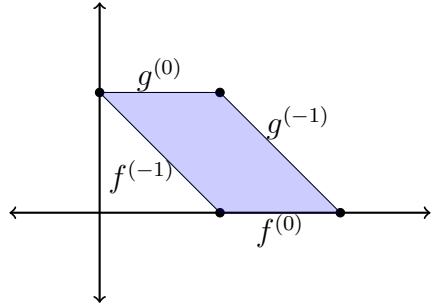
\begin{figure}[!htpb]
\begin{center}
\begin{tikzpicture}[scale=0.8,baseline=-3]
\draw[<->,thick] (-1.5,0) -- (5.5,0);
\draw[<->,thick] (0,-1.5) -- (0,3.5);
\filldraw (2,0) circle (2pt);
\filldraw (2,2) circle (2pt);
\filldraw (0,2) circle (2pt);
\filldraw (4,0) circle (2pt);
\draw (0,2) -- (2,0) -- (4,0) -- (2,2) -- cycle;
\fill[blue,opacity=0.2] (0,2) -- (2,0) -- (4,0) -- (2,2) -- cycle;
\node at (3,-0.3) {$f^{(0)}$};
\node at (0.7,0.6) {$f^{(-1)}$};
\node at (1,2.3) {$g^{(0)}$};
\node at (3.3,1.5) {$g^{(-1)}$};
\end{tikzpicture}
\caption{The Newton polygon of Equation~\eqref{ex2}.}
\label{f.ex2}
\end{center}
\end{figure}

\noindent
To begin with, the boundary of the lower Newton polygon consists of one edge of
slope zero and one edge of slope $1$. The edge of slope zero has a solution of
the form
\be
f^{(0)}(t,q)=
\sum_{k=0}^{\infty}\alpha_{k}(q)t^{k}\frac{\theta(\rind^{-1}t;q)}{\theta(t;q)}
\;\text{where}\;
\alpha_{k}(q)q^{2k}\rind^{2}+q^{k}
\rind\alpha_{k-1}(q)-q^{k}\rind\alpha_{k}(q)-\alpha_{k-1}(q)
=0.
\ee
Therefore, as $\alpha_{k}(q)=0$ for $k<0$, we get the indicial polynomial
\be
\alpha_{0}(q)(\rind^{2}-\rind) = 0
\ee
and so $\rind=1$. Then we have
\be
(q^{2k}-q^{k})\alpha_{k}(q)-(1-q^k)\alpha_{k-1}(q)=0
\quad\text{so}\quad
-q^{k}a_{k}=a_{k-1}.
\ee
Therefore, normalising so that $\alpha_{0}(q)=1$,
\be
\hat{f}^{(0)}(t,q)=\sum_{k=0}^{\infty}(-1)^{k}q^{-k(k+1)/2}t^{k}.
\ee
This is of course divergent for $|q|<1$ so we must $q$-Borel resum. The affect on
the Newton polygon is as follows
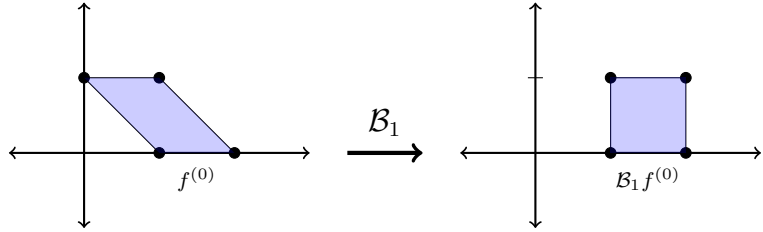
\begin{center}
\begin{tikzpicture}
\draw[<->,thick] (-1,0) -- (3,0);
\draw[<->,thick] (0,-1) -- (0,2);
\filldraw (1,0) circle (2pt);
\filldraw (1,1) circle (2pt);
\filldraw (0,1) circle (2pt);
\filldraw (2,0) circle (2pt);
\draw (0,1) -- (1,0) -- (2,0) -- (1,1) -- cycle;
\fill[blue,opacity=0.2] (0,1) -- (1,0) -- (2,0) -- (1,1) -- cycle;
\node at (1.5,-0.35) {\begin{tiny}$f^{(0)}$\end{tiny}};
\node at (4,0.4) {$\Bor_{1}$};
\draw[->,ultra thick] (3.5,0) -- (4.5,0);
\draw[<->,thick] (-1+6,0) -- (3+6,0);
\draw[<->,thick] (0+6,-1) -- (0+6,2);
\draw (-0.1+6,1) -- (0.1+6,1);
\filldraw (1+6,0) circle (2pt);
\filldraw (1+6,1) circle (2pt);
\filldraw (2+6,1) circle (2pt);
\filldraw (2+6,0) circle (2pt);
\draw (1+6,1) -- (1+6,0) -- (2+6,0) -- (2+6,1) -- cycle;
\fill[blue,opacity=0.2] (1+6,1) -- (1+6,0) -- (2+6,0) -- (2+6,1) -- cycle;
\node at (1.5+6,-0.35) {\begin{tiny}$\Bor_{1}f^{(0)}$\end{tiny}};
\end{tikzpicture}
\end{center}
Now notice that
\be
\Bor_{1}{\hat{f}^{(0)}}(\xi,q) = \sum_{k=0}^{\infty}\xi^{k}
= \frac{1}{1-\xi}
\ee
and so we find
\be
\label{ex2.f0}
f^{(0)}(t,\lambda ,q)
=
\mathcal{L}_{1}\mathcal{B}_{1}(\hat{f}^{(0)})(t,\lambda ,q)
=
\frac{1}{\theta(\lambda ;q)}\sum_{\ell}(-1)^{\ell}
\frac{q^{\ell(\ell+1)/2}\lambda ^{\ell}}{1-\lambda tq^{\ell}}.
\ee
This is the Appell-Lerch sum whose modular properties were studied in Zwegers
thesis~\cite{Zwegersthesis}. Now the bottom edge of slope $-1$ must be divided by
a $\theta$-function to get a zero slope. The effect on the Newton polygon is as
follows.
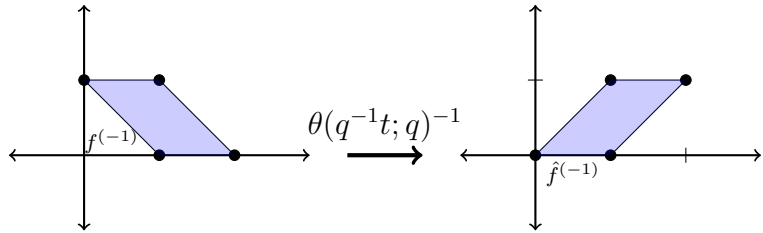
\begin{center}
\begin{tikzpicture}
\draw[<->,thick] (-1,0) -- (3,0);
\draw[<->,thick] (0,-1) -- (0,2);
\filldraw (1,0) circle (2pt);
\filldraw (1,1) circle (2pt);
\filldraw (0,1) circle (2pt);
\filldraw (2,0) circle (2pt);
\draw (0,1) -- (1,0) -- (2,0) -- (1,1) -- cycle;
\fill[blue,opacity=0.2] (0,1) -- (1,0) -- (2,0) -- (1,1) -- cycle;
\node at (0.37,0.19) {\begin{tiny}$f^{(-1)}$\end{tiny}};
\node at (4,0.4) {$\theta(q^{-1}t;q)^{-1}$};
\draw[->,ultra thick] (3.5,0) -- (4.5,0);
\draw[<->,thick] (-1+6,0) -- (3+6,0);
\draw[<->,thick] (0+6,-1) -- (0+6,2);
\draw (-0.1+6,1) -- (0.1+6,1);
\draw (2+6,-0.1) -- (2+6,0.1);
\filldraw (1+6,0) circle (2pt);
\filldraw (1+6,1) circle (2pt);
\filldraw (0+6,0) circle (2pt);
\filldraw (2+6,1) circle (2pt);
\draw (0+6,0) -- (1+6,0) -- (2+6,1) -- (1+6,1) -- cycle;
\fill[blue,opacity=0.2] (0+6,0) -- (1+6,0) -- (2+6,1) -- (1+6,1) -- cycle;
\node at (0.5+6,-0.25) {\begin{tiny}$\hat{f}^{(-1)}$\end{tiny}};
\end{tikzpicture}
\end{center}
Therefore, the bottom edge of slope minus one has a solution of the form
\be
f^{(-1)}(t,q)
=
\theta(q^{-1}t;q)^{-1}\hat{f}^{(-1)}(t,q)
=
\theta(q^{-1}t;q)^{-1}\sum_{k=0}^{\infty}\alpha_{k}(q)t^{k}
\frac{\theta(\rind^{-1}t;q)}{\theta(t;q)}
\ee
where
\be
(1-q^{k}\rind)\alpha_{k}(q)+q^{k}\rind(1-q^{k-1}\rind)\alpha_{k-1}(q)=0.
\ee
Therefore, the indicial polynomial is given by
\be
(1-\rind)\alpha_{0}(q)=0
\ee
and so $\rind=1$. Then we see that
\be
(1-q)\alpha_{1}(q)=0
\ee
and so $\alpha_{k}=0$ for $k\neq 0$. Therefore,
\be
\label{ex2.f-1}
f^{(-1)}(t,q)=\theta(q^{-1}t;q)^{-1}\,.
\ee
Now following similar calculations we find that
\be
\Bor_{1}\hat{g}^{(0)}(\xi,q)
=
-\sum_{k=-\infty}^{-1}\xi^{k}
=
-\frac{\xi^{-1}}{1-\xi^{-1}}
=
\frac{1}{1-\xi}
\ee
and so
\be
\label{ex2.g0}
g^{(0)}(t,\mu ,q)
=
\frac{1}{\theta(\mu ;q)}\sum_{\ell}(-1)^{\ell}
\frac{q^{\ell(\ell+1)/2}\mu ^{\ell}}{1-\mu tq^{\ell}} \,.
\ee
Then finally, we have
\be
\label{ex2.g1}
g^{(-1)}(t,q)=\theta(q^{-1}t;q)^{-1} \,.
\ee
Consider the inhomogeneous gauge transformation (see Section~\ref{sub.cat})
\be
U(t,\lambda,q)
=
\begin{pmatrix}
  0 & 1\\
  1 & -t
\end{pmatrix}
\begin{pmatrix}
  1 & 0\\
  f^{(0)}(t,\lambda,q) & f^{(-1)}(t,q)
\end{pmatrix}
\ee
and the similar one for $V$. With the definition of the fundamental matrices
given in Equations~\eqref{UVex2} and associated monodromy, we then have that

\be
\begin{aligned}
\vM(t,\lambda,q)
&=
\begin{pmatrix}
  1 & 0\\
  g^{(0)}(t,\mu,q) & g^{(-1)}(t,q)
\end{pmatrix}^{-1}
\begin{pmatrix}
  1 & 0\\
  f^{(0)}(t,\lambda,q) & f^{(-1)}(t,q)
\end{pmatrix}\\
&=
\theta(q^{-1}t,q)
\begin{pmatrix}
  \theta(q^{-1}t,q)^{-1} & 0\\
  -g^{(0)}(t,\mu,q) & 1
\end{pmatrix}
\begin{pmatrix}
  1 & 0\\
  f^{(0)}(t,\lambda,q) & \theta(q^{-1}t,q)^{-1}
\end{pmatrix}\\
&=
\begin{pmatrix}
  1 & 0\\
  \vM_{2,1}(t,\lambda ,\mu ,q) & 1
\end{pmatrix}
\end{aligned}
\ee

where

\be
\begin{aligned}
\vM_{2,1}(t,\lambda ,\mu ,q)
&=
\theta(q^{-1}t;q)f^{(0)}(t,\lambda ,q)-\theta(q^{-1}t;q)g^{(0)}(t,\mu ,q)\\
&=
\frac{\theta(q^{-1}t;q)}{\theta(\lambda ;q)}\sum_{\ell}(-1)^{\ell}
\frac{q^{\ell(\ell+1)/2}\lambda ^{\ell}}{1-\lambda tq^{\ell}}
-\frac{\theta(q^{-1}t;q)}{\theta(\mu ;q)}\sum_{\ell}(-1)^{\ell}
\frac{q^{\ell(\ell+1)/2}\mu ^{\ell}}{1-\mu tq^{\ell}} \,.
\end{aligned}
\ee
Using Equation~(7) of Proposition~1.4 of \cite{Zwegersthesis}, we can deduce that
\be
\label{M21sander}
\vM_{2,1}(t,\lambda ,\mu ,q)
=
-\frac{(q;q)_{\infty}^{3}\theta(q^{-1}t;q)\theta(\lambda ^{-1}\mu ;q)
  \theta(\lambda ^{-1}\mu ^{-1}t^{-1};q)}{\theta(\lambda ^{-1};q)\theta(\mu ;q)
  \theta(\lambda ^{-1}t^{-1};q)\theta(\mu ^{-1}t^{-1};q)}.
\ee
We now give a second proof of the above equation using elliptic functions and
residues, which is more general and applicable to our third example. Let $m_{2,1}$
denote the function defined by the right-hand side of~\eqref{M21sander}.
We need to prove that the function $\calE_1$, defined by
\be
\calE_1(t,\lambda,\mu,q)
:=
f^{(0)}(t,\lambda,q)
-g^{(0)}(t,\mu,q)
-\frac{m_{2,1}(t,\lambda,\mu,q)}{\theta(q^{-1}t;q)}
\ee
is identically zero. This will follow from the facts, that
$\calE_1$ is holomorphic at $\BC^\times$ (see Lemma~\ref{lem.resAL} below),
and a solution of a first order equation~\eqref{AL.hom.equ}
(see Lemma~\ref{lem.vanish}). 

\begin{lemma}
\label{lem.vanish}
  If $h(t,q)$ is a solution to the equation
\be
\label{AL.hom.equ}
h(qt,q)+th(t,q)=0
\ee
that is holomorphic for $t\in\BC^{\times}$ then $h(t,q)=0$.
\end{lemma}

\begin{proof}
  Every solution is of the form $C(t,q)/\theta(q^{-1}t;q)$ for some elliptic
  function $C(t,q)$. Therefore, $C(t,q)=h(t,q)\theta(q^{-1}t;q)$ is a
  holomorphic elliptic function and therefore constant. However, this implies
  that $h$ has simple poles unless $C=0$.
\end{proof}

\begin{lemma}
  \label{lem.resAL}
  \rm{(a)} The function $f^{(0)}(t,\lambda,q)$ has simple poles at
$t\in \lambda^{-1}q^{\BZ}$ with residue
\be
\begin{aligned}
\Res_{t=q^{-m}\lambda^{-1}}f^{(0)}(t,\lambda,q)\frac{dt}{2\pi it}
&=
\frac{-1}{\theta(q^{m}\lambda;q)}\\
\Res_{t=q^{-m}\mu^{-1}}g^{(0)}(t,\mu,q)\frac{dt}{2\pi it}
&=
\frac{-1}{\theta(q^{m}\mu;q)} \,.
\end{aligned}
\ee
  \rm{(b)}
  The function $\theta(q^{-1}t;q)^{-1}m_{2,1}(t,\lambda,\mu,q)$ has simple
  poles at $t\in q^{\BZ}\lambda^{-1},q^{\BZ}\mu^{-1}$ with residues
  \be
\begin{aligned}
  \Res_{t=q^{-m}\lambda^{-1}}\theta(q^{-1}t;q)^{-1}
  m_{2,1}(t,\lambda,\mu,q)\frac{dt}{2\pi it}
&=
\frac{-1}{\theta(q^{m}\lambda;q)}\\
\Res_{t=q^{-m}\mu^{-1}}\theta(q^{-1}t;q)^{-1}
m_{2,1}(t,\lambda,\mu,q)\frac{dt}{2\pi it}
&=
\frac{1}{\theta(q^{m}\mu;q)}\,.
\end{aligned}
\ee
\end{lemma}
\begin{proof}
  Part (a) follows from a calculation of the residue of the only term in the
  sum that contributes to the residue in Equations~\eqref{ex2.f0}~\eqref{ex2.g0}
  and Equation~\eqref{theta.fun2}. Now for part (b) we note that
\be
\Res_{t=1}\frac{1}{\theta(t;q)}\frac{dt}{2\pi it}=\frac{1}{(q;q)_{\infty}^{3}}
\ee
which follows easily from the Jacobi triple product~\eqref{theta} for example.
Therefore, we see
\be
\begin{aligned}
  &\Res_{t=q^{-m}\lambda^{-1}}
  \frac{m_{2,1}(t,\lambda,\mu,q)}{\theta(q^{-1}t;q)}\frac{dt}{2\pi it}\\
&=
(-1)^{m}q^{m(m+1)/2}\Res_{t=1}
-\frac{(q;q)_{\infty}^{3}\theta(\lambda ^{-1}\mu ;q)
\theta(\mu ^{-1}t^{-1};q)}{\theta(\lambda ^{-1};q)\theta(\mu ;q)
\theta(t^{-1};q)\theta(\lambda\mu ^{-1}t^{-1};q)}\frac{dt}{2\pi it}\\
&=
(-1)^{m}q^{m(m+1)/2}
\frac{\theta(\lambda ^{-1}\mu ;q)
\theta(\mu ^{-1};q)}{\theta(\lambda ^{-1};q)\theta(\mu ;q)
\theta(\lambda\mu ^{-1};q)}\\
&=
\frac{-1}{\theta(q^{m}\lambda;q)}
\end{aligned}
\ee
where we have repeatedly used Equations~\eqref{theta.fun}
and~\eqref{theta.fun2}. A similar computation calculates the other residues
or using Equation~\eqref{theta.fun} we can show that
\be
m_{2,1}(t,\lambda,\mu,q)=-m_{2,1}(t,\mu,\lambda,q) \,.
\ee
\end{proof}

Noting that $\calE_{1}(t,\lambda,\mu,q)$ satisfies Equation~\eqref{AL.hom.equ}
these lemmata show that the potential simple poles of $\calE_{1}(t,\lambda,\mu,q)$
cancel and therefore it is holomorphic on $t\in\BC^{\times}$ and therefore
vanishes.

For completeness, we give a third proof of~\eqref{M21sander} using  
Lemma~\ref{lem.fund}.
\be
\begin{aligned}
\vM_{2,1}(t,\lambda ,\mu ,q)
&=
\theta(q^{-1}t;q)\Lap_{1}\left(\frac{1}{1-\xi}\right)(t,\lambda ;q)
-\theta(q^{-1}t;q)\Lap_{1}\left(\frac{1}{1-\xi}\right)(t,\mu ;q)\\
&=
\frac{\theta(q^{-1}t;q)\theta(\lambda ^{-1}\mu ;q)(q;q)_{\infty}^{3}}{
  \theta(\lambda ^{-1};q)\theta(\mu ;q)}\Res_{\xi=1}\frac{\theta(\lambda ^{-1}
  \mu ^{-1}t^{-1}\xi;q)}{\theta(\xi\lambda ^{-1}t^{-1};q)
  \theta(\xi\mu ^{-1}t^{-1};q)(1-\xi)}\\
&=
-\frac{(q;q)_{\infty}^{3}\theta(q^{-1}t;q)\theta(\lambda ^{-1}\mu ;q)
  \theta(\lambda ^{-1}\mu ^{-1}t^{-1};q)}{\theta(\lambda ^{-1};q)\theta(\mu ;q)
  \theta(\lambda ^{-1}t^{-1};q)\theta(\mu ^{-1}t^{-1};q)} 
\end{aligned}
\ee
concluding the proof of Equation~\eqref{Mex2b}.
From the explicit expression for $\vM_{2,1}$ from Equation~\eqref{M12ex2} and the
modularity of the Dedekind $\eta$-function and the Jacobi $\theta$-function it
follows that $M$ satisfies Equation~\eqref{Mmod} with weight $\kappa=(0,1)$. It
follows that the two cocycles of Equation~\eqref{UVex2} agree.
This implies that neither cocycle depends on $\lambda$ or $\mu$. This
is equivalent to observations in \cite{Zwegersthesis} that the slash operator
acting on the Appell-Lerch sums depends only on the difference of two Jacobi
variables. Then using~\cite[Proposition~1.5]{Zwegersthesis} we can give an
explicit formula for the cocycle in terms of the Mordell integral~\eqref{mordell}
and elementary functions. In particular, we have
\be
\begin{aligned}
f^{(0)}(qt,\lambda ,q)+tf^{(0)}(t,\lambda ,q)
&=1\\
f^{(-1)}(qt,q)+tf^{(-1)}(t,q)
&=0\\
\ti f^{(0)}(t,\lambda ,q)\;\e
\left(\frac{(z+1/2-\tau/2)^{2}}{2\tau}-\frac{1}{8}\right)\sqrt{\tau}
&=
\tau f^{(0)}(t,\lambda ,q)-\tau t^{-\frac{1}{2}}q^{\frac{1}{8}}h(t,q)\\
f^{(-1)}(\ti t,\ti q)\;\e\left(\frac{(z+1/2-\tau/2)^{2}}{2\tau}
  -\frac{1}{8}\right)
\sqrt{\tau}
&=
f^{(-1)}(t,q) \,.
\end{aligned}
\ee
This means that
\be
\begin{aligned}
\det U(t,\lambda ,q)
&=
f^{(-1)}(qt,q)f^{(0)}(t,q)-f^{(-1)}(t,q)f^{(0)}(qt,q)\\
&=
f^{(-1)}(t,q)\left(-tf^{(0)}(t,q)-f^{(0)}(qt,q)\right)\\
&=
-f^{(-1)}(t,q).
\end{aligned}
\ee
Therefore, we see that
\be
\Om_{S}(z,\tau)
=
\begin{pmatrix}
0 & 1\\
1 & -\tilde{t}
\end{pmatrix}
\begin{pmatrix}
\tau^{-1} & 0\\
\frac{\e\left(\frac{1}{8}\right)}{\sqrt{\tau}}\tilde{t}^{-\frac{1}{2}}
\tq^{\frac{1}{8}}\e\left(\frac{z^2}{2\tau}\right)h(z,\tau) &
\frac{\e\left(\frac{1}{8}\right)}{\sqrt{\tau}}\e
\left(-\frac{(z+1/2-\tau/2)^{2}}{2\tau}\right)
\end{pmatrix}
\begin{pmatrix}
0 & 1\\
1 & -t
\end{pmatrix}^{-1}.
\ee
We can vary the contour in Equation~\eqref{mordell} to get the integral
\be
\int_{\zeta\BR}\frac{e^{\pi i (\tau x^{2} + 2 i z x)}}{2\cosh(\pi x)}dx
\ee
for $|\zeta|=1$ and $\zeta\neq\pm i$. This is a convergent integral to a
holomorphic function in $(z,\tau)$ when $\Im(\tau\zeta^2)>0$. Therefore, we
see that from the uniqueness properties of the solutions to the functional
equations of $h$ the Mordell integral~\cite[Proposition~1.2]{Zwegersthesis}
these functions give an analytic extension of $h$ the Mordell integral to
the cut plane $\BC'$. Therefore, the cocycle $\Om_{U,S}(t,q)$ extends to a
holomorphic function for $z\in\BC$ and $\tau\in\BC'$. Since $\Om_{U,T}=I$,
part (c) of Theorem~\ref{thm.ST} concludes that the cocycle $\Om_U$ is modular.

%%%%%%%%%%%%%%%%%%%%%%%%%%%%%%%%%%%%%%%%%%%%%%%%%%%%%%%%%%%%%%%%%%%%%%%%%%%%
%%%%%%%%%%%%%%%%%%%%%%%%%%%%%%%%%%%%%%%%%%%%%%%%%%%%%%%%%%%%%%%%%%%%%%%%%%%% 

\section{A $q$-difference equation of the $4_1$ knot}
\label{sec.41}

This section is devoted to the proof of Theorem~\ref{thm.41},
Theorem~\ref{thm.41b} and Theorem~\ref{thm.41c}.

\subsection{Solutions}
\label{subsec.41sols}

The $q$-difference equation~\eqref{41x1} has Newton polygon shown in Figure
~\ref{f.41x1}.

\begin{figure}[!htpb]
\begin{center}
\begin{tikzpicture}[scale=0.8,baseline=-3]
\draw[<->,thick] (-3.5,0) -- (3.5,0);
\draw[<->,thick] (0,-1.5) -- (0,3.5);
\foreach \x in {-2,2}
\draw[thick] (\x,-2pt) -- (\x,2pt);
\filldraw (0,0) circle (2pt);
\filldraw (0,2) circle (2pt);
\filldraw (-2,2) circle (2pt);
\filldraw (2,2) circle (2pt);
\node at (-1,2.4) {$g^{(0,0)}$};
\node at (1,2.4) {$g^{(0,1)}$};
\node at (-1.6,0.8) {$f^{(-1)}$};
\node at (1.6,0.8) {$f^{(1)}$};
\draw (0,0) -- (2,2) -- (-2,2) -- cycle;
\fill[blue,opacity=0.2] (0,0) -- (2,2) -- (-2,2) -- cycle;
\end{tikzpicture}
\caption{The Newton polygon of Equation~\eqref{41x1}.}
\label{f.41x1}
\end{center}
\end{figure}
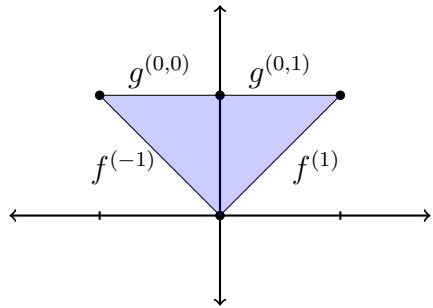

\noindent
The boundary of the lower Newton polygon has edges of slope $-1$ and $1$. The
edge with slope $-1$ must be divided by a $\theta$-function to get a power series
solution. The effect on the Newton polygon is as follows.
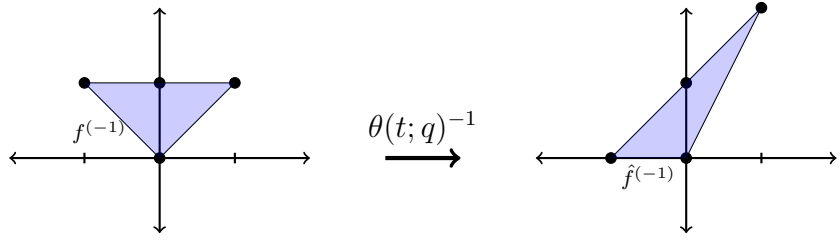
\begin{center}
\begin{tikzpicture}
\draw[<->,thick] (-2,0) -- (2,0);
\draw[<->,thick] (0,-1) -- (0,2);
\foreach \x in {-1,1}
\draw[thick] (\x,-2pt) -- (\x,2pt);
\filldraw (0,0) circle (2pt);
\filldraw (0,1) circle (2pt);
\filldraw (-1,1) circle (2pt);
\filldraw (1,1) circle (2pt);
\draw (0,0) -- (1,1) -- (-1,1) -- cycle;
\fill[blue,opacity=0.2] (0,0) -- (1,1) -- (-1,1) -- cycle;
\node at (-0.8,0.35) {\begin{tiny}$f^{(-1)}$\end{tiny}};
\draw[->,ultra thick] (3,0) -- (4,0);
\node at (3.5,0.4) {$\theta(t;q)^{-1}$};
\draw[<->,thick] (-2+7,0) -- (2+7,0);
\draw[<->,thick] (0+7,-1) -- (0+7,2);
\foreach \x in {-1+7,1+7}
\draw[thick] (\x,-2pt) -- (\x,2pt);
\filldraw (0+7,0) circle (2pt);
\filldraw (0+7,1) circle (2pt);
\filldraw (-1+7,0) circle (2pt);
\filldraw (1+7,2) circle (2pt);
\draw (0+7,0) -- (1+7,2) -- (-1+7,0) -- cycle;
\fill[blue,opacity=0.2] (0+7,0) -- (1+7,2) -- (-1+7,0) -- cycle;
\node at (-0.5+7,-0.25) {\begin{tiny}$\hat{f}^{(-1)}$\end{tiny}};
\end{tikzpicture}
\end{center}
Therefore, the bottom edge of the Newton polygon with slope $-1$ has a solution
of the form
\be
f^{(-1)}(t,q)
=
\theta(t;q)^{-1}\hat{f}^{(-1)}(t,q)
=
\theta(t;q)^{-1}\sum_{k=0}^{\infty}\alpha_{k}^{(-1)}(q)t^{k}
\frac{\theta(\rind^{-1}t;q)}{\theta(t;q)}
\ee
where
\be
(1-q^{-k-1}\rind^{-1})\alpha_{k}^{(-1)}(q)-2\alpha_{k-1}^{(-1)}(q)
-q^{k}\rind\alpha_{k-2}^{(-1)}(q)=0.
\ee
Therefore, this edge has indicial polynomial
\be
(1-q^{-1}\rind^{-1})\alpha_{0}^{(-1)}(q)
\ee
and so $\rind=q^{-1}$. Now notice that if we take the $(-1)$-$q$-Borel transform we
see the effect on the Newton polygon is as follows.
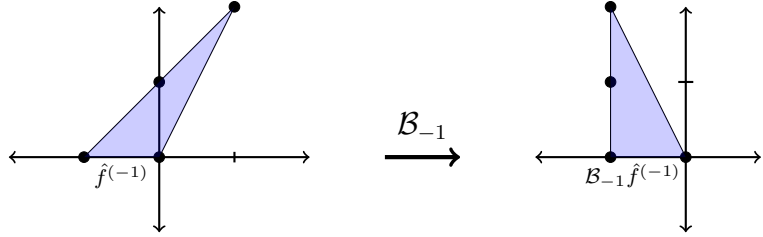
\begin{center}
\begin{tikzpicture}
\draw[<->,thick] (-2,0) -- (2,0);
\draw[<->,thick] (0,-1) -- (0,2);
\foreach \x in {-1,1}
\draw[thick] (\x,-2pt) -- (\x,2pt);
\filldraw (0,0) circle (2pt);
\filldraw (0,1) circle (2pt);
\filldraw (-1,0) circle (2pt);
\filldraw (1,2) circle (2pt);
\draw (0,0) -- (1,2) -- (-1,0) -- cycle;
\fill[blue,opacity=0.2] (0,0) -- (1,2) -- (-1,0) -- cycle;
\node at (-0.5,-0.25) {\begin{tiny}$\hat{f}^{(-1)}$\end{tiny}};
\draw[ultra thick,->] (3,0) -- (4,0);
\node at (3.5,0.4) {$\Bor_{-1}$};
\draw[<->,thick] (-2+7,0) -- (1+7,0);
\draw[<->,thick] (0+7,-1) -- (0+7,2);
\foreach \x in {1}
\draw[thick] (-0.1+7,\x) -- (0.1+7,\x);
\filldraw (0+7,0) circle (2pt);
\filldraw (-1+7,1) circle (2pt);
\filldraw (-1+7,0) circle (2pt);
\filldraw (-1+7,2) circle (2pt);
\draw (0+7,0) -- (-1+7,2) -- (-1+7,0) -- cycle;
\fill[blue,opacity=0.2] (0+7,0) -- (-1+7,2) -- (-1+7,0) -- cycle;
\node at (-0.7+7,-0.25) {\begin{tiny}$\Bor_{-1}\hat{f}^{(-1)}$\end{tiny}};
\end{tikzpicture}
\end{center}
This means that this Borel transform will satisfy the $q$-difference equation
\be
(1-q^{-1}t)^{2}\Bor_{-1}\hat{f}^{(-1)}(q^{-1}t;q)
=
\Bor_{-1}\hat{f}^{(-1)}(t,q).
\ee
Therefore, normalising so that $\alpha_{0}(q)=-(q;q)_{\infty}^{2}$ we have
\be
\Bor_{-1}\hat{f}^{(-1)}(t,q)
=
\sum_{k=0}^{\infty}(-1)^{k}q^{-k(k+1)/2}\alpha_{k}^{(-1)}(q)t^{k}
=
\frac{-(q;q)_{\infty}^{2}}{(t;q)_{\infty}^{2}}.
\ee
Therefore, we see that
\be
\alpha^{(-1)}_{k}(q)=(q;q)_{\infty}^{2}(-1)^{k+1}q^{k(k+1)/2}
\sum_{\ell=0}^{k}\frac{1}{(q;q)_{\ell}(q;q)_{k-\ell}}.
\ee
In particular, we have
\be
\label{ex41.f-1}
f^{(-1)}(t,q)=(q;q)_{\infty}^{2}\theta(q^{-1}t;q)^{-1}\sum_{k=0}^{\infty}
\sum_{\ell=0}^{k}(-1)^{k}\frac{q^{k(k+1)/2}}{(q;q)_{\ell}(q;q)_{k-\ell}}t^{k} \,.
\ee
Now the bottom edge of slope $1$ must be multiplied by a $\theta$-function to get
a power series solution. However, this will then be divergent so we must take
$(1/2)$-$q$-Borel resummation. The effect on the Newton polygon is as follows.
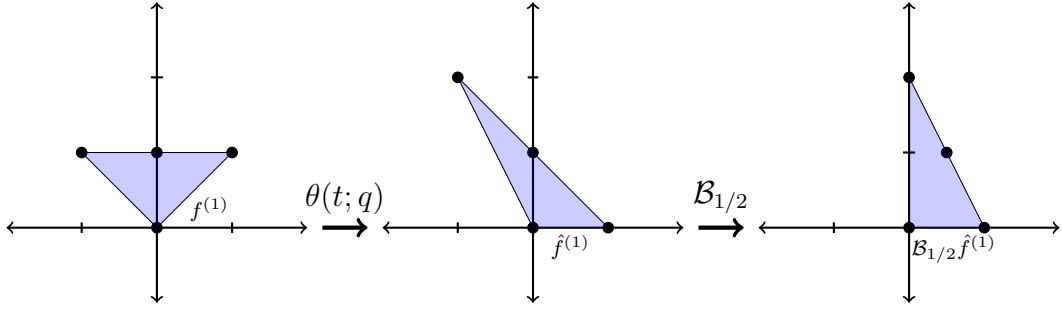
\begin{center}
\begin{tikzpicture}
\draw[<->,thick] (-2-5,0) -- (2-5,0);
\draw[<->,thick] (0-5,-1) -- (0-5,3);
\foreach \x in {-1-5,1-5}
\draw[thick] (\x,-2pt) -- (\x,2pt);
\foreach \x in {2}
\draw[thick] (-5.08,\x) -- (-4.92,\x);
\filldraw (0-5,0) circle (2pt);
\filldraw (0-5,1) circle (2pt);
\filldraw (-1-5,1) circle (2pt);
\filldraw (1-5,1) circle (2pt);
\draw (0-5,0) -- (1-5,1) -- (-1-5,1) -- cycle;
\fill[blue,opacity=0.2] (0-5,0) -- (1-5,1) -- (-1-5,1) -- cycle;
\node at (0.7-5,0.25) {\begin{tiny}$f^{(1)}$\end{tiny}};
\draw[->,ultra thick] (-2.8,0) -- (-2.2,0);
\node at (-2.5,0.4) {$\theta(t;q)$};
\draw[<->,thick] (-2,0) -- (2,0);
\draw[<->,thick] (0,-1) -- (0,3);
\foreach \x in {-1,1}
\draw[thick] (\x,-2pt) -- (\x,2pt);
\foreach \x in {2}
\draw[thick] (-2pt,\x) -- (2pt,\x);
\filldraw (0,0) circle (2pt);
\filldraw (0,1) circle (2pt);
\filldraw (-1,2) circle (2pt);
\filldraw (1,0) circle (2pt);
\draw (0,0) -- (-1,2) -- (1,0) -- cycle;
\fill[blue,opacity=0.2] (0,0) -- (-1,2) -- (1,0) -- cycle;
\node at (0.5,-0.25) {\begin{tiny}$\hat{f}^{(1)}$\end{tiny}};
\draw[->,ultra thick] (2.2,0) -- (2.8,0);
\node at (2.5,0.4) {$\Bor_{1/2}$};
\draw[<->,thick] (-2+5,0) -- (2+5,0);
\draw[<->,thick] (0+5,-1) -- (0+5,3);
\foreach \x in {1}
\draw[thick] (4.92,\x) -- (5.08,\x);
\foreach \x in {4}
\draw[thick] (\x,-2pt) -- (\x,2pt);
\filldraw (0+5,0) circle (2pt);
\filldraw (1/2+5,1) circle (2pt);
\filldraw (1+5,0) circle (2pt);
\filldraw (0+5,2) circle (2pt);
\draw (0+5,0) -- (0+5,2) -- (1+5,0) -- cycle;
\fill[blue,opacity=0.2] (0+5,0) -- (0+5,2) -- (1+5,0) -- cycle;
\node at (0.6+5,-0.28) {\begin{tiny}$\Bor_{1/2}\hat{f}^{(1)}$\end{tiny}};
\end{tikzpicture}
\end{center}
By symmetry of the $q$-difference equation, one can easily use the previous
solution to check that the formal solution to this edge is given by
$q\mapsto q^{-1}$ which gives
\be
f^{(1)}(t,q)
=
\frac{\theta(t;q)}{(q;q)_{\infty}^{2}}\hat{f}^{(1)}(t,q)
=
\frac{\theta(t;q)}{(q;q)_{\infty}^{2}}\sum_{k=0}^{\infty}\sum_{\ell=0}^{k}
\frac{q^{\ell^2-\ell k}}{(q;q)_{\ell}(q;q)_{k-\ell}}t^{k}.
\ee
One then sees that
\be
\begin{aligned}
\Bor_{1/2}\hat{f}^{(1)}(\xi,q)
&=
\sum_{k=0}^{\infty}\sum_{\ell=0}^{k}(-1)^{k}
\frac{q^{k(k+1)/4+\ell^2-\ell k}}{(q;q)_{\ell}(q;q)_{k-\ell}}\xi^{k}
&=
\sum_{k,\ell=0}^{\infty}(-1)^{k+\ell}
\frac{q^{k(k+1)/4-k\ell/2+\ell(\ell+1)/4}}{(q;q)_{k}(q;q)_{\ell}}\xi^{k+\ell}
\end{aligned}
\ee
is holomorphic for $|\xi|<|q^{-1/4}|$. Then, using the functional equation
\be
(1-q^{1/2}\xi^2)\Bor_{1/2}\hat{f}^{(1)}(\xi,q)+
2\xi\Bor_{1/2}\hat{f}^{(1)}(q^{1/2}\xi,q)-\Bor_{1/2}\hat{f}^{(1)}(q\xi,q)=0,
\ee
we can analytically extend away from $\xi\in\pm q^{-1/4+\frac{1}{2}\BZ_{\leq0}}$
and we see that there are poles at $\xi\in\pm q^{-1/4+\frac{1}{2}\BZ_{\leq0}}$.
Therefore, we finally define
\be
\label{ex41.f1}
\begin{aligned}
f^{(1)}(t,\lambda,q)
&=
\frac{\theta(t;q)}{(q;q)_{\infty}^{2}}\Lap_{1/2}\Bor_{1/2}
\hat{f}^{(1)}(t,\lambda,q)\\
&=
\frac{\theta(t;q)}{(q;q)_{\infty}^{2}}\sum_{n\in\BZ}\frac{\Bor_{1/2}\hat{f}^{(1)}
(q^{\frac{n}{2}}\lambda t,q)}{\theta(q^{\frac{n}{2}}\lambda;q^{\frac{1}{2}})}\\
&=
\frac{\theta(t;q)}{(q;q)_{\infty}^{2}\theta(\lambda;q^{1/2})}
\sum_{n\in\BZ}(-1)^{n}q^{n(n+1)/4}\lambda^{n}\Bor_{1/2}
\hat{f}^{(1)}(q^{\frac{n}{2}}\lambda t,q).
\end{aligned}
\ee
Now the top of the Newton polygon has one edge of slope $0$. We find the solution
satisfies
\be
g^{(0)}(t,q)
=
\sum_{k=0}^{\infty}\beta_{k}(q)t^{-k}\frac{\theta(\rind^{-1}t;q)}{\theta(t;q)}
\quad\text{where}\quad
\beta_{k}(q)-q^{-k}\rind^{-1}(1-q^{k}\rind^{-1})^{2}\beta_{k+1}(q)=0.
\ee
Therefore, as the indicial polynomial is
\be
(1-q^{-1}\rind^{-1})^{2}\beta_{0}(q)=0
\ee
we take $\rind=q^{-1}e^\ve$ and expand to order $\ve^{2}$ to find
solutions
\be
\begin{aligned}
g^{(0)}(t,\ve,q)
&=
\sum_{k=0}^{\infty}(-1)^{k}\frac{q^{k(k+1)/2}}{(q;q)_{k}^{2}}t^{-1-k}\\
&+\left(\sum_{k=0}^{\infty}
  \left(\frac{1}{2}E_{1}(q)-\frac{1}{2}-\frac{\theta'(t^{-1};q)}{\theta(t^{-1};q)}+\sum_{j=1}^{k}\frac{1+q^{j}}{1-q^{j}}\right)
  (-1)^{k}\frac{q^{k(k+1)/2}}{(q;q)_{k}^{2}}t^{-1-k}\right)\ve+O(\ve^{2})\,,
\end{aligned}
\ee
where $E_{1}(q)=1-4\sum_{j=1}^{\infty}\frac{q^{j}}{1-q^{j}}$.
Therefore, we have solutions
\be
\label{ex41.g00}
\begin{aligned}
g^{(0,0)}(t,q)
&=
\sum_{k=0}^{\infty}(-1)^{k}\frac{q^{k(k+1)/2}}{(q;q)_{k}^{2}}t^{-1-k}\\
g^{(0,1)}(t,q)
&=
\sum_{k=0}^{\infty}\left(\frac{1}{2}E_{1}(q)-\frac{1}{2}-\frac{\theta'(t^{-1};q)}{\theta(t^{-1};q)}+\sum_{j=1}^{k}\frac{1+q^{j}}{1-q^{j}}\right)(-1)^{k}\frac{q^{k(k+1)/2}}{(q;q)_{k}^{2}}t^{-1-k}.
\end{aligned}
\ee

\subsection{Monodromy}

Consider the fundamental matrices $U$ and $V$ given by Equations~\eqref{UVex41}
and the associated monodromy. Using the modular transformation properties of
the Jacobi $\theta$-function Equation~\eqref{thetaS} and the Dedekind $\eta$-function,
it is easy to see that Equation~\eqref{M412} implies~\eqref{M41c}.

Note that each of the functions $\vM_{2,1},\vM_{2,1}|_{-1}T,\vM_{2,1}|_{-1}S$ has
$\SL_2(\BZ)$-stabiliser $\langle T^2,TST\rangle$, $\langle T^2,S\rangle$ and
$\langle ST^2S,T\rangle$, respectively, and that the second group is the
$\theta$-subgroup and the last group is $\Gamma_{0}(2)$. The appearance of
$\Gamma_{0}(2)$ is a consequence of the $(1/2)$-$q$-Borel transform below. 

The rest of this section is devoted to the proof that the monodromy is given
by Equation~\eqref{M412}. To achieve this, we need to write $f^{(-1)}$ and $f^{(1)}$
as linear combinations of $g^{(0,0)}$ and $g^{(0,1)}$, the coefficients in these
expressions give the entries of the monodromy matrix. 

Firstly we determine the second column of the monodromy, using an adaption of an
argument in~\cite{Morita:conn-mon}. Using the $(-1)$-$q$-Laplace transform or the
Meinardus trick (see for example~\cite{Meinardus} and~\cite[p.54]{Zagier:dilog})
and shifting the contour of $\oint:=\int_{|t|=\ve}$ we obtain that
\be
\begin{tiny}
\begin{aligned}
  \label{f1=g01}
&\hat{f}^{(-1)}(t,q)%\\&
=
(q;q)_{\infty}^{2}\oint_{0}\frac{\theta(t/\xi;q)}{(\xi;q)_{\infty}^{2}}
\frac{d\xi}{2\pi i \xi}%\\&
=
-(q;q)_{\infty}^{2}\sum_{k=0}^{\infty}\left(\Res_{\xi=q^{-k}}\right)
\frac{\theta(t/\xi;q)}{(\xi;q)_{\infty}^{2}}\frac{d\xi}{2\pi i \xi}\\
&=
-(q;q)_{\infty}^{2}\sum_{k=0}^{\infty}\left(\Res_{\ve=0}\right)
\frac{\theta(tq^{k}e^{\ve};q)}{(q^{-1}e^{-\ve};q^{-1})_{k}^2
  (e^{-\ve};q)_{\infty}^{2}}\frac{-d\ve}{2\pi i }\\
&=
-(q;q)^{4}\sum_{k=0}^{\infty}\left(\Res_{\ve=0}\right)
\frac{(-1)^{k}q^{-k(k+1)/2}t^{-k}e^{-k\ve}
  \theta(te^{\ve};q)}{\theta(e^{\ve};q)^{2}}
\frac{(qe^{\ve};q)_{\infty}^{2}}{q^{-k(k+1)}
  e^{-2k\ve}(qe^{\ve};q)_{k}^2}\frac{-d\ve}{2\pi i }\\
&=
(q;q)_{\infty}^{6}\Res_{\ve=0}
\frac{\theta(te^{\ve};q)}{\theta(e^{\ve};q)^{2}}
\frac{(qe^{\ve};q)_{\infty}^{2}}{(q;q)_{\infty}^{2}}
\sum_{k=0}^{\infty}
\frac{(-1)^{k}q^{k(k+1)/2}t^{-k}e^{k\ve}}{(qe^{\ve};q)_{k}^2}
\frac{d\ve}{2\pi i }\\
&=
-\theta(t^{-1};q)\sum_{k=0}^{\infty}\left(\frac{1}{2}E_{1}(q)-\frac{1}{2}+\sum_{j=1}^{k}\frac{1+q^{j}}{1-q^{j}}\right)
\frac{(-1)^{k}q^{k(k+1)/2}t^{-k-1}}{(q;q)_{k}^2}+\theta'(t^{-1};q)
\sum_{k=0}^{\infty}\frac{(-1)^{k}q^{k(k+1)/2}t^{-k-1}}{(q;q)_{k}^2}.
\end{aligned}
\end{tiny}
\ee
Therefore, we see that
\be\label{fg}
f^{(-1)}(t,q)
=
g^{(0,1)}(t,q) \,, 
\ee
which implies that $\vM_{1,2}=0$ and $\vM_{2,2}=1$. Lemma~\ref{lem.41det} below
implies that $\det(M)=-1$, therefore 
\be
\label{41:paritalmono}
\vM(t,\lambda,q) =
\begin{pmatrix}
-1 & 0\\
* & 1
\end{pmatrix} \,.
\ee
To finish the proof we will show that the function
\be
\label{41mono:intfunc}
\calE_2:=f^{(1)}+g^{(0,0)}-m_{2,1}g^{(0,1)}
\ee
vanishes identically, where $m_{2,1}$ denotes the function given in
Equation~\eqref{M412}. The function $\calE_2$ is meromorphic of
$t \in \BC^{\times}$ and has potential simple poles located at
$t\in\pm q^{-1/4-\frac{1}{2}\BZ}\lambda^{-1}$ (coming from $f^{(1)}$ and $\vM_{2,1}$)
and potential simple poles at $q^\BZ$ (coming from $g^{(0,1)}$).
Lemma~\ref{lem.resf1} below gives
\be
\label{41res}
\begin{aligned}
&\Res_{t=\pm q^{-1/4-n}\lambda^{-1}}f^{(1)}(t,\lambda,q)\frac{dt}{2\pi it}\\
&=
\frac{\theta(\pm q^{3/4}\lambda^{-1};q)\theta(\pm q^{-1/4};q)
  \theta(\pm q^{-3/4};q)
  \theta(\pm q^{-3/4}\lambda;q)}{2(q;q)_{\infty}^{6}
  \theta(q^{-1}\lambda;q)\theta(q^{-3/2}\lambda;q)}
f^{(-1)}(\pm q^{-1/4-n} \lambda^{-1},q) \,.
\end{aligned}
\ee
On the other hand,
\be
\begin{aligned}
  &\Res_{t=\pm q^{-1/4-n}\lambda^{-1}}m_{2,1}(t,\lambda,q)
  \frac{dt}{2\pi it}\\
&=
\frac{\Res_{t=\pm q^{-1/4-n}\lambda^{-1}}\theta(qt;q)\theta(t\lambda;q)
  \theta(t\lambda q^{-1/2};q)
  \theta(t\lambda^{2}q^{-1/2};q)}{\theta(t\lambda q^{1/4};q)
  \theta(-t\lambda q^{1/4};q)\theta(t\lambda q^{-1/4};q)
  \theta(-t\lambda q^{-1/4};q)\theta(q^{-1}\lambda;q)\theta(q^{-3/2}
  \lambda;q)}\frac{dt}{2\pi it}\\
&=
\frac{(-1)^{n}q^{n(n-1)/2}}{(q;q)_{\infty}^{3}}\frac{\theta(\pm q^{3/4-n}
  \lambda^{-1};q)\theta(\pm q^{-1/4-n};q)\theta(\pm q^{-3/4-n};q)
  \theta(\pm q^{-3/4-n}\lambda;q)}{
  \theta(-q^{-n};q)\theta(q^{-1/2-n};q)
  \theta(-q^{-1/2-n};q)\theta(q^{-1}\lambda;q)\theta(q^{-3/2}\lambda;q)}\\
&=
\frac{\theta(\pm q^{3/4}\lambda^{-1};q)\theta(\pm q^{-1/4};q)
  \theta(\pm q^{-3/4};q)
  \theta(\pm q^{-3/4}\lambda;q)}{(q;q)_{\infty}^{3}
  \theta(-1;q)\theta(q^{-1/2};q)
  \theta(-q^{-1/2};q)\theta(q^{-1}\lambda;q)\theta(q^{-3/2}\lambda;q)}\\
&=
\frac{\theta(\pm q^{3/4}\lambda^{-1};q)\theta(\pm q^{-1/4};q)
  \theta(\pm q^{-3/4};q)
  \theta(\pm q^{-3/4}\lambda;q)}{2(q;q)_{\infty}^{6}
  \theta(q^{-1}\lambda;q)\theta(q^{-3/2}\lambda;q)} \,.
\end{aligned}
\ee
Therefore,
\be
\Res_{t=\pm q^{-1/4-n}\lambda^{-1}} \calE_2(t,\lambda,q) \frac{dt}{2\pi it}=0 \,.
\ee
A similar computation or the fact everything is elliptic in
$\lambda\mapsto q^{1/2}\lambda$ also shows that
\be
\Res_{t=\pm q^{-3/4-n}\lambda^{-1}} \calE_2(t,\lambda,q) \frac{dt}{2\pi it}=0 \,.
\ee
Now notice that $m_{2,1}(q^{\BZ},\lambda,q)=0$ which implies $m_{2,1}g^{(0,1)}$
is holomorphic at $t\in q^{\BZ}$. We then see that $\calE_2(t,\lambda,q)$ extends
to a holomorphic function of $t \in \BC^{\times}$. Therefore, from
Corollary~\ref{cor:41holsol} we see that $\calE_2(t,\lambda,q)=C(q)g^{(0,0)}(t,q)$.
Finally, we note that 
\be
\frac{-C(q)}{(q;q)_{\infty}^{2}}
=
\lim_{r\rightarrow\infty}\frac{\calE_{2}(q^{r}t,\lambda,q)}{\theta(q^{r}t;q)}
=
0
\ee
which follows from Lemma~\ref{lem.fglimits}.
Thus, Equation~\eqref{M412} follows from the residue equation~\eqref{41res}
and from a computation of the determinants of $U$ and $V$. We discuss these
in the next sections.

\begin{lemma}
\label{lem.fglimits}
  For $t$ in some compact set where the functions are holomorphic, we have
\be
\begin{aligned}
\lim_{r\rightarrow\infty}\frac{g^{(0,0)}(q^{r}t,q)}{\theta(q^{r}t;q)}
&=
\frac{-1}{(q;q)_{\infty}^{2}}\qquad
&
\lim_{r\rightarrow\infty}\frac{f^{(-1)}(q^{r}t,q)}{\theta(q^{r}t;q)}
&=
0
\\
\lim_{r\rightarrow\infty}\frac{g^{(0,1)}(q^{r}t,q)}{\theta(q^{r}t;q)}
&=
0\qquad
&
\lim_{r\rightarrow\infty}\frac{f^{(1)}(q^{r}t,q)}{\theta(q^{r}t;q)}
&=
\frac{1}{(q;q)_{\infty}^{2}}\,.
  \end{aligned}
\ee
\end{lemma}

\begin{proof}
We have
\be
\begin{aligned}
g^{(0,0)}(q^{r}t,q)
&=
\sum_{k=0}^{\infty}(-1)^{k}\frac{q^{k(k+1)/2-rk-r}}{(q;q)_{k}^{2}}t^{-k-1}\\
&=
q^{-r(r+1)/2}\sum_{k=0}^{\infty}(-1)^{k}
\frac{q^{(k-r)(k-r+1)/2}}{(q;q)_{k}^{2}}t^{-k-1}\\
&=
(-1)^{r}q^{-r(r+1)/2}t^{-r-1}\sum_{k=-r}^{\infty}(-1)^{k}
\frac{q^{k(k+1)/2}}{(q;q)_{k+r}^{2}}t^{-k}\,.
\end{aligned}
\ee
The first equality then follows from
\begin{align*}
\lim_{r\rightarrow\infty}\frac{g^{(0,0)}(q^{r}t,q)}{\theta(q^{r}t;q)}
&=
\lim_{r\rightarrow\infty}\frac{-1}{\theta(t^{-1};q)}
\sum_{k=-r}^{\infty}(-1)^{k}\frac{q^{k(k+1)/2}}{(q;q)_{k+r}^{2}}t^{-k}
\\
&=
\frac{-1}{\theta(t^{-1};q)}\sum_{k\in\BZ}(-1)^{k}
\frac{q^{k(k+1)/2}}{(q;q)_{\infty}^{2}}t^{-k}
=\frac{-1}{(q;q)_{\infty}^{2}} \,.
\end{align*}
We can show similarly that
\be
\begin{tiny}
\begin{aligned}
g^{(0,1)}(q^{r}t,q)
&=
(-1)^{r}q^{-r(r+1)/2}t^{-r-1}\sum_{k=-r}^{\infty}\left(\frac{1}{2}E_{1}(q)-r-\frac{1}{2}-\frac{\theta'(t^{-1};q)}{\theta(t^{-1};q)}+\sum_{j=1}^{k+r}\frac{1+q^{j}}{1-q^{j}}\right)
(-1)^{k}\frac{q^{k(k+1)/2}}{(q;q)_{k+r}^{2}}t^{-k} \,.
\end{aligned}
\end{tiny}
\ee
Noting that $E_{1}(q)/2-1/2+\sum_{j=1}^{k}\frac{1+q^{j}}{1-q^{j}}=k+O(q^{k})$ completes the proof of the limits of
$g^{(0,0)},g^{(0,1)}$. The limits of $f^{(\pm1)}$ follow from Watson's lemma for
$q$-Borel resummation and keeping track of the $\theta$ prefactors.
\end{proof}

\subsection{Residues}

Note that $g^{(0,0)}$ is holomorphic on $\BC^{\times}\cup\{\infty\}$. In this
subsection we compute the residues of the meromorphic functions $f^{(-1)}=g^{(0,1)}$
and $f^{(1)}$.

\begin{lemma}
  \label{lem.resg01}
  The functions $f^{(-1)}(t,q)=g^{(0,1)}(t,q)$ have simple poles at $t = q^\BZ$ with
  residues
\be
\Res_{t=q^{n}}f^{(-1)}(t,q)\frac{dt}{2\pi i t}
=
\Res_{t=q^{n}}g^{(0,1)}(t,q)\frac{dt}{2\pi i t}
=
g^{(0,0)}(q^{n},q)\,.
\ee
\end{lemma}

\begin{proof}
We have
\be
\Res_{t=q^{k}}\frac{\theta'(t^{-1};q)}{\theta(t^{-1};q)}\frac{dt}{2\pi it}
=
-1 \,.
\ee
To finish we note that the proof of the equality $f^{(-1)}=g^{(0,1)}$ is
independent of this Lemma and follows from Equation~\eqref{f1=g01}.
\end{proof}

\begin{lemma}
  \label{lem.resf1}
  \rm{(a)} The function $\Bor_{1/2}\hat{f}^{(1)}(\xi,q)$ has simple poles at
$\xi=\pm q^{-1/4-\BZ/2}$ with residue
\be
\begin{aligned}
\Res_{\xi=\pm q^{-1/4-n/2}}\Bor_{1/2}\hat{f}^{(1)}(\xi,q)\frac{d\xi}{2\pi i\xi}
&=
\frac{-\theta(\pm q^{-1/4};q^{1/2})}{2(q;q)_{\infty}^{2}}
\sum_{\ell=0}^{k}(\pm1)^{k}\frac{q^{\frac{k(k+2)}{4}}}{(q;q)_{\ell}(q;q)_{k-\ell}}.
\end{aligned}
\ee
  \rm{(b)}
  The function $f^{(1)}(t,\lambda,q)$ has simple poles at
  $t=\pm q^{-1/4+\BZ}\lambda^{-1}$ with residue
  \be
\begin{aligned}
&\Res_{t=\pm q^{-1/4-n}\lambda^{-1}}f^{(1)}(t,\lambda,q)\frac{dt}{2\pi it}\\
&=
\frac{\theta(\pm q^{3/4}\lambda^{-1};q)\theta(\pm q^{-1/4};q)
  \theta(\pm q^{-3/4};q) \theta(\pm q^{-3/4}\lambda;q)}{2(q;q)_{\infty}^{6}
  \theta(q^{-1}\lambda;q)\theta(q^{-3/2}\lambda;q)}
f^{(-1)}(\pm q^{-1/4-n} \lambda^{-1},q) \,.
\end{aligned}
\ee
\end{lemma}

\begin{proof}
For part (a), consider the following auxiliary function,
\be
\label{aux.fun.H}
H_{r}(\xi,q)=\sum_{k}\frac{\xi^{2k}}{(q;q)_{k}(q;q)_{k+r}},
\ee
which we analytically continue to $\xi\notin\pm q^{\frac{1}{2}\BZ_{\leq0}}$.
From the power series at $\xi=0$, we see that
\be
\Bor_{1/2}\hat{f}^{(1)}(q^{-1/4}\xi,q)
=
\sum_{r\in\BZ}(-1)^{r}q^{r^{2}/4}\xi^{r}H_{r}(\xi,q).
\ee
We have the relation
\be
H_{r-1}(\xi,q)-\xi H_{r}(\xi,q)
=
-q^{1-r}\left(H_{r-2}(\xi,q)-H_{r-1}(\xi,q)\right)
\ee
and from this we can see that
\be
\Res_{\xi=1}\left(H_{r-1}(\xi,q)-H_{r}(\xi,q)\right)\frac{d\xi}{2\pi i\xi}
=
(-1)^{r}q^{-r(r-1)/2}\Res_{\xi=1}\left(H_{-1}(\xi,q)-H_{0}(\xi,q)\right)
\frac{d\xi}{2\pi i\xi}.
\ee
Since
\be
\lim_{r\rightarrow\infty}H_{r}(\xi,q)
=
\frac{1}{(q;q)_{\infty}(\xi^{2};q)_{\infty}},
\ee
it follows that 
\be
\label{res.aux.fun.H}
\Res_{\xi=1}H_{r}(\xi,q)\frac{d\xi}{2\pi i\xi}
=
\Res_{\xi=1}H_{0}(\xi,q)\frac{d\xi}{2\pi i\xi}
=
\Res_{\xi=1}H_{\infty}(\xi,q)\frac{d\xi}{2\pi i\xi}
=
\frac{-1}{2(q;q)_{\infty}^{2}}.
\ee
Let
\be
R_{\pm}(n,q)=\Res_{\xi=\pm q^{-1/4-n/2}}
\Bor_{1/2}\hat{f}^{(1)}(\xi,q)\frac{d\xi}{2\pi i\xi} \,.
\ee
Then,
\be
\begin{aligned}
R(0,q)
&=
\Res_{\xi=\pm 1}\sum_{r\in\BZ}(-1)^{r}q^{r^{2}/4}\xi^{r}H_{r}(\xi,q)
\frac{d\xi}{2\pi i\xi}
&=
\frac{-\theta(\pm q^{-1/4};q^{1/2})}{2(q;q)_{\infty}^{2}} \,.
\end{aligned}
\ee
Now from the functional equation of $\Bor_{1/2}\hat{f}^{(1)}(\xi,q)$,
we deduce that
\be
\begin{aligned}
0
&=
\Res_{\xi=\pm q^{-1/4-n/2}}(1-q^{1/2}\xi^2)\Bor_{1/2}\hat{f}^{(1)}(\xi,q)+
2\xi\Bor_{1/2}\hat{f}^{(1)}(q^{1/2}\xi,q)-\Bor_{1/2}\hat{f}^{(1)}(q\xi,q)
\frac{d\xi}{2\pi i\xi}\\
&=
(1-q^{-n})R_{\pm}(n,q)\pm2q^{-1/4-(n-1)/2}R_{\pm}(n-1,q)-R_{\pm}(n-2,q).
\end{aligned}
\ee
Note that $R_{\pm}(n,q)=0$ for $n<0$ and that when $R_{\pm}(-1,q)$ is zero we have
a unique solution determined by $R_{\pm}(0,q)$. One can then check that
\be\label{Rpm}
R_{\pm}(n,q)
=
\frac{-\theta(\pm q^{-1/4};q^{1/2})}{2(q;q)_{\infty}^{2}}
\sum_{\ell=0}^{n}(\pm1)^{n}\frac{q^{\frac{n(n+2)}{4}}}{(q;q)_{\ell}(q;q)_{n-\ell}} \,.
\ee
For part (b), we compute  
\be
\begin{small}
\begin{aligned}
&\Res_{t=\pm q^{-1/4-n}\lambda^{-1}}f^{(1)}(t,\lambda,q)\frac{dt}{2\pi it}\\
&=
\Res_{t=\pm q^{-1/4-n}\lambda^{-1}}\frac{\theta(t;q)}{(q;q)_{\infty}^{2}}
\Lap_{1/2}\Bor_{1/2}\hat{f}^{(1)}(t,q)\frac{dt}{2\pi it}\\
&=
\frac{\theta(\pm q^{-1/4-n}\lambda^{-1};q)}{(q;q)_{\infty}^{2}
  \theta(\lambda;q^{1/2})}
\sum_{k\in\BZ}(-1)^{k}q^{(k-2n)(k-2n-1)/4}\lambda^{-k+2n}R_{\pm}(k,q)\\
&=
\frac{\theta(\pm q^{-1/4-n}\lambda^{-1};q)}{(q;q)_{\infty}^{2}
  \theta(\lambda;q^{1/2})}\frac{-\theta(\pm q^{-1/4};q^{1/2})}{2(q;q)_{\infty}^{2}}
q^{n(2n+1)/2}\lambda^{2n}
\sum_{k=0}^{\infty}\sum_{\ell=0}^{k}(-1)^{k}\frac{
  q^{\frac{k(k+1)}{2}}}{(q;q)_{\ell}(q;q)_{k-\ell}}(\pm q^{-1/4-n}
\lambda^{-1})^{k}\\
&=
-\frac{\theta(\pm q^{-1/4-n}\lambda^{-1};q)\theta(\pm q^{-1/4};q^{1/2})}{
  2(q;q)_{\infty}^{4}\theta(\lambda;q^{1/2})}q^{n(2n+1)/2}\lambda^{2n}
\hat{f}^{(-1)}(\pm q^{-1/4-n/2}\lambda^{-1},q)\\
&=
-\frac{\theta(\pm q^{-1/4-n}\lambda^{-1};q)\theta(\pm q^{-1/4};q)
  \theta(\pm q^{-3/4};q)\theta(\pm q^{-5/4-n}\lambda^{-1};q)}{2(q;q)_{\infty}^{6}
  \theta(\lambda;q)\theta(q^{-1/2}\lambda;q)q^{-n(2n+1)/2}\lambda^{-2n}}
f^{(-1)}(\pm q^{-1/4-n}\lambda^{-1},q)\\
&=
-\frac{\theta(\pm q^{-1/4}\lambda^{-1};q)\theta(\pm q^{-1/4};q)
  \theta(\pm q^{-3/4};q)\theta(\pm q^{-5/4}\lambda^{-1};q)}{2(q;q)_{\infty}^{6}
  \theta(\lambda;q)\theta(q^{-1/2}\lambda;q)}
f^{(-1)}(\pm q^{-1/4-n}\lambda^{-1},q)\\
&=
\frac{\theta(\pm q^{3/4}\lambda^{-1};q)\theta(\pm q^{-1/4};q)\theta(\pm q^{-3/4};q)
  \theta(\pm q^{-3/4}\lambda;q)}{2(q;q)_{\infty}^{6}
  \theta(q^{-1}\lambda;q)\theta(q^{-3/2}\lambda;q)}
f^{(-1)}(\pm q^{-1/4-n} \lambda^{-1},q) \,.
\end{aligned}
\end{small}
\ee
\end{proof}

\subsection{Determinants}

The functional equation \eqref{41x1} for $U,V$ defined in~\eqref{UVex41}
can be written in the form 
\be
U(qt,\lambda,q)
=
A(t;q)U(t,\lambda,q)
\quad\text{and}\quad
V(qt,q)
=
A(t;q)V(t,q)
\ee
where
\be\label{comp.mat.41.x1}
A(t,q)
=
\begin{pmatrix}
0 & 1\\
-q^{-2} & 2q^{-1}-q^{-2}t^{-1}
\end{pmatrix} \,.
\ee
Therefore, we see that
\be
\det(U(qt,\lambda,q))=\det(U(t,\lambda,q))q^{-2}
\quad\text{and}\quad
\det(V(qt,q))=\det(V(t,q))q^{-2}.
\ee

\begin{lemma}
\label{lem.41det}
  We have
\be
-\det(U(t,\lambda,q)) = \det(V(t,q)) = q^{-1}t^{-2} \,.
\ee
  \end{lemma}
  
\begin{proof}
Firstly notice that both $\det(U(t,\lambda,q))t^2$ and $\det(V(t,q))t^2$
are elliptic functions in $t$. Furthermore,
\be
\det(U(t,\lambda,q))
=
f^{(1)}(t,\lambda,q)f^{(-1)}(qt,q)-f^{(1)}(qt,\lambda,q)f^{(-1)}(t,q)
\ee
has potentially simple poles in $t$ at $\pm q^{-1/4-n/2}\lambda^{-1}$.
Lemma~\ref{lem.resf1} implies that
%\be
\begin{align*}
\Res_{t=\pm q^{-1/4-n}\lambda^{-1}}\det(U(t,\lambda,q)) =
&\frac{\theta(\pm q^{3/4}\lambda^{-1};q)\theta(\pm q^{-1/4};q)
  \theta(\pm q^{-3/4};q)\theta(\pm q^{-3/4}\lambda;q)}{2(q;q)_{\infty}^{6}
  \theta(q^{-1}\lambda;q)\theta(q^{-3/2}\lambda;q)} \\
& \times
f^{(-1)}(\pm q^{-1/4-n}
  \lambda^{-1},q)f^{(-1)}(\pm q^{3/4-n}\lambda^{-1},q)\\
&-\frac{\theta(\pm q^{3/4}\lambda^{-1};q)\theta(\pm q^{-1/4};q)
  \theta(\pm q^{-3/4};q)\theta(\pm q^{-3/4}\lambda;q)
}{2(q;q)_{\infty}^{6}
  \theta(q^{-1}\lambda;q)\theta(q^{-3/2}\lambda;q)}\\
& \times
f^{(-1)}(\pm q^{3/4-n}\lambda^{-1},q)f^{(-1)}(\pm q^{-1/4-n}\lambda^{-1},q) \\
  =&\; 0 \,.
\end{align*}
%\ee
A similar calculation shows that
%\be
$\Res_{t=\pm q^{1/4-n}\lambda^{-1}}\det(U(t,\lambda,q))=0$.
%\ee
Therefore, $U(t,\lambda;q)t^2$ is elliptic and holomorphic in
$t\in\BC^{\times}$, hence it is constant in $t$. Now considering the limit as
$t\rightarrow0$, using the definition of $f^{(\pm1)}$ and their asymptotic
expansions (given by their formal power series expansions, by a $q$-version of
Watson's lemma), it follows that
\be
\begin{aligned}
\det(U(t,\lambda,q))t^2
&=
\lim_{t\rightarrow 0}\det(U(t,\lambda,q))t^2
=
\lim_{t\rightarrow 0}\frac{\theta(t;q)}{\theta(t;q)}t^2
-\lim_{t\rightarrow 0}\frac{\theta(qt;q)}{\theta(q^{-1}t;q)}t^2\\
&=
-\lim_{t\rightarrow 0}q^{-1}\frac{\theta(t;q)}{\theta(t;q)}
=
-q^{-1} \,.
\end{aligned}
\ee

For $V$ notice that
\be
\det(V(t,q))t^{2}
=
g^{(0,0)}(t,q)g^{(0,1)}(qt,q)t^{2}-g^{(0,0)}(qt,q)g^{(0,1)}(t,q)t^{2}.
\ee
This has potentially simple poles at $t\in q^{\BZ}$. There is no non-constant
elliptic function that satisfies this~\cite{Akhiezer}.
Therefore, $\det(V(t,q))t^{2}$ is constant. Then notice that
\be
\det(V(t,q))t^{2}
=
\lim_{t\rightarrow\infty}\det(V(t,q))t^{2}
=
q^{-1}.
\ee
\end{proof}

\begin{corollary}
  \label{cor:41holsol}
  If $h(t,q)$ is holomorphic for $t\in\BC^{\times}$
  and satisfies the $q$-difference equation~\eqref{41x1} then
\be
h(t,q)=C_{0}(q)g^{(0,0)}(t,q) \,.
\ee
\end{corollary}

\begin{proof}
Every solution to equation~\eqref{41x1} can be written in the form
\be
h(t,q)=C_{0}(t,q)g^{(0,0)}(t,q)+C_{1}(t,q)g^{(0,1)}(t,q)
\ee
for some elliptic functions $C_0$ and $C_1$. Now we see that
\be
\det
\begin{pmatrix}
g^{(0,0)}(t,q) & h(t,q)\\
g^{(0,0)}(qt,q) & h(qt,q)\\
\end{pmatrix}
=
C_{1}(t,q)q^{-1}t^{-2}
\ee
is holomorphic and therefore $C_{1}(t,q)=C_{1}(q)$ is independent of $t$. Now
notice that if $g(t_{0},q)=0$ then $g(qt_{0},q)\neq0$ as otherwise this would imply
$\det(V(t_{0},q))=0$. This means that if $C_{1}(q)\neq0$ then $C_{0}(t,q)$ has
simple poles at $t\in q^{\BZ}$ and no other poles for $t\in\BC^{\times}$ which
contradicts the fact that $C_{0}(t,q)$ is elliptic as this would give an isomorphism
from the elliptic curve to $\mathbb{CP}^{1}$. Therefore
$C_{1}(q)=0$. Then again noting that if $g(t_{0},q)=0$ then $g(qt_{0},q)\neq0$
we see that $C_{0}(t,q)$ must be constant in $t$.

We remark that an alternative proof can be given using the fact that all holomorphic
solutions on $\BC^{\times}$ have a convergent Laurent/Fourier series expansion
\be
\sum_{k\in\BZ}\alpha_{k}(q)t^{k} \,.
\ee
This is a consequence of Cauchy's theorem and a detailed discussion of this
can be found, for example, in~\cite{Ahlfors}. In our case, the functional equations
imply this expansion is divergent if $\alpha_{-1}\neq0$. This forces
$\alpha_{-1}=0$, and the functional equation then implies that $\alpha_{k}=0$
  for $k<0$ and that $\alpha_{k}$ are
uniquely determined by $C_{0}(q)=\alpha_{0}(q)$ for $k \geq 0$ which implies
that our function is equal to $C_{0}(q)g^{(0,0)}(t,q)$.
\end{proof}

\subsection{State integral}
\label{sub.state}

In this section we discuss the second part of Theorem~\ref{thm.41} concerning
the extension of the cocycle to the cut plane. The main idea is to use
descendant state integrals, following~\cite[Eqn.(41)]{GGM}, defined by
% \be
% \mathcal{I}_{A,B}(z,\tau)
% =
% \frac{i}{\sqrt{\tau}}\int_{-i\sqrt{\tau}(\BR+i\ve)}
% \frac{(-q^{1/2}\e(x);q)_{\infty}^{B}}{(-\tq^{1/2}\e(x/\tau);\tq)_{\infty}^{B}}
% \e\left(\frac{A}{2\tau}x^{2}-\frac{zx}{\tau}\right)dx \,.
% \ee
\be
\mathcal{I}_{A,B}(z,\tau)
=
\int_{\BR+i\ve}
\Phi_{\sfb}(x)^B
\exp\left(-A\pi i x^{2}-2\pi\frac{zx}{\sfb}\right)dx.
\ee
Using the holomorphic extension of the quantum dilogarithm $\Phi$ from
Theorem~\ref{thm.ex1}, this integral can be shown to extend to a holomorphic
function in for $(z,\tau)\in\BC\times\BC'$. We are interested in the integrals
$\mathcal{I}_{1,2}(z,\tau)$ which, using the residue theorem, factorise as an
elementary function holomorphic for $\tau\in\BC'$ times the following bilinear
combination, see~\cite[Thm.3]{GGM},
\be
\label{gs.fac.int}
\begin{aligned}
\calI(z,\tau)
&=\tau^{1/2} g^{(0,0)}(\tilde{t},\tq)g^{(0,1)}(t,q)
-\tau^{-1/2}g^{(0,1)}(\tilde{t},\tq)g^{(0,0)}(t,q)\\
&\quad+\tau^{-1/2}\left(\tau\frac{\theta'(t^{-1};q)}{\theta(t^{-1};q)}
  -\frac{\theta'(\tilde{t}^{-1};\tq)}{\theta(\tilde{t}^{-1};\tq)}
  -\frac{1}{2}+\frac{\tau}{2}-z\right)g^{(0,0)}(\tilde{t},\tq)g^{(0,0)}(t,q)\\
&=\tau^{1/2} g^{(0,0)}(\tilde{t},\tq)g^{(0,1)}(t,q)
-\tau^{-1/2} g^{(0,1)}(\tilde{t},\tq)g^{(0,0)}(t,q) \,.
\end{aligned}
\ee
where the equality follows from Equation~\eqref{S.mod.thd}. Therefore,
$\calI(z,\tau)$ extends to a holomorphic function for $(z,\tau)\in\BC\times\BC'$.
Now, we see that 
{\small
\be
\begin{aligned}
\Om_{V,S}(z,\tau)
&=
(V|_{\kappa}S)(z,\tau)V(z,\tau)^{-1}\\
&=
qt^2
\begin{pmatrix}
  g^{(0,0)}(\tilde{t},\tq) & g^{(0,1)}(\tilde{t},\tq)\\
  g^{(0,0)}(\tq\tilde{t},\tq) & g^{(0,1)}(\tq\tilde{t},\tq)\\
\end{pmatrix}
\begin{pmatrix}
  \tau^{1/2} & 0\\
  0 & \tau^{-1/2}
\end{pmatrix}
\begin{pmatrix}
  g^{(0,1)}(qt,q) & -g^{(0,1)}(t,q)\\
  -g^{(0,0)}(qt,q) & g^{(0,0)}(t,q)
\end{pmatrix}
.
\end{aligned}
\ee
}
The entries of $\Om_{V,S}$ are then elementary functions holomorphic for
$\tau\in\BC'$ times $\calI(z+n+m\tau,\tau)$ for some $n,m\in\{-1,0,1\}$.
This shows that the cocycle $\Om_{V,S}$ extends to a holomorphic function for
$(z,\tau)\in\BC\times\BC'$; see~\cite[Thm.14]{GGM:peacock} following the proof
of~\cite[Thm.1.1]{GK:qseries}. Now noting that
\be
\Av(\vM) = \Delta_{\kappa,\gamma} (\Av(\vM)|_{\kappa} \gamma)
\ee
and $\Om_{\Av(U),T}=I$, part (c) of Theorem~\ref{thm.ST} proves part \rm{(b)}
of Theorem~\ref{thm.41}.

\subsection{The inhomogeneous equation}
\label{sec.4.inhom.x1}

We now consider Equation~\eqref{41x1inhom}.
We can convert this into a homogeneous equation of one degree higher
\be
tf(t;q)+(1-3qt)f(qt;q)+(3tq^2-1)f(q^2t;q)-tq^3f(q^3t;q)=0
\ee
with Newton polygon shown in Figure~\ref{f.413}.

\begin{figure}[!htpb]
\begin{center}
\begin{tikzpicture}[scale=0.8,baseline=-3]
\draw[<->,thick] (-1.5,0) -- (7.5,0);
\draw[<->,thick] (0,-1.5) -- (0,3.5);
\foreach \x in {2,4,6}
\draw[thick] (\x,-2pt) -- (\x,2pt);
\filldraw (0,2) circle (2pt);
\filldraw (2,0) circle (2pt);
\filldraw (2,2) circle (2pt);
\filldraw (4,0) circle (2pt);
\filldraw (4,2) circle (2pt);
\filldraw (6,2) circle (2pt);
\node at (1,2.3) {$g^{(0,0)}$};
\node at (3,2.3) {$g^{(0,1)}$};
\node at (5,2.3) {$g^{(0,2)}$};
\node at (0.5,0.8) {$f^{(-1)}$};
\node at (5.5,0.8) {$f^{(1)}$};
\node at (3,-0.3) {$f^{(0)}$};
\draw (0,2) -- (2,0) -- (4,0) -- (6,2) -- cycle;
\fill[blue,opacity=0.2] (0,2) -- (2,0) -- (4,0) -- (6,2) -- cycle;
\end{tikzpicture}
\caption{The Newton polygon of Equation~\eqref{41x1inhom}.}
\label{f.413}
\end{center}
\end{figure}
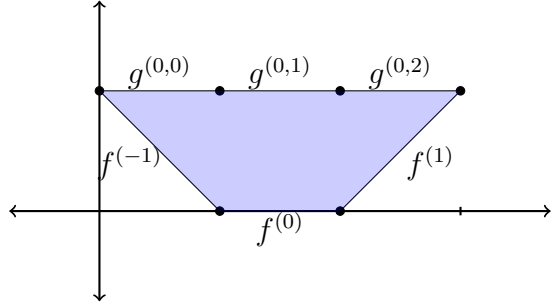

\noindent
Now from the general theory the solutions $f^{(\pm1)}$ and $g^{(0,0)}$ and
$g^{(0,1)}$ are the same as the solutions in Section~\ref{subsec.41sols}. We do
find two additional
solutions which are in fact solutions to the inhomogeneous equation, namely
\be
\hat{f}^{(0)}(t,q)=\sum_{k=0}^{\infty}(-1)^{k}q^{-k(k+1)/2}(q;q)_{k}^{2}t^{k} \,,
\ee
a divergent formal power series solution at $t=0$ and
\be
\label{g2}
{\small
\begin{aligned}
g^{(0,2)}(t;q)
&=
\sum_{k=0}^{\infty}\left(\frac{1}{2}\Bigg(\frac{1}{2}E_{1}(q)-\frac{1}{2}+\sum_{j=1}^{k}\frac{1+q^{j}}{1-q^{j}}\right)^{2}
  -\left(\frac{1}{2}E_{1}(q)-\frac{1}{2}+\sum_{j=1}^{k}\frac{1+q^{j}}{1-q^{j}}\right)\frac{\theta'(t^{-1};q)}{\theta(t^{-1};q)}\\
  &\qquad\qquad+\frac{1}{2}\frac{\theta''(t^{-1};q)}{\theta(t^{-1};q)}
  +\sum_{j=1}^{k}\frac{q^{j}}{(1-q^{j})^2}-\frac{1}{24}-\frac{1}{24}E_{2}(q)\Bigg)
(-1)^{k}\frac{q^{k(k+1)/2}}{(q;q)_{k}^{2}}t^{-1-k}\,,
\end{aligned}
}
\ee
where $E_{2}(q)=1-24\sum_{k=0}^{\infty}q^{k}/(1-q^{k})^2$,
a convergent solution at $t=\infty$. So, we must resum $\hat{f}^{(0)}(t,q)$. For $|\xi|<1$, we have
\be
\Bor_{1}\hat{f}^{(0)}(\xi,q)
=
\sum_{k=0}^{\infty}(q;q)_{k}^{2}\xi^{k}
\ee
which can be analytically continued away from $\xi\in q^{\BZ_{\leq0}}$ using the
relation
\be
(1-\xi)\Bor_{1}\hat{f}^{(0)}(\xi,q)+2q\xi \Bor_{1}\hat{f}^{(0)}(q\xi,q)
-q^2\xi \Bor_{1}\hat{f}^{(0)}(q^{2}\xi,q) = 1.
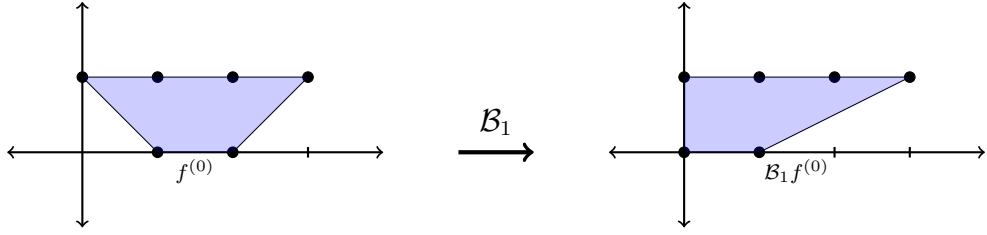
\ee
\begin{center}
\begin{tikzpicture}
\draw[<->,thick] (-1,0) -- (4,0);
\draw[<->,thick] (0,-1) -- (0,2);
\foreach \x in {1,2,3}
\draw[thick] (\x,-2pt) -- (\x,2pt);
\filldraw (0,1) circle (2pt);
\filldraw (1,0) circle (2pt);
\filldraw (1,1) circle (2pt);
\filldraw (2,0) circle (2pt);
\filldraw (2,1) circle (2pt);
\filldraw (3,1) circle (2pt);
\draw (0,1) -- (1,0) -- (2,0) -- (3,1) -- cycle;
\fill[blue,opacity=0.2] (0,1) -- (1,0) -- (2,0) -- (3,1) -- cycle;
\node at (1.5,-0.25) {\begin{tiny}$f^{(0)}$\end{tiny}};
\draw[->,ultra thick] (3+2,0) -- (4+2,0);
\node at (3.5+2,0.4) {$\Bor_{1}$};
\draw[<->,thick] (-1+8,0) -- (4+8,0);
\draw[<->,thick] (0+8,-1) -- (0+8,2);
\foreach \x in {1+8,2+8,3+8}
\draw[thick] (\x,-2pt) -- (\x,2pt);
\filldraw (0+8,1) circle (2pt);
\filldraw (0+8,0) circle (2pt);
\filldraw (1+8,1) circle (2pt);
\filldraw (1+8,0) circle (2pt);
\filldraw (2+8,1) circle (2pt);
\filldraw (3+8,1) circle (2pt);
\draw (0+8,1) -- (0+8,0) -- (1+8,0) -- (3+8,1) -- cycle;
\fill[blue,opacity=0.2] (0+8,1) -- (0+8,0) -- (1+8,0) -- (3+8,1) -- cycle;
\node at (1.5+8,-0.25) {\begin{tiny}$\Bor_{1}f^{(0)}$\end{tiny}};;
\end{tikzpicture}
\end{center}
Therefore, we define
\be\label{f0x1}
f^{(0)}(t,\lambda_{2},q)
=
(\Lap_{1}\Bor_{1}\hat{f}^{(0)})(t,\lambda_{2},q)
=
\frac{1}{\theta(\lambda_{2};q)}\sum_{k\in\BZ}(-1)^{k}q^{k(k+1)}
\lambda_{2}^{k}\Bor_{1}\hat{f}^{(0)}(q^{k}\lambda_{2}t,q).
\ee
Now we consider the matrices
\be
\label{U413}
\begin{aligned}
&U(t,\lambda_{1},\lambda_{2},q)
=
W(f^{(1)}(t,\lambda_{1},q),f^{(-1)}(t,q),f^{(0)}(t,\lambda_{2},q))\\
&=
\begin{pmatrix}
0 & 1 & 0\\
0 & 0 & 1\\
q^{-2}t^{-1} & -q^{-2} & q^{-2}t^{-1}-2q^{-1}
\end{pmatrix}
\begin{pmatrix}
0 & 0 & 1\\
f^{(1)}(t,\lambda_{1},q) & f^{(-1)}(t,q) & f^{(0)}(t,\lambda_{2},q)\\
f^{(1)}(qt,\lambda_{1},q) & f^{(-1)}(qt,\lambda_{1},q) & f^{(0)}(qt,\lambda_{2},q)
\end{pmatrix}
\end{aligned}
\ee
and
\be
\label{V413}
\begin{aligned}
V(t,q)
&=
W(g^{(0,0)}(t,q),g^{(0,1)}(t,q),g^{(0,2)}(t,q))\\
&=
\begin{pmatrix}
0 & 1 & 0\\
0 & 0 & 1\\
q^{-2}t^{-1} & -q^{-2} & q^{-2}t^{-1}-2q^{-1}
\end{pmatrix}
\begin{pmatrix}
0 & 0 & 1\\
g^{(0,0)}(t,q) & g^{(0,1)}(t,q) & g^{(0,2)}(t,q)\\
g^{(0,0)}(qt,q) & g^{(0,1)}(qt,q) & g^{(0,2)}(qt,q)\\
\end{pmatrix}.
\end{aligned}
\ee
With these formulae we will now go on to prove Theorem~\ref{thm.41b}.

\begin{proof}[Proof of Theorem~\ref{thm.41b}]
  The identification of the first two columns of $\vM(t,\lambda_{1},\lambda_{2},q)$
  with those in Equation~\eqref{M413} follows from Theorem~\ref{thm.41}.
  Equations~\eqref{U413} and~\eqref{V413}, together with Lemma~\ref{lem.41det}
  imply that 
\be
-\det(U(t,\lambda_{1},\lambda_{2},q))
=
\det(V(t,q))
=
q^{-3}t^{-3} \,.
\ee
Hence, $\det(\vM(t,\lambda_{1},\lambda_{2},q))=-1$ which in turn
implies that $\vM_{3,3}(t,\lambda_{2},q)=1$.

Now, Equation~\eqref{M413}, written in the form $U(t,\lambda_{1},\lambda_{2},q)=
V(t,q) \vM(t,\lambda_{1},\lambda_{2},q)$, together with Equations~\eqref{U413}
and~\eqref{V413} and the fact that $\vM_{3,3}=1$ imply that
\be
\label{fgs0}
f^{(0)} = \vM_{1,3} g^{(0,0)} + \vM_{2,3} g^{(0,1)} + g^{(0,2)} \,.
\ee
Thus, to determine the two remaining entries of the monodromy matrix, we need to
show that the function
\be
\label{fgs}
\calE_3 :=f^{(0)} -(m_{1,3} g^{(0,0)} + m_{2,3} g^{(0,1)} + g^{(0,2)}) 
\ee
vanishes identically, where 
\be
\label{m1323}
m_{1,3}(t,q) = \wp(t,q), \qquad
m_{2,3}(t,\lambda_{2},q)
= \frac{1}{2}\frac{\wp'(t,q)-\wp'(\lambda_{2},q)}{\wp(t,q)-\wp(\lambda_{2},q)}
\ee
denote the two entries of the matrix on the right hand-side of Equation~\eqref{M413}.
Note that $\calE_3(t,\lambda_{2},q)$ is a meromorphic function of $t$
with potential simple poles at $\lambda_2^{-1} q^\BZ$ (coming from $f^{(0)}$ and
$m_{2,3}$) and potential double poles at $q^\BZ$ (coming from
$g^{(0,1)}$, $g^{(0,2)}$, $m_{1,3}$ and $m_{2,3}$). 
Since
$\Res_{t=q^{-m}\lambda_{2}^{-1}} m_{2,3}(t,\lambda_{2},q) \frac{dt}{2\pi it} = 1 \,,
$
combined with equation~\eqref{fg} gives
\be
\Res_{t=q^{-m}\lambda_{2}^{-1}}m_{2,3}(t,\lambda_{2},q)
g^{(0,1)}(t,q) \frac{dt}{2\pi it} =  f^{(-1)}(q^{-m}\lambda_{2}^{-1},q) \,,
\ee
we see from Lemma~\ref{lem.resf0} that $\calE_3$ has no poles at
$\lambda_{2}^{-1} q^\BZ$. Now noting that
\be
m_{2,3}(t,\lambda_{2},q)
=
\frac{\theta'(t^{-1};q)}{\theta'(t^{-1};q)}
+\frac{\theta'(\lambda^{-1};q)}{\theta'(\lambda^{-1};q)}
-\frac{\theta'(\lambda^{-1}t^{-1};q)}{\theta'(\lambda^{-1}t^{-1};q)}+\frac{1}{2}\,,
\ee
the only terms that contribute to the polar part of
\be
\label{gsoff0}
  g^{(0,2)}(t,q)+m_{2,3}(t, \lambda_{2},q)
  g^{(0,1)}(t,q)+m_{1,3}(t,q) g^{(0,0)}(t,q)
\ee
at $t=q^{m}\ve$ for $\ve\sim 1$ are
\be
\begin{aligned}
&\sum_{k=0}^{\infty}\bigg(
-\left(k-2E^{(k)}_{1}(q)\right)\frac{\theta'(\ve^{-1};q)}{\theta(\ve^{-1};q)}
+\frac{1}{2}\frac{\theta''(q^{-m}\ve^{-1};q)}{\theta(q^{-m}\ve^{-1};q)}\\
&\qquad+\left(k-m-\frac{\theta'(\ve^{-1};q)}{\theta(\ve^{-1};q)}-2E^{(k)}_{1}(q)
\right)\frac{\theta'(\ve^{-1};q)}{\theta(\ve^{-1};q)}
-\frac{1}{2}\frac{\theta'(\ve^{-1};q)}{\theta(\ve^{-1};q)}\\
&\qquad+\frac{\theta'(\ve^{-1};q)^2}{\theta(\ve^{-1};q)^2}
-\frac{\theta''(\ve^{-1};q)}{\theta(\ve^{-1};q)}\bigg)(-1)^{k}
\frac{q^{k(k+1)/2-m-mk}}{(q;q)_{k}^{2}}\ve^{-1-k}\\
&=\sum_{k=0}^{\infty}\bigg(
\frac{1}{2}\frac{\theta''(q^{-m}\ve^{-1};q)}{\theta(q^{-m}\ve^{-1};q)}
-\left(m+\frac{1}{2}\right)\frac{\theta'(\ve^{-1};q)}{\theta(\ve^{-1};q)}
-\frac{\theta''(\ve^{-1};q)}{\theta(\ve^{-1};q)}\bigg)
(-1)^{k}\frac{q^{k(k+1)/2}}{(q;q)_{k}^{2}}\ve^{-1-k}.
\end{aligned}
\ee
This has a potentially simple pole at $\ve=1$ however noting that
\be
\begin{aligned}
  \Res_{\ve=1}\frac{\theta''(q^{-m}\ve^{-1};q)}{\theta(q^{-m}\ve^{-1};q)}
  \frac{d\ve}{2\pi i\ve}
&=
-2m+1\\
\Res_{\ve=1}\frac{\theta'(\ve^{-1};q)}{\theta(\ve^{-1};q)}\frac{d\ve}{2\pi i\ve}
&=
-1
\end{aligned}
\ee
we see that the residue vanishes and therefore \eqref{gsoff0} has a removable
singularity at all points $t\in q^{\BZ}$. Therefore, we see that $\calE_3$ is
a holomorphic (for $t \in \BC^\times$) solution of Equation~\eqref{41x1} and
therefore from Corollary~\ref{cor:41holsol} we see that
$\calE_3(t,\lambda_{2},q) = C(q)g^{(0,0)}(t,q)$. 
To finish the proof we note that for $t$ on some compact set where the functions
are holomorphic, Lemma~\ref{lem.fglimits} and Lemma~\ref{lem.fglimits.inhom}
implies that
\begin{align*}
-\frac{C(q)}{(q;q)_{\infty}^{2}} &=
\lim_{r\rightarrow\infty}\frac{C(q) g^{(0,0)}(t^r,q)}{\theta(q^{r}t;q)} \\  
&=\lim_{r\rightarrow\infty}\frac{
f^{(0)}(t,\lambda_{2},q)-(g^{(0,2)}(t,q) + m_{2,3}(t,\lambda_{2},q)
g^{(0,1)}(t,q) + m_{1,3}(t,q)  g^{(0,0)}(t,q))}{\theta(q^{r}t;q)} \\ &=
\frac{1}{(q;q)_{\infty}^{2}}(m_{1,3}(t^{-1},q)-m_{1,3}(t,q))= 0 \,,
\end{align*}
where we note that $m_{1,3}(t,q)$ is the Weierstrass $\wp$-function plus a
constant which is even in $z$ where $t=\e(z)$.

Our next task is to compute the residues of $f^{(0)}$. 

\begin{lemma}
\label{lem.resf0}
\rm{(a)}
The function $\Bor_{1}\hat{f}^{(0)}(\xi,q)$ has simple poles at $\xi \in q^\BZ$
with residue
\be
\label{resBf0}
R(m,q) :=
\Res_{\xi=q^{-m}}\Bor_{1}\hat{f}^{(0)}(\xi,q)\frac{d\xi}{2\pi i\xi}
= -(q;q)_{\infty}^{2}\sum_{\ell=0}^{m}\frac{q^{m}}{(q;q)_{\ell}(q;q)_{m-\ell}} \,.
\ee
\rm{(b)}
The function $f^{(0)}(t,q)$ has simple poles at $t\in \lambda_{2}^{-1}q^{\BZ}$
with residue
\be
\label{resf0}
\Res_{t=q^{-m}\lambda_{2}^{-1}}f^{(0)}(t,\lambda_{2},q)\frac{dt}{2\pi it} =
f^{(-1)}(q^{-m}\lambda_{2}^{-1},q) \,.
\ee
\end{lemma}

\begin{proof}
For part (a), notice that
\be
\begin{aligned}
\Bor_{1}\hat{f}^{(0)}(\xi,q)
&=
\sum_{k=0}^{\infty}(q;q)_{k}^{2}\xi^{k}
=
(q;q)_{\infty}^{2}\sum_{k=0}^{\infty}\frac{1}{(q^{k+1};q)_{\infty}}\xi^{k}\\
&=
(q;q)_{\infty}^{2}\sum_{k,\ell,n=0}^{\infty}
\frac{q^{kn+k\ell+n+\ell}}{(q;q)_{\ell}(q;q)_{n}}\xi^{k}
=
(q;q)_{\infty}^{2}\sum_{\ell,n=0}^{\infty}\frac{q^{n+\ell}}{
  (q;q)_{\ell}(q;q)_{n}(1-q^{n+\ell}\xi)}
\end{aligned}
\ee
is the analytic continuation of $\Bor_{1}\hat{f}^{(0)}$. From this, it clearly
follows that the only poles are simple, located at $\xi=q^\BZ$. (Alternatively,
one can also use the linear $q$-difference equation satisfied by this function,
though this is not needed here). We compute the residue as follows:
\be
\begin{aligned}
R(m,q)
&=
\Res_{\xi=q^{-m}}(q;q)_{\infty}^{2}\sum_{\ell,n=0}^{\infty}
\frac{q^{n+\ell}}{(q;q)_{\ell}(q;q)_{n}(1-q^{n+\ell}\xi)}\frac{d\xi}{2\pi i\xi}\\
&=
-(q;q)_{\infty}^{2}\sum_{\ell=0}^{m}\frac{q^{m}}{(q;q)_{\ell}(q;q)_{m-\ell}}.
\end{aligned}
\ee

For part (b), the pole structure is clear noting that Equation~\eqref{f0x1}
is convergent if all the terms are and that
$\Bor_{1}\hat{f}^{(0)}(q^{k}t\lambda_{2},q)$ has poles at
$t\in\lambda_{2}^{-1}q^{\BZ_{\leq -k}}$. We compute the residues as follows:
\be
\begin{aligned}
\Res_{t=q^{-m}\lambda_{2}^{-1}}&\,f^{(0)}(t,\lambda_{2},q)\frac{dt}{2\pi it}
=
\Res_{t=q^{-m}\lambda_{2}^{-1}}(\Lap_{1}\Bor_{1}\hat{f}^{(0)})(t,\lambda_{2},q)
\frac{dt}{2\pi it}\\
&=
\Res_{t=q^{-m}\lambda_{2}^{-1}}\frac{1}{\theta(\lambda_{2},q)}
\sum_{k\in\BZ}(-1)^{k}q^{k(k+1)/2}\lambda_{2}^{k}\Bor_{1}\hat{f}^{(0)}(q^{k}
\lambda_{2}t,q)\frac{dt}{2\pi it}\\
&=
\frac{1}{\theta(\lambda_{2},q)}\sum_{k\in\BZ}(-1)^{k}q^{k(k+1)/2}
\lambda_{2}^{k}R(m-k,q)\\
% &=
% \frac{1}{\theta(\lambda_{2},q)}
% \sum_{k=0}^{\infty}(-1)^{m-k}q^{(m-k)(m-k+1)/2}\lambda_{2}^{m-k}R(k,q)\\
&=
\frac{1}{\theta(\lambda_{2},q)}(-1)^{m}q^{m(m+1)/2}\lambda_{2}^{m}
\sum_{k=0}^{\infty}(-1)^{k}q^{k(k-1)/2}(q^{-m}\lambda_{2}^{-1})^{k}R(k,q)\\
% &=
% \frac{-(q;q)_{\infty}^{2}}{\theta(\lambda_{2},q)}(-1)^{m}q^{m(m+1)/2}
% \lambda_{2}^{m}\sum_{k=0}^{\infty}(-1)^{k}q^{k(k-1)/2}
% (q^{-m}\lambda_{2}^{-1})^{k}\sum_{\ell=0}^{\infty}
% \frac{q^{k}}{(q;q)_{\ell}(q;q)_{k-\ell}}\\
&=
\frac{-(q;q)_{\infty}^{2}}{\theta(\lambda_{2},q)}(-1)^{m}q^{m(m+1)/2}
\lambda_{2}^{m}\sum_{k,\ell=0}^{\infty}(-1)^{k}
\frac{q^{k(k+1)/2}}{(q;q)_{\ell}(q;q)_{k-\ell}}(q^{-m}\lambda_{2}^{-1})^{k}\\
% &=
% \frac{1}{\theta(q^{m}\lambda_{2},q)}\hat{f}^{(-1)}(q^{-m}\lambda_{2}^{-1},q)\\
&=
f^{(-1)}(q^{-m}\lambda_{2}^{-1},q) \,.
\end{aligned}
\ee
\end{proof}
\begin{lemma}
\label{lem.fglimits.inhom}
  For $t$ in some compact set where the functions are holomorphic, we have
\be
\begin{aligned}
\lim_{r\rightarrow\infty}\frac{g^{(0,2)}(q^{r}t,q)}{\theta(q^{r}t;q)}
&=
\frac{1}{(q;q)_{\infty}^2} m_{1,3}(t^{-1},q)
&\qquad
\lim_{r\rightarrow\infty}\frac{f^{(0)}(q^{r}t,q)}{\theta(q^{r}t;q)}
&=
0 \,.
  \end{aligned}
\ee
\end{lemma}

\begin{proof}
We have
\be
\begin{aligned}
&(-1)^{r}q^{r(r+1)/2}t^{r+1}g^{(0,2)}(q^{r}t,q)\\
&=
\sum_{k=-r}^{\infty}
\Bigg(\frac{1}{2}\left(\frac{1}{2}E_{1}(q)-r-\frac{1}{2}-\frac{\theta'(t^{-1};q)}{\theta(t^{-1};q)}+\sum_{j=1}^{k+r}\frac{1+q^{j}}{1-q^{j}}\right)^{2}\\
&\qquad\qquad-\left(\frac{1}{2}E_{1}(q)-r-\frac{1}{2}-\frac{\theta'(t^{-1};q)}{\theta(t^{-1};q)}+\sum_{j=1}^{k+r}\frac{1+q^{j}}{1-q^{j}}\right)
  \frac{\theta'(t^{-1};q)}{\theta(t^{-1};q)}\\
  &\qquad\qquad+\frac{1}{2}
  \frac{\theta''(t^{-1};q)}{\theta(t^{-1};q)}+\sum_{j=1}^{k+r}\frac{q^{j}}{(1-q^{j})^2}-\frac{1}{24}-\frac{1}{24}E_{2}(q)\Bigg)(-1)^{k}
\frac{q^{k(k+1)/2}}{(q;q)_{k+r}^{2}}t^{-k} \,.
\end{aligned}
\ee
Then noting that $1/24-E_2(q)/24-\sum_{j=1}^{k}q^j/(1-q^j)^2)=O(q^{k+1})$ the proof then proceeds as in Lemma~\ref{lem.fglimits}.
\end{proof}

To finish the proof of
Theorem~\ref{thm.41b} we will use the state integrals introduced in
~\cite[Eqn.(73)]{GGMW:trivial}.
% \be
% \label{41inhom.sint}
% \frac{i}{\sqrt{\tau}}\int_{-i\sqrt{\tau}(\BR+i\ve)}
% \frac{(-q^{1/2}\e(x);q)_{\infty}^{2}}{(-\tq^{1/2}\e(x/\tau);\tq)_{\infty}^{2}}
% %\Phi_{\sqrt{\tau}}\left(\frac{ix}{\sqrt{\tau}}\right)^2
% \frac{\e\left(\frac{1}{2\tau}x^{2}-\frac{zx}{\tau}\right)}{1-\e(x-1/2-\tau/2)}dx
% \,.
% \ee
\be\label{41inhom.sint}
\int_{\BR+i\ve}
\Phi_{\sfb}(x)^2
\frac{\exp\left(-\pi i x^{2}-2\pi\frac{zx}{\sfb}\right)}{1+\tq^{1/2}
  \exp\left(-\frac{2\pi x}{\sfb}\right)}dx\,.
\ee
The factorisation of this integral was done in \cite[Thm.7]{GGMW:trivial}. The
fact this module is not self dual (see Proposition~\ref{41inhom.not.sd}) means
that additional functions arise in the factorisation. It was shown, using
Equation~\eqref{S.mod.thd}, that Equation~\eqref{41inhom.sint} factors as an
elementary function holomorphic in $\BC'$ times
\be
\begin{aligned}
\calI(z,\tau)
=&
\tau^{2}g^{(0,2)}(\tilde{t},\tq)
+\tau g^{(0,1)}(\tilde{t},\tq)L_{0}(t,q)-g^{(0,0)}(\tilde{t},\tq)L_{1}(t,q)
\end{aligned}
\ee
where
\be
\label{41.stateint.fac.inhom}
\begin{aligned}
L_{0}(t,q)
&=
1-\frac{1}{2}E_{1}(q)+\frac{\theta'(t^{-1};q)}{\theta(t^{-1};q)}
+\sum_{k=1}^{\infty}(-1)^{k}\frac{q^{k(k+1)/2}}{(q;q)_{k}^{2}(1-q^{k})}t^{-k}\\
L_{1}(t,q)
&=
-\frac{5}{12}+\frac{1}{2}E_{1}-E_{1}(q)^{2}-\frac{1}{24}E_{2}(q)
-\left(\frac{1}{8}-\frac{1}{8}E_{1}(q)\right)\frac{\theta'(t^{-1};q)}{\theta(t^{-1};q)}
-\frac{\theta''(t^{-1};q)}{2\theta(t^{-1};q)}\\
&\quad+\sum_{k=1}^{\infty}(-1)^{k}
\frac{q^{k(k+1)/2}}{(q;q)_{k}^{2}(1-q^{k})}t^{-k}
\left(\frac{1}{2}E_{1}(q)-\frac{1}{2}-\frac{\theta'(t^{-1};q)}{\theta(t^{-1};q)}+\sum_{j=1}^{k}\frac{1+q^{j}}{1-q^{j}}
  +\frac{q^{k}}{1-q^{k}}\right)\,.
\end{aligned}
\ee
These functions $L_{i}$ where then shown to satisfy
\be
\begin{aligned}
  L_{0}(t,q)-L_{0}(qt,q)&=-qtg^{(0,0)}(qt,q)\\
  L_{1}(t;q)-L_{1}(qt,q)&=-qtg^{(0,1)}(qt,q).
\end{aligned}
\ee
Notice that functions satisfying these equations are unique up to the
addition of an elliptic function. Therefore, with exactly the same argument in
~\cite[Thm.4]{GGMW:trivial} with the addition of checking the principal parts
of the LHS and RHS we can then show that
\be
{\small
V(t,q)^{-1}
\begin{pmatrix}
0 & 1 & 0\\
0 & 0 & 1\\
q^{-2}t^{-1} & -q^{-2} & q^{-2}t^{-1}-2q^{-1}
\end{pmatrix}
=
\begin{pmatrix}
  -L_{1}(t,q) & -qt^2g^{(0,1)}(qt,q) & qt^2g^{(0,1)}(t,q)\\
  L_{0}(t,q)& qt^2g^{(0,0)}(qt,q) & -qt^2g^{(0,0)}(qt,q)\\
  1 & 0 & 0
\end{pmatrix}.
}
\ee

\noindent
Therefore, we see that the entries of $\Om_{V,S}$ are combinations of elementary
functions times $\calI(z+n+m\tau,\tau)$ for $m,n\in\{-2,-1,0,1,2\}$. Using the
explicit expressions for the monodromy one can prove it satisfies Equation
~\eqref{M41c} with equal weights $(2,1,0)$ and again using part (c) of Theorem
~\ref{thm.ST} complete the proof.
\end{proof}

\subsection{The $x$-deformation}

We will first discuss the two variable holonomic system given by
the homogeneous equations
\begin{subequations}
\be
\label{41x}
\begin{aligned}
tq^{-1}f(q^{-1}t,x,q)+(1-(x^{-1}+x)t)f(t,x,q)+tqf(qt,x,q)
=
0\,,
\end{aligned}
\ee
\be
\label{41x2}
\begin{aligned}
(1-qx)(1-q^{-1}x^{2})f(t,qx,q)&\\
-(x-1)^2(x+1)&(x^2t-x-(q^{-1}+q)t-x^{-1}+x^{-2}t)f(t,x,q)\\
+&(1-qx^2)(1-q^{-1}x)f(t,q^{-1}x,q)
=
0 \,,
  \end{aligned}
  \ee
\be
\label{41x3}
(1-xq)(f(t,qx,q)-x^{-1}f(qt,qx,q))=(1-x^{-1})(f(t,x,q)-qxf(qt,x,q)) \,.
\ee
\end{subequations}
This is not a random system of equations, instead they are the defining equations
of the $t$-deformation of the homogeneous $\hat{A}$-polynomial of the $4_{1}$
knot. This system appeared in~\cite[Eqn.(10)]{GK:desc} and
~\cite[Eqn.(134)]{GGM:peacock}, and it is $q$-holonomic of rank 2 in the variables
$(t,x)$. The system is symmetric under the involution $x \leftrightarrow x^{-1}$
(which corresponds to the Weyl symmetry in the color of the Jones polynomial of
the knot) and our solutions will also be invariant under this involution. As a
result the monodromy connecting $x=0$ to $x=\infty$ is the identity. 

To construct solutions, we will apply the Frobenius method to the
Equation~\eqref{41x}. This has Newton polygon depicted in Figure~\ref{f.41x1}.
Notice that the indicial polynomial $(\rind-1)^2$ for the top edge of
Equation~\eqref{41x1} now becomes $(\rind-x)(\rind-x^{-1})$ for Equation~\eqref{41x}.
We can then normalise so that the solutions satisfy the full system of equations
~\eqref{41x},~\eqref{41x2},~\eqref{41x3}. The solutions to this homogeneous
system are then given by
\be
\label{sols.41.x}
\begin{aligned}
f^{(-1)}(t,x,q)
&=
\frac{\theta(q^{-1}x;q)}{(1-x)\theta(q^{-1}t;q)(q;q)_{\infty}}\sum_{k=0}^{\infty}
\sum_{\ell=0}^{k}(-1)^{k}\frac{q^{k(k+1)/2}x^{2\ell-k}}{(q;q)_{\ell}
  (q;q)_{k-\ell}}t^{k}\\
f^{(1)}(t,x,\lambda_{1},q)
&=
\frac{\theta(t;q)(q;q)_{\infty}}{(1-x)\theta(x;q)}
\Lap_{1/2}\Bor_{1/2}\hat{f}^{(1)}(t,x,\lambda_{1},q)\\
g^{(0,x^{-1})}(t,x,q)
&=
\frac{\theta(q^{-1}x;q)\theta(tx;q)(qx^{2};q)_{\infty}}{\theta(t;q)
  \theta(x^{2};q)(1-x)(q;q)_{\infty}}\sum_{k=0}^{\infty}(-1)^{k}
\frac{q^{k(k+1)/2}x^{k}}{(q;q)_{k}(qx^{2};q)_{k}}t^{-k-1}\\
g^{(0,x)}(t,x,q)
&=
\frac{\theta(q^{-1}x^{-1};q)\theta(tx^{-1};q)(qx^{-2};q)_{\infty}}{\theta(t;q)
  \theta(x^{-2};q)(1-x^{-1})(q;q)_{\infty}}\sum_{k=0}^{\infty}(-1)^{k}
\frac{q^{k(k+1)/2}x^{-k}}{(q;q)_{k}(qx^{-2};q)_{k}}t^{-k-1}
\end{aligned}
\ee
where
\be
\hat{f}^{(1)}(t,x,q)
=
\sum_{k=0}^{\infty}\sum_{\ell=0}^{k}\frac{q^{\ell^{2}-k\ell}
  x^{2\ell-k}}{(q;q)_{\ell}(q;q)_{k-\ell}}t^{k}.
\ee
Notice that we have
\be
\Bor_{1/2}\hat{f}^{(1)}(t,x,q)
=
\sum_{k=0}^{\infty}\sum_{\ell=0}^{k}(-1)^{k}
\frac{q^{k(k+1)/4+\ell^{2}-k\ell}x^{2\ell-k}}{(q;q)_{\ell}(q;q)_{k-\ell}}t^{k}
\ee
which satisfies
\be
(1-q^{1/2}\xi^2)\Bor_{1/2}\hat{f}^{(1)}(\xi,q)+
(x+x^{-1})\xi\Bor_{1/2}\hat{f}^{(1)}(q^{1/2}\xi,q)-\Bor_{1/2}
\hat{f}^{(1)}(q\xi,q)=0.
\ee
We can complete this to include the inhomogeneous solutions to the system of
Equations~\eqref{41x.inhom}, \eqref{41x2.inhom}, \eqref{41x3.inhom}.
Equations~\eqref{41x.inhom} can be made third order homogenous which has
Newton polygon depicted in Figure~\ref{f.413}. The indicial polynomial of
the top edge now factors as $(\rind-x)(\rind-x^{-1})(\rind-1)$.
This gives two additional solutions as in Section~\ref{sec.4.inhom.x1}.
Firstly we have
\be
\hat{f}^{(0)}(t,x,q)=\sum_{k=0}^{\infty}(-1)^{k}q^{-k(k+1)/2}
(qx;q)_{k}(qx^{-1};q)_{k}t^{k} \,,
\ee
a divergent formal power series solution at $t=0$. Then
\be
  \Bor_{1}f^{(0)}(\xi,x,q)=\sum_{k=0}^{\infty}(qx;q)_{k}(qx^{-1};q)_{k}t^{k} \,,
\ee
which can be analytically continued away from $\xi\in q^{\BZ_{\leq0}}$ using the
relation
\be
(1-\xi)\Bor_{1}\hat{f}^{(0)}(\xi,q)+(x+x^{-1})q\xi \Bor_{1}\hat{f}^{(0)}(q\xi,q)
-q^2\xi \Bor_{1}\hat{f}^{(0)}(q^{2}\xi,q) = 1.
\ee
So we define
\be
\label{f0x}
f^{(0)}(t,x,\lambda_{2},q)
=
(\Lap_{1}\Bor_{1}\hat{f}^{(0)})(t,x,\lambda_{2},q)
=
\frac{1}{\theta(\lambda_{2};q)}\sum_{k\in\BZ}(-1)^{k}q^{k(k+1)}
\lambda_{2}^{k}\Bor_{1}\hat{f}^{(0)}(q^{k}\lambda_{2}t,x,q).
\ee
The second solution is then given by
\be
\label{gGM}
g^{(0)}(t,x,q)
=
\sum_{k=0}^{\infty}(-1)^{k}\frac{q^{k(k+1)/2}}{(x;q)_{k+1}(x^{-1};q)_{k+1}}t^{-k-1}.
\ee
We then take
\be
\label{UVex41x}
\begin{aligned}
  U(t,x,\lambda_{1},\lambda_{2},q) &=
  W(f^{(1)}(t,x,\lambda_{1},q), f^{(-1)}(t,x,q), f^{(0)}(t,x,\lambda_{2},q))\:,\\
V(t,q) &= W(g^{(0,x^{-1})}(t,x,q), g^{(0,x)}(t,x,q), g^{(0,1)}(t,x,q)) \,.
\end{aligned}
\ee
Before we give the proof note the immediate symmetries
\be
\begin{aligned}
g^{(0,x^{-1})}(t,x^{-1},q)
&=
g^{(0,x)}(t,x,q),\qquad
&
f^{(-1)}(t,x^{-1},q)
&=
f^{(-1)}(t,x,q),
\\
g^{(0,x)}(t,x^{-1},q)
&=
g^{(0,x^{-1})}(t,x,q),\qquad
&
f^{(-1)}(t,x^{-1},q)
&=
f^{(-1)}(t,x,q),
\\
g^{(0,1)}(t,x^{-1},q)
&=
g^{(0,1)}(t,x,q),\qquad
&
f^{(0)}(t,x^{-1},q)
&=
f^{(0)}(t,x,q)\,.
  \end{aligned}
\ee
\begin{proof}[Proof of Theorem~\ref{thm.41c}]
  First we prove the existence of the elliptic function $m_{2,1}$
  with the properties listed in Theorem~\ref{thm.41c}. 
  Note that the only restriction to an existence of an elliptic function
  with prescribed poles and residues is the vanishing of
  the sum of the residues on a fundamental domain. Here, there are six
  residues, and their sum vanishing is equivalent to
\be
\begin{small}
\begin{aligned}
  1
  &=\frac{\theta(q^{3/4}\lambda_{1}^{-1};q)\theta(q^{-1/4}x;q)\theta(q^{-3/4}x;q)
  \theta(q^{-3/4}\lambda_{1};q)}{2\theta(x;q)\theta(x^{-1};q)
  \theta(q^{-1}\lambda_{1};q)\theta(q^{-3/2}\lambda_{1};q)}\\
  &+\frac{\theta(-q^{3/4}\lambda_{1}^{-1};q)\theta(-q^{-1/4}x;q)\theta(-q^{-3/4}x;q)
  \theta(-q^{-3/4}\lambda_{1};q)}{2\theta(x;q)\theta(x^{-1};q)
  \theta(q^{-1}\lambda_{1};q)\theta(q^{-3/2}\lambda_{1};q)}\\
  &+\frac{\theta(q^{1/4}\lambda_{1}^{-1};q)\theta(q^{-1/4}x;q)\theta(q^{-3/4}x;q)
  \theta(q^{-1/4}\lambda_{1};q)}{2\theta(x;q)\theta(x^{-1};q)
  \theta(q^{-1/2}\lambda_{1};q)\theta(q^{-1}\lambda_{1};q)}\\
  &+\frac{\theta(-q^{1/4}\lambda_{1}^{-1};q)\theta(-q^{-1/4}x;q)\theta(-q^{-3/4}x;q)
  \theta(-q^{-1/4}\lambda_{1};q)}{2\theta(x;q)\theta(x^{-1};q)
  \theta(q^{-1/2}\lambda_{1};q)\theta(q^{-1}\lambda_{1};q)}\\
  &=\left(\frac{\theta(q^{3/4}\lambda_{1}^{-1};q)
  \theta(q^{-3/4}\lambda_{1};q)}{2
  \theta(q^{-1}\lambda_{1};q)\theta(q^{-3/2}\lambda_{1};q)}
  +\frac{\theta(q^{1/4}\lambda_{1}^{-1};q)
  \theta(q^{-1/4}\lambda_{1};q)}{2
  \theta(q^{-1/2}\lambda_{1};q)\theta(q^{-1}\lambda_{1};q)}\right)
  \frac{\theta(q^{-1/4}x;q)\theta(q^{-3/4}x;q)}{\theta(x;q)\theta(x^{-1};q)}\\
  &+\left(\frac{\theta(-q^{3/4}\lambda_{1}^{-1};q)
  \theta(-q^{-3/4}\lambda_{1};q)}{2
  \theta(q^{-1}\lambda_{1};q)\theta(q^{-3/2}\lambda_{1};q)}
  +\frac{\theta(-q^{1/4}\lambda_{1}^{-1};q)
  \theta(-q^{-1/4}\lambda_{1};q)}{2
  \theta(q^{-1/2}\lambda_{1};q)\theta(q^{-1}\lambda_{1};q)}\right)
  \frac{\theta(-q^{-1/4}x;q)\theta(-q^{-3/4}x;q)}{\theta(x;q)\theta(x^{-1};q)}\,.
\end{aligned}
\end{small}
\ee
To prove this identity, notice that the right hand side is elliptic in
$\lambda_{1}\mapsto q^{1/2}\lambda_{1}$ and $x\mapsto qx$. Moreover, $\lambda_{1}$
has only one potential simple pole in its fundamental domain and therefore this
is constant in $\lambda_{1}$. Then $x$ has a potential double pole which cancels.
Then specialising the RHS to $\lambda_{1}=q^{1/4}$ and $x=-q^{1/4}$ gives
\be
\frac{\theta(q^{1/2};q)
  \theta(q^{-1/2};q)}{2
  \theta(q^{-3/4};q)\theta(q^{-5/4};q)}
\frac{\theta(-1;q)\theta(-q^{-1/2};q)}{\theta(-q^{1/4};q)\theta(-q^{-1/4};q)}.
\ee
This can be proven to equal one by elementary means using the Jacobi triple
product identity. Therefore, such an $m_{2,1}(t,x,\lambda_{1},q)$ exists.
Explicitly, we have
\be
\label{m21explicit}
\begin{tiny}
\begin{aligned}
  &m_{2,1}(t,x,\lambda_{1},q)
  =
  \frac{-1}{2}\frac{\theta'(tx;q)}{\theta(tx;q)}
  +\frac{-1}{2}\frac{\theta'(tx^{-1};q)}{\theta(tx^{-1};q)}\\
  +&\frac{\theta(q^{3/4}\lambda_{1}^{-1};q)\theta(q^{-1/4}x;q)\theta(q^{-3/4}x;q)
  \theta(q^{-3/4}\lambda_{1};q)}{2\theta(x;q)\theta(x^{-1};q)
  \theta(q^{-1}\lambda_{1};q)\theta(q^{-3/2}\lambda_{1};q)}
\frac{\theta'(t\lambda_{1} q^{1/4})}{\theta(t\lambda_{1} q^{1/4})}%\\
  +\frac{\theta(- q^{3/4}\lambda_{1}^{-1};q)\theta(- q^{-1/4}x;q)\theta(- q^{-3/4}x;q)
  \theta(- q^{-3/4}\lambda_{1};q)}{2\theta(x;q)\theta(x^{-1};q)
  \theta(q^{-1}\lambda_{1};q)\theta(q^{-3/2}\lambda_{1};q)}
\frac{\theta'(-t\lambda_{1} q^{1/4})}{\theta(-t\lambda_{1} q^{1/4})}\\
  +&\frac{\theta(q^{1/4}\lambda_{1}^{-1};q)\theta(q^{-1/4}x;q)\theta(q^{-3/4}x;q)
  \theta(q^{-1/4}\lambda_{1};q)}{2\theta(x;q)\theta(x^{-1};q)
  \theta(q^{-1/2}\lambda_{1};q)\theta(q^{-1}\lambda_{1};q)}
\frac{\theta'(t\lambda_{1} q^{3/4})}{\theta(t\lambda_{1} q^{3/4})}%\\
  +\frac{\theta(- q^{1/4}\lambda_{1}^{-1};q)\theta(- q^{-1/4}x;q)\theta(- q^{-3/4}x;q)
  \theta(- q^{-1/4}\lambda_{1};q)}{2\theta(x;q)\theta(x^{-1};q)
  \theta(q^{-1/2}\lambda_{1};q)\theta(q^{-1}\lambda_{1};q)}
\frac{\theta'(-t\lambda_{1} q^{3/4})}{\theta(-t\lambda_{1} q^{3/4})}\\
  &\qquad\qquad\qquad\quad-\frac{-1}{2}\frac{\theta'(x;q)}{\theta(x;q)}
  -\frac{-1}{2}\frac{\theta'(x^{-1};q)}{\theta(x^{-1};q)}%\\
  -\frac{\theta(q^{3/4}\lambda_{1}^{-1};q)\theta(q^{-1/4}x;q)\theta(q^{-3/4}x;q)
  \theta(q^{-3/4}\lambda_{1};q)}{2\theta(x;q)\theta(x^{-1};q)
  \theta(q^{-1}\lambda_{1};q)\theta(q^{-3/2}\lambda_{1};q)}
\frac{\theta'(\lambda_{1} q^{1/4})}{\theta(\lambda_{1} q^{1/4})}\\
-&\frac{\theta(- q^{3/4}\lambda_{1}^{-1};q)\theta(- q^{-1/4}x;q)
  \theta(- q^{-3/4}x;q)\theta(- q^{-3/4}\lambda_{1};q)}{2\theta(x;q)\theta(x^{-1};q)
  \theta(q^{-1}\lambda_{1};q)\theta(q^{-3/2}\lambda_{1};q)}
\frac{\theta'(-\lambda_{1} q^{1/4})}{\theta(-\lambda_{1} q^{1/4})}%\\
  -\frac{\theta(q^{1/4}\lambda_{1}^{-1};q)\theta(q^{-1/4}x;q)\theta(q^{-3/4}x;q)
  \theta(q^{-1/4}\lambda_{1};q)}{2\theta(x;q)\theta(x^{-1};q)
  \theta(q^{-1/2}\lambda_{1};q)\theta(q^{-1}\lambda_{1};q)}
\frac{\theta'(\lambda_{1} q^{3/4})}{\theta(\lambda_{1} q^{3/4})}\\
&\qquad\qquad\qquad\qquad\qquad\qquad\qquad\qquad
-\frac{\theta(- q^{1/4}\lambda_{1}^{-1};q)\theta(- q^{-1/4}x;q)\theta(- q^{-3/4}x;q)
  \theta(- q^{-1/4}\lambda_{1};q)}{2\theta(x;q)\theta(x^{-1};q)
  \theta(q^{-1/2}\lambda_{1};q)\theta(q^{-1}\lambda_{1};q)}
\frac{\theta'(-\lambda_{1} q^{3/4})}{\theta(-\lambda_{1} q^{3/4})} \,.
\end{aligned}
\end{tiny}
\ee

Next we calculate the second column of the monodromy matrix using the same
argument as Equation~\eqref{f1=g01} from \cite{Morita:conn-mon}.
\be
\begin{small}
\begin{aligned}
&\hat{f}^{(-1)}(t,x;q)%\\ &
=
\oint_{0}\Bor_{-1}(g)(\xi,x;q)\theta(t/\xi;q)\frac{d\xi}{2\pi i \xi}%\\ &
=
\oint_{0}\frac{\theta(t/\xi;q)}{(x\xi;q)_{\infty}
  (x^{-1}\xi;q)_{\infty}}\frac{d\xi}{2\pi i \xi}\\
&=
-\sum_{k=0}^{\infty}\left(\Res_{\xi=x^{\pm}q^{-k}}\right)
\frac{\theta(t/\xi;q)}{(x\xi;q)_{\infty}(x^{-1}\xi;q)_{\infty}}
\frac{d\xi}{2\pi i \xi}\\
% &=
% -\sum_{k=0}^{\infty}\left(\Res_{\ve=0}\right)
% \frac{\theta(txq^{k}e^{\ve};q)}{(q^{-k}e^{-\ve};q)_{\infty}
%   (x^{-2}q^{-k}e^{-\ve};q)_{\infty}}
% \frac{-x^{-1}q^{-k}e^{-\ve}d\ve}{2\pi i x^{-1}q^{-k}e^{-\ve}}\\
% &\quad\quad-\sum_{k=0}^{\infty}\left(\Res_{\ve=0}\right)
% \frac{\theta(tx^{-1}q^{k}e^{\ve};q)}{
%   (x^{2}q^{-k}e^{-\ve};q)_{\infty}
%   (q^{-k}e^{-\ve};q)_{\infty}}
% \frac{-xq^{-k}e^{-\ve}d\ve}{2\pi i xq^{-k}e^{-\ve}}\\
&=
-\sum_{k=0}^{\infty}\left(\Res_{\ve=0}\right)
\frac{\theta(txq^{k}e^{\ve};q)}{
  (q^{-1}e^{-\ve};q^{-1})_{k}
  (e^{-\ve};q)_{\infty}(x^{-2}q^{-1}
  e^{-\ve};q^{-1})_{k}(x^{-2}e^{-\ve};q)_{\infty}}
\frac{-d\ve}{2\pi i}\\
&\quad\quad-\sum_{k=0}^{\infty}\left(\Res_{\ve=0}\right)
\frac{\theta(tx^{-1}q^{k}e^{\ve};q)}{
  (x^{2}q^{-1}e^{-\ve};q^{-1})_{k}(x^{2}
  e^{-\ve};q)_{\infty}(q^{-1}e^{-\ve};q^{-1})_{k}(
  e^{-\ve};q)_{\infty}}\frac{-d\ve}{2\pi i}\\
% &=
% \sum_{k=0}^{\infty}\frac{\theta(txq^{k};q)}{
%   (q^{-1};q^{-1})_{k}(q;q)_{\infty}(x^{-2}q^{-1};q^{-1})_{k}(x^{-2};q)_{\infty}}
% +\sum_{k=0}^{\infty}\frac{\theta(tx^{-1}q^{k};q)}{
%   (x^{2}q^{-1};q^{-1})_{k}(x^{2};q)_{\infty}(q^{-1};q^{-1})_{k}(q;q)_{\infty}}\\
&=
\sum_{k=0}^{\infty}\frac{\theta(txq^{k};q)(qx^{2};q)_{\infty}}{
  (q^{-1};q^{-1})_{k}(x^{-2}q^{-1};q^{-1})_{k}\theta(x^{2};q)}
+\sum_{k=0}^{\infty}\frac{\theta(tx^{-1}q^{k};q)(qx^{-2};q)_{\infty}}{
  (x^{2}q^{-1};q^{-1})_{k}\theta(x^{-2};q)(q^{-1};q^{-1})_{k}}\\
&=
\frac{\theta(tx;q)(qx^{2};q)_{\infty}}{\theta(x^{2};q)}
\sum_{k=0}^{\infty}(-1)^{k}
\frac{q^{k(k+1)/2}x^{k}t^{-k}}{(q;q)_{k}(x^{2}q;q)_{k}}
+\frac{\theta(tx^{-1};q)(qx^{-2};q)_{\infty}}{\theta(x^{-2};q)}
\sum_{k=0}^{\infty}(-1)^{k}
\frac{q^{k(k+1)/2}x^{-k}t^{-k}}{(q;q)_{k}(qx^{-2};q)_{k}}
\\
&=
-\frac{\theta(q^{-1}t;q)(1-x)(q;q)_{\infty}}{
  \theta(q^{-1}x;q)}g^{(0,x^{-1})}(t,x,q)
-\frac{\theta(q^{-1}t;q)(1-x^{-1})(q;q)_{\infty}}{
  \theta(q^{-1}x^{-1};q)}g^{(0,x)}(t,x,q) \,.
\end{aligned}
\end{small}
\ee
Therefore, we see that
\be\label{fm1ingx}
\begin{aligned}
f^{(-1)}(t,x,q)
&=
-g^{(0,x^{-1})}(t,x,q)
-
g^{(0,x)}(t,x,q).
\end{aligned}
\ee
Now note that the bottom row of $\vM$ is given by $(0,0,1)$ simply from the fact
the inhomogeneous module contains the homogeneous as a sub-module and then the
inhomogeneity normalises the last entry. Therefore, from Lemma~\ref{lem.det41x},
we see that
\be
\begin{pmatrix}
1 & -1 & 0\\
-1 & -1 & 0\\
0 & 0 & 1
\end{pmatrix}^{-1}
\vM(t,x,\lambda,q)
=
\begin{pmatrix}
  \frac{\theta(x^{-2};q)
    \theta(t;q)^2
    (q;q)_{\infty}^3}{2\theta(q^{-1}x;q)^2\theta(tx;q)\theta(tx^{-1};q)} & 0 & *\\
* & 1 & *\\
0 & 0 & 1
\end{pmatrix}\,.
\ee
Consider the function
\be
\begin{aligned}
  \calE_{4}&:=f^{(1)}(t,x,\lambda_{1},q)-\frac{\theta(x^{-2};q)
  \theta(t;q)^2(q;q)_{\infty}^3}{2\theta(q^{-1}x;q)^2\theta(tx;q)\theta(tx^{-1};q)}
  \left(g^{(0,x^{-1})}(t,x,q)-g^{(0,x)}(t,x,q)\right)\\
  &\quad-m_{2,1}(t,x,\lambda_{1},q)\left(-g^{(0,x^{-1})}(t,x,q)-g^{(0,x)}(t,x,q)
  \right)\\
  &=
  f^{(1)}(t,x,\lambda_{1},q)-\frac{\theta(x^{-2};q)
  \theta(t;q)^2(q;q)_{\infty}^3}{2\theta(q^{-1}x;q)^2\theta(tx;q)\theta(tx^{-1};q)}
  \left(g^{(0,x^{-1})}(t,x,q)+g^{(0,x)}(t,x,q)\right)\\
  &\quad-m_{2,1}(t,x,\lambda_{1},q)f^{(-1)}(t,x,q)
\end{aligned}
\ee
where is $m_{2,1}$ has the properties stated in Theorem~\ref{thm.41c}.
Then by the definition of $m_{2,1}$ along with Lemma~\ref{lem.41.gx} and
Lemma~\ref{lem.resf1.x} we see that $\calE_{4}$ is holomorphic and satisfies
Equation~\eqref{41x} and therefore must be zero from Lemma~\ref{lem.41.x.holsol}.
This proves the first column. Now for the last column consider the function 
\be
\begin{aligned}
  \calE_{5}
  &:=
  f^{(0)}(t,\lambda_{2},x,q)
  -
  g^{(0,1)}(t,x,q)\\
  &-
  \frac{\theta(x^{-2};q)
    \theta(t;q)^2(q;q)_{\infty}^{3}}{2\theta(q^{-1}x;q)^2\theta(tx;q)
    \theta(tx^{-1};q)}\left(g^{(0,x^{-1})}(t,x,q)-g^{(0,x)}(t,x,q)\right)\\
  &-
  \left(\frac{\theta'(t\lambda_{2})}{\theta(t\lambda_{2})}
-\frac{\theta'(tx)}{2\theta(tx)}-\frac{\theta'(tx^{-1})}{2\theta(tx^{-1})}
-\frac{\theta'(\lambda_{2})}{\theta(\lambda_{2})}-\frac{1}{2}\right)
  \\ & \times 
  \left(-g^{(0,x^{-1})}(t,x,q)-g^{(0,x)}(t,x,q)\right).
\end{aligned}
\ee
This is holomorphic from Lemma~\ref{lem.41.gx}, Lemma~\ref{lem.resf1.x} and
Lemma~\ref{lem.resf0.x}. Moreover, $\calE_{5}$ satisfies Equation~\eqref{41x}
and therefore vanishes from Lemma~\ref{lem.41.x.holsol}.

The entries of the RHS
of Equation~\eqref{m3x3x} have the following transformation properties under
the $S$ matrix
\be
\begin{small}
\begin{aligned}
  \frac{\theta(x^{-2};\tq)
    \theta(t;\tq)^2(\tq;\tq)_{\infty}^{3}}{2\theta(\tq^{-1}x;\tq)^2
    \theta(tx;\tq)\theta(tx^{-1};\tq)}
  &=\tau\frac{\theta(x^{-2};q)
    \theta(t;q)^2(q;q)_{\infty}^3}{2\theta(q^{-1}x;q)^2\theta(tx;q)
    \theta(tx^{-1};q)}\\
  \left(\frac{\theta'(\ti t\ti \lambda_{2})}{\theta(\ti t\ti \lambda_{2})}
    -\frac{\theta'(\ti t\ti x)}{2\theta(\ti t\ti x)}
    -\frac{\theta'(\ti t\ti x^{-1})}{2\theta(\ti t\ti x^{-1})}
  -\frac{\theta'(\ti \lambda_{2})}{\theta(\ti \lambda_{2})}-\frac{1}{2}\right)
  &=\tau\left(\frac{\theta'(t\lambda_{2})}{\theta(t\lambda_{2})}
  -\frac{\theta'(tx)}{2\theta(tx)}-\frac{\theta'(tx^{-1})}{2\theta(tx^{-1})}
  -\frac{\theta'(\lambda_{2})}{\theta(\lambda_{2})}-\frac{1}{2}\right)\\
\end{aligned}
\end{small}
\ee
while $m_{2,1}$ has three elements in it's $\SL_{2}(\BZ)$ orbit
$m_{2,1},m_{2,1}|_{1}S,m_{2,1}|_{1}T$. To see this notice that $m_{2,1}|_{1}T^2$
simply permutes the terms in RHS of Equation~\eqref{m21explicit}. Then using
Equations~\eqref{thetaS}~\eqref{S.mod.thd} we can explicitly compute
$m_{2,1}|_{1}S$ and $m_{2,1}|_{1}TS$ which can be used to show that
$m_{2,1}|_{1}ST=m_{2,1}|_{1}S$ and $m_{2,1}|_{1}TS=m_{2,1}|_{1}T$.
Altogether, this shows the monodromy satisfies Equation~\eqref{M41x}.

\begin{lemma}
  \label{lem.41.x.holsol}
  If $h(t,x,q)$ is holomorphic for $t\in\BC^{\times}$ and satisfies
  Equation~\eqref{41x} then it vanishes.
\end{lemma}

\begin{proof}
  If $h(t,x,q)$ is holomorphic for $t\in\BC^{\times}$ then is has an Laurent
  series expansion
\be
h(t,x,q)=\sum_{k\in\BZ}\a_{k}(x,q)t^{k}.
\ee
Therefore Equation~\eqref{41x} determines the coefficients $\a_{k}(x,q)$ from
two initial conditions say $\a_{0}$ and $\a_{-1}$. However for any non-zero choice
of $(\a_{0},\a_{-1})$ we find that $\a_{k}\sim O(q^{-k^{2}/2})$ for $k\sim-\infty$
and therefore it is divergent unless $\a_{0}=\a_{-1}=0$. Notice the difference here
with that of Corollary~\ref{cor:41holsol} comes from the fact that $\a_{-1}=0$ does
not imply that $\a_{k}=0$ for $k<0$.
\end{proof}
  
\begin{lemma}
  \label{lem.41.gx}
We have
\be
\begin{aligned}
  \Res_{t=q^{m}x}\frac{\theta(t;q)^{2}}{\theta(tx;q)
\theta(tx^{-1};q)}g^{(0,x^{-1})}(t,x,q)\frac{dt}{2\pi i t}
  &=
  \frac{-\theta(x;q)^2f^{(-1)}(q^{m}x,x,q)}{\theta(x^{2};q)(q;q)_{\infty}^{3}}\\
  \Res_{t=q^{m}x^{-1}}\frac{\theta(t;q)^{2}}{\theta(tx;q)
\theta(tx^{-1};q)}g^{(0,x)}(t,x,q)\frac{dt}{2\pi i t} 
  &=
  \frac{-\theta(x^{-1};q)^2f^{(-1)}(q^{m}x^{-1},x,q)}{
    \theta(x^{-2};q)(q;q)_{\infty}^{3}}\\
\end{aligned}
\ee
while $g^{(0,1)}(t,x,q)$ is holomorphic for $t\in\BC^{\times}\cup\{\infty\}$.
\end{lemma}

\begin{proof}
This follows from Equation~\eqref{fm1ingx} and the fact that
\be
\Res_{t=q^{m}x^{\pm}}\frac{\theta(t;q)^{2}}{\theta(tx;q)
\theta(tx^{-1};q)}\frac{dt}{2\pi i t}
=
\frac{\theta(x^{\pm};q)^2}{\theta(x^{\pm2};q)(q;q)_{\infty}^3}\,.
\ee
\end{proof}

\begin{lemma}
  \label{lem.resf1.x}
  \rm{(a)} The function $\Bor_{1/2}\hat{f}^{(1)}(\xi,x,q)$ has simple poles at
$\xi=\pm q^{-1/4-\BZ_{\geq0}/2}$ with residue
\be
\begin{aligned}
\Res_{\xi=\pm q^{-1/4-n/2}}\Bor_{1/2}\hat{f}^{(1)}(\xi,x,q)\frac{d\xi}{2\pi i\xi}
&=
\frac{-\theta(\pm q^{-1/4}x;q^{1/2})}{2(q;q)_{\infty}^{2}}
\sum_{\ell=0}^{n}(\pm1)^{n}
\frac{q^{\frac{n(n+2)}{4}}x^{2\ell-n}}{(q;q)_{\ell}(q;q)_{n-\ell}}.
\end{aligned}
\ee
  \rm{(b)}
  The function $f^{(1)}(t,x,\lambda_{1},q)$ has simple poles at
  $t=\pm q^{-1/4+\BZ}\lambda_{1}^{-1}$
  with residue
  \be
\begin{aligned}
  &\Res_{t=\pm q^{-1/4-n}\lambda_{1}^{-1}}f^{(1)}(t,x,\lambda_{1},q)
  \frac{dt}{2\pi it}\\
&=
\frac{\theta(\pm q^{3/4}\lambda_{1}^{-1};q)\theta(\pm q^{-1/4}x;q)
  \theta(\pm q^{-3/4}x;q)
  \theta(\pm q^{-3/4}\lambda_{1};q)}{2\theta(x;q)
  \theta(x^{-1};q)
  \theta(q^{-1}\lambda_{1};q)\theta(q^{-3/2}\lambda_{1};q)}
f^{(-1)}(\pm q^{-1/4-n} \lambda_{1}^{-1},x,q) \,.
\end{aligned}
\ee
\end{lemma}
\begin{proof}
Firstly note that the singularities are determined by the functional equation
for $\Bor_{1/2}\hat{f}^{(1)}(\xi,x,q)$. For part (a), recall the auxiliary
function $H_{r}(\xi,q)$ defined in Equation~\eqref{aux.fun.H}. Note that
from the power series at $\xi=0$, we see that
\be
\Bor_{1/2}\hat{f}^{(1)}(q^{-1/4}\xi,x,q)
=
\sum_{r\in\BZ}(-1)^{r}q^{r^{2}/4}x^{r}\xi^{r}H_{r}(\xi,q).
\ee
Let
\be
R_{\pm}(n,x,q)=\Res_{\xi=\pm q^{-1/4-n/2}}
\Bor_{1/2}\hat{f}^{(1)}(\xi,x,q)\frac{d\xi}{2\pi i\xi}.
\ee
Then, from Equation~\eqref{res.aux.fun.H},
\be
\begin{aligned}
R_{\pm}(0,x,q)
&=
\Res_{\xi=\pm 1}\sum_{r\in\BZ}(-1)^{r}q^{r^{2}/4}\xi^{r}x^{r}H_{r}(\xi,q)
\frac{d\xi}{2\pi i\xi}
&=
\frac{-\theta(\pm q^{-1/4}x;q^{1/2})}{2(q;q)_{\infty}^{2}}.
\end{aligned}
\ee
Now from the functional equation of $\Bor_{1/2}\hat{f}^{(1)}(\xi,x,q)$, we can
deduce that
\be
\begin{aligned}
&\Res_{\xi=\pm q^{-1/4-n/2}}(1-q^{1/2}\xi^2)\Bor_{1/2}\hat{f}^{(1)}(\xi,q)+
(x+x^{-1})\xi\Bor_{1/2}\hat{f}^{(1)}(q^{1/2}\xi,q)-\Bor_{1/2}\hat{f}^{(1)}(q\xi,q)
\frac{d\xi}{2\pi i\xi}\\
&=
(1-q^{-n})R_{\pm}(n,x,q)\pm(x+x^{-1})q^{-1/4-(n-1)/2}R_{\pm}(n-1,x,q)
-R_{\pm}(n-2,x,q)=0.
\end{aligned}
\ee
Note that $R_{\pm}(n,x,q)=0$ for $n<0$ and that when $R_{\pm}(-1,x,q)$ is zero we
have a unique solution determined by $R_{\pm}(0,x,q)$. One can then check that
\be
R_{\pm}(n,x,q)
=
\frac{-\theta(\pm q^{-1/4}x;q^{1/2})}{2(q;q)_{\infty}^{2}}
\sum_{\ell=0}^{n}(\pm1)^{n}\frac{q^{\frac{n(n+2)}{4}}
  x^{2\ell-n}}{(q;q)_{\ell}(q;q)_{n-\ell}} \,.
\ee
For part (b), we compute
{\small
\be
\begin{aligned}
  &\Res_{t=\pm q^{-1/4-n}\lambda_{1}^{-1}}f^{(1)}(t,x,\lambda_{1},q)
  \frac{dt}{2\pi it}\\
&=
\Res_{t=\pm q^{-1/4-n}\lambda_{1}^{-1}}
\frac{\theta(t;q)(q;q)_{\infty}}{(1-x)\theta(x;q)}
\Lap_{1/2}\Bor_{1/2}\hat{f}^{(1)}(t,x,\lambda_{1},q)\frac{dt}{2\pi it}\\
&=
\frac{\theta(\pm q^{-1/4-n}\lambda_{1}^{-1};q)(q;q)_{\infty}}{(1-x)\theta(x;q)
  \theta(\lambda_{1};q^{1/2})}
\sum_{k\in\BZ}(-1)^{k}q^{(k-2n)(k-2n-1)/4}\lambda_{1}^{-k+2n}R_{\pm}(k,x,q)\\
&=
\frac{\theta(\pm q^{-1/4-n}\lambda_{1}^{-1};q)(q;q)_{\infty}}{(1-x)\theta(x;q)
  \theta(\lambda_{1};q^{1/2})}
\frac{-\theta(\pm q^{-1/4}x;q^{1/2})}{2(q;q)_{\infty}^{2}q^{-n(2n+1)/2}\lambda_{1}^{-2n}}
\sum_{k=0}^{\infty}\sum_{\ell=0}^{k}(-1)^{k}\frac{
  q^{\frac{k(k+1)}{2}}x^{2\ell-k}}{(q;q)_{\ell}(q;q)_{k-\ell}}
(\pm q^{-1/4-n}\lambda_{1}^{-1})^{k}\\
&=
-\frac{\theta(\pm q^{-1/4-n}\lambda_{1}^{-1};q)\theta(\pm q^{-1/4}x;q^{1/2})}{
  2(q;q)_{\infty}(1-x)\theta(x;q)\theta(\lambda_{1};q^{1/2})}q^{n(2n+1)/2}
\lambda_{1}^{2n}
  \frac{(1-x)}{\theta(q^{-1}x;q)}
\hat{f}^{(-1)}(\pm q^{-1/4-n/2}\lambda_{1}^{-1},x,q)\\
&=
\frac{\theta(\pm q^{-1/4-n}\lambda_{1}^{-1};q)\theta(\pm q^{-1/4}x;q)
  \theta(\pm q^{-3/4}x;q)\theta(\pm q^{-5/4-n}
  \lambda_{1}^{-1};q)}{2x\theta(x;q)^2
  \theta(\lambda_{1};q)\theta(q^{-1/2}\lambda_{1};q)q^{-n(2n+1)/2}\lambda_{1}^{-2n}}
f^{(-1)}(\pm q^{-1/4-n}\lambda_{1}^{-1},x,q)\\
&=
\frac{\theta(\pm q^{-1/4}\lambda_{1}^{-1};q)\theta(\pm q^{-1/4}x;q)
  \theta(\pm q^{-3/4}x;q)\theta(\pm q^{-5/4}
  \lambda_{1}^{-1};q)}{2x\theta(x;q)^2
  \theta(\lambda_{1};q)\theta(q^{-1/2}\lambda_{1};q)}
f^{(-1)}(\pm q^{-1/4-n}\lambda_{1}^{-1},x,q)\\
&=
\frac{\theta(\pm q^{3/4}\lambda_{1}^{-1};q)\theta(\pm q^{-1/4}x;q)
  \theta(\pm q^{-3/4}x;q)
  \theta(\pm q^{-3/4}\lambda_{1};q)}{2\theta(x;q)
  \theta(x^{-1};q)\theta(q^{-1}\lambda_{1};q)\theta(q^{-3/2}\lambda_{1};q)}
f^{(-1)}(\pm q^{-1/4-n} \lambda_{1}^{-1},x,q) \,.
\end{aligned}
\ee
}
\end{proof}
\begin{lemma}
\label{lem.resf0.x}
\rm{(a)}
The function $\Bor_{1}\hat{f}^{(0)}(\xi,x,q)$ has simple poles at $\xi \in q^\BZ$
with residue
\be
\label{resBf0.x}
R(m,x,q)
:=
\Res_{\xi=q^{-m}}\Bor_{1}\hat{f}^{(0)}(\xi,x,q)\frac{d\xi}{2\pi i\xi}
= -(qx;q)_{\infty}(qx^{-1};q)_{\infty}\sum_{\ell=0}^{m}
\frac{q^{m}x^{2\ell-m}}{(q;q)_{\ell}(q;q)_{m-\ell}} \,.
\ee
\rm{(b)}
The function $f^{(0)}(t,x,\lambda_{2},q)$ has simple poles at
$t\in \lambda_{2}^{-1}q^{\BZ}$ with residue
\be
\label{resf0.x}
\Res_{t=q^{-m}\lambda_{2}^{-1}}f^{(0)}(t,x,\lambda_{2},q)\frac{dt}{2\pi it} =
f^{(-1)}(q^{-m}\lambda_{2}^{-1},x,q) \,.
\ee
\end{lemma}
\begin{proof}
For part (a), notice that
{\small
\be
\begin{aligned}
\Bor_{1}\hat{f}^{(0)}(\xi,x,q)
&=
\sum_{k=0}^{\infty}(qx;q)_{k}(qx^{-1};q)_{k}\xi^{k}  %\\ &
=
(qx;q)_{\infty}(qx^{-1};q)_{\infty}\sum_{k=0}^{\infty}
\frac{1}{(q^{k+1}x;q)_{\infty}(q^{k+1}x^{-1};q)_{\infty}}\xi^{k}\\
&=
(qx;q)_{\infty}(qx^{-1};q)_{\infty}\sum_{k=0}^{\infty}
\sum_{n,\ell=0}^{\infty}\frac{q^{kn+k\ell+n+\ell}
  x^{\ell-n}}{(q;q)_{n}(q;q)_{\ell}}\xi^{k}\\
&=
(qx;q)_{\infty}(qx^{-1};q)_{\infty}\sum_{n,\ell=0}^{\infty}
\frac{q^{n+\ell}x^{\ell-n}}{(q;q)_{n}(q;q)_{\ell}(1-q^{n+\ell}\xi)}\\
&=
(qx;q)_{\infty}(qx^{-1};q)_{\infty}\sum_{n=0}^{\infty}\sum_{\ell=0}^{n}
\frac{q^{n}x^{2\ell-n}}{(q;q)_{\ell}(q;q)_{n-\ell}(1-q^{n}\xi)}.
\end{aligned}
\ee
}
The residue then follows.
For part (b), the pole structure is clear noting that Equation~\eqref{f0x}
is convergent if all the terms are and that
$\Bor_{1}\hat{f}^{(0)}(q^{k}t\lambda_{2},x,q)$ has poles at
$t\in\lambda_{2}^{-1}q^{\BZ_{\leq -k}}$. We compute the residues
$\rho_m:=\Res_{t=q^{-m}\lambda_{2}^{-1}}f^{(0)}(t,x,\lambda_{2},q)\frac{dt}{2\pi it}$
as follows:
\be
\begin{aligned}
%\Res_{t=q^{-m}\lambda_{2}^{-1}}f^{(0)}(t,x,\lambda_{2},q)\frac{dt}{2\pi it}
  \rho_m
  &=
\Res_{t=q^{-m}\lambda_{2}^{-1}}(\Lap_{1}\Bor_{1}\hat{f}^{(0)})(t,x,\lambda_{2},q)
\frac{dt}{2\pi it}\\
&=
\Res_{t=q^{-m}\lambda_{2}^{-1}}\frac{1}{\theta(\lambda_{2},q)}
\sum_{k\in\BZ}(-1)^{k}q^{k(k+1)/2}\lambda_{2}^{k}\Bor_{1}\hat{f}^{(0)}(q^{k}
\lambda_{2}t,x,q)\frac{dt}{2\pi it}\\
&=
\frac{1}{\theta(\lambda_{2},q)}\sum_{k\in\BZ}(-1)^{k}q^{k(k+1)/2}
\lambda_{2}^{k}R(m-k,x,q)\\
% &=
% \frac{1}{\theta(\lambda_{2},q)}
% \sum_{k=0}^{\infty}(-1)^{m-k}q^{(m-k)(m-k+1)/2}\lambda_{2}^{m-k}R(k,q)\\
&=
\frac{1}{\theta(\lambda_{2},q)}(-1)^{m}q^{m(m+1)/2}\lambda_{2}^{m}
\sum_{k=0}^{\infty}(-1)^{k}q^{k(k-1)/2}(q^{-m}\lambda_{2}^{-1})^{k}R(k,x,q)\\
% &=
% \frac{-(q;q)_{\infty}^{2}}{\theta(\lambda_{2},q)}(-1)^{m}q^{m(m+1)/2}
% \lambda_{2}^{m}\sum_{k=0}^{\infty}(-1)^{k}q^{k(k-1)/2}
% (q^{-m}\lambda_{2}^{-1})^{k}\sum_{\ell=0}^{\infty}
% \frac{q^{k}}{(q;q)_{\ell}(q;q)_{k-\ell}}\\
&=
\frac{-(qx;q)_{\infty}(qx^{-1};q)_{\infty}}{
  \theta(\lambda_{2},q)}(-1)^{m}q^{m(m+1)/2}
\lambda_{2}^{m}\sum_{k,\ell=0}^{\infty}(-1)^{k}
\frac{q^{k(k+1)/2}x^{2\ell-k}}{
  (q;q)_{\ell}(q;q)_{k-\ell}}(q^{-m}\lambda_{2}^{-1})^{k}\\
% &=
% \frac{1}{\theta(q^{m}\lambda_{2},q)}\hat{f}^{(-1)}(q^{-m}\lambda_{2}^{-1},q)\\
&=
f^{(-1)}(q^{-m}\lambda_{2}^{-1},x,q) \,.
\end{aligned}
\ee
\end{proof}

The functional equation \eqref{41x} for $U,V$ defined in~\eqref{UVex41x}
can be written in the form
\be
U(qt,x,\lambda_{1},\lambda_{2},q)
=
A(t,x,q)U(t,x,\lambda_{1},\lambda_{2},q)
\quad\text{and}\quad
V(qt,x,q)
=
A(t,x,q)V(t,x,q)
\ee
where
\be
A(t,x,q)
=
\begin{pmatrix}
0 & 1 & 0\\
0 & 0 & 1\\
q^{-3} & q^{-3}t^{-1}-(1+x+x^{-1})q^{-2} & (1+x+x^{-1})q^{-1}-q^{-3}t^{-1}
\end{pmatrix}.
\ee
Therefore, we see that
\begin{equation*}
\det(U(qt,x,\lambda_{1},\lambda_{2},q))=\det(U(t,x,\lambda_{1},\lambda_{2},q))q^{-3},
\qquad \det(V(qt,x,q))=\det(V(t,x,q))q^{-3} \,.
\end{equation*}

\begin{lemma}
  \label{lem.det41x} 
  We have:
\be
\begin{aligned}
\det(U(t,x,\lambda_{1},\lambda_{2},q))
&=
-\frac{q^{-3}x}{(1-x)^2}t^{-3}\\
\det(V(t,x,q))
&=
\frac{\theta(q^{-1}x;q)^2\theta(tx;q)\theta(tx^{-1};q)}{\theta(x^{-2};q)
  \theta(t;q)^2(q;q)_{\infty}^3}\frac{q^{-3}x}{(1-x)^2}t^{-3}\,.
\end{aligned}
\ee
\end{lemma}

\begin{proof}
Firstly notice that
\be
\begin{tiny}
\begin{aligned}
&U(t,x,\lambda_{1},\lambda_{2},q)\\
&=
W(f^{(1)}(t,x,\lambda_{1},q),f^{(-1)}(t,x,q),f^{(0)}(t,x,\lambda_{2},q))\\
&=
\begin{pmatrix}
0 & 1 & 0\\
0 & 0 & 1\\
q^{-2}t^{-1} & -q^{-2} & q^{-2}t^{-1}-(x+x^{-1})q^{-1}
\end{pmatrix}
\begin{pmatrix}
0 & 0 & 1\\
f^{(1)}(t,x,\lambda_{1},q) & f^{(-1)}(t,x,q) & f^{(0)}(t,x,\lambda_{2},q)\\
f^{(1)}(qt,x,\lambda_{1},q) & f^{(-1)}(qt,x,\lambda_{1},q) &
f^{(0)}(qt,x,\lambda_{2},q)
\end{pmatrix}
\end{aligned}
\end{tiny}
\ee

\noindent
with a similar expression for $V$. Therefore,
\be
\label{det.3x3.in.2x2}
\begin{aligned}
  \det(U(t,x,\lambda_{1},\lambda_{2},q))
  &=
  q^{-2}t^{-1}(f^{(1)}(t,x,\lambda_{1},q)f^{(-1)}(qt,x,q)-f^{(1)}(qt,x,
\lambda_{1},q)f^{(-1)}(t,x,q))\\
  \det(V(t,x,q))
  &=
  q^{-2}t^{-1}(g^{(0,x^{-1})}(t,x,q)g^{(0,x)}(qt,x,q)
-g^{(0,x^{-1})}(qt,x,q)g^{(0,x)}(t,x,q))\,.
\end{aligned}
\ee
Now notice that both $\det(U(t,x,\lambda_{1},\lambda_{2},q))t^3$ and
$\det(V(t,x,q))t^3$ are elliptic functions in $t$. Furthermore,
$\det(U(t,x,\lambda_{1},q))$ has potentially simple poles in $t$ at
$t\in\pm q^{-1/4-\BZ/2}\lambda^{-1}$ and $t\in q^{\BZ}$. However, note that
the poles of $f^{(-1)}$ at $t\in q^{\BZ}$ cancel with zeros of $f^{(1)}$.
Lemma~\ref{lem.resf1.x} implies that
  \be
\begin{aligned}
& \Res_{t=\pm q^{-1/4-n}\lambda^{-1}}\det(U(t,x,\lambda_{1},\lambda_{2},q)q^{2}t)=\\
&\frac{\theta(\pm q^{3/4}\lambda_{1}^{-1};q)\theta(\pm q^{-1/4}x;q)
  \theta(\pm q^{-3/4}x;q)\theta(\pm q^{-3/4}\lambda_{1};q)
  %f^{(-1)}(\pm q^{-1/4-n}\lambda_{1}^{-1},x,q)
  %f^{(-1)}(\pm q^{3/4-n}\lambda_{1}^{-1},x,q)
}{2(q;q)_{\infty}^{2}\theta(x;q)\theta(x^{-1};q)
  \theta(q^{-1}\lambda_{1};q)\theta(q^{-3/2}\lambda_{1};q)}\\
   & \times f^{(-1)}(\pm q^{-1/4-n}
  \lambda_{1}^{-1},x,q)f^{(-1)}(\pm q^{3/4-n}
  \lambda_{1}^{-1},x,q) \\
&\quad-\frac{\theta(\pm q^{3/4}\lambda_{1}^{-1};q)\theta(\pm q^{-1/4}x;q)
  \theta(\pm q^{-3/4}x;q)
  \theta(\pm q^{-3/4}\lambda_{1};q)}{2(q;q)_{\infty}^{2}\theta(x;q)\theta(x^{-1};q)
  \theta(q^{-1}\lambda_{1};q)\theta(q^{-3/2}\lambda_{1};q)}\\
& \times f^{(-1)}(\pm q^{3/4-n}
  \lambda_{1}^{-1},x,q)f^{(-1)}(\pm q^{-1/4-n}
  \lambda_{1}^{-1},x,q) = 0 \,.
\end{aligned}
\ee
A similar calculation shows that
\be
\Res_{t=\pm q^{-3/4-n}\lambda_{1}^{-1}}\det(U(t,x,\lambda_{1},\lambda_{2},q))=0 \,.
\ee
Therefore, we see that $U(t,x,\lambda_{1},q)t^2$ is elliptic and holomorphic in
$t\in\BC^{\times}$. Therefore, it is constant in $t$. Now considering the limit as
$t\rightarrow0$, using the definition of $f^{(\pm1)}$ and their asymptotic expansions
(given by their formal power series expansions, by a version of Watson's lemma),
it follows that
\be
\begin{aligned}
\det(U(t,x,\lambda_{1},\lambda_{2},q))q^2t^3
&=
\lim_{t\rightarrow 0}\det(U(t,x,\lambda_{1},\lambda_{2},q))q^2t^3\\
&=
\lim_{t\rightarrow 0}\frac{\theta(t;q)\theta(q^{-1}x;q)}{\theta(t;q)
  \theta(x;q)(1-x)^2}t^2
-\lim_{t\rightarrow 0}\frac{\theta(qt;q)\theta(q^{-1}x;q)}{\theta(q^{-1}t;q)
  \theta(x;q)(1-x)^2}t^2\\
&=
-\lim_{t\rightarrow 0}q^{-1}\frac{\theta(t;q)\theta(x;q)}{\theta(t;q)
  \theta(x;q)(1-x)^2}
=
-\frac{q^{-1}x}{(1-x)^2} \,.
\end{aligned}
\ee
Now noting that $\theta(t;q)^2/\theta(tx;q)\theta(tx^{-1};q)$
is elliptic we see that %for $V$
\be
\det(V(t,x,q))t^{3}\frac{\theta(t;q)^2}{\theta(tx;q)\theta(tx^{-1};q)}
\ee
is elliptic and holomorphic in $t$ and therefore constant in $t$. Then we see that
\be
\begin{aligned}
  &\lim_{t\rightarrow\infty}\det(V(t,x,q))q^{2}t^{3}
  \frac{\theta(t;q)^2}{\theta(tx;q)\theta(tx^{-1};q)}\\
&=
\lim_{t\rightarrow\infty}(g^{(0,x^{-1})}(t,x,q)g^{(0,x)}(qt,x,q)
-g^{(0,x^{-1})}(qt,x,q)g^{(0,x)}(t,x,q))t^{2}\frac{\theta(t;q)^2}{\theta(tx;q)
  \theta(tx^{-1};q)}\\
&=
q^{-1}\frac{(qx^2;q)_{\infty}(qx^{-2};q)_{\infty}(1-x^2)
\theta(q^{-1}x;q)
\theta(q^{-1}x^{-1};q)}{
  (1-x)^2\theta(x^{2};q)\theta(x^{-2};q)(q;q)_{\infty}^2}\\
&=
-q^{-1}x^{2}\frac{\theta(x^2;q)\theta(q^{-1}x;q)
\theta(q^{-1}x^{-1};q)}{\theta(x^{2};q)\theta(x^{-2};q)(1-x)^2(q;q)_{\infty}^3}\,.
\end{aligned}
\ee
\end{proof}
To finish the proof of
Theorem~\ref{thm.41c} we will use the state integrals introduced in
~\cite[Equ.(139)]{GGMW:trivial} for $w$ in a neighbourhood of zero.
% \be\label{41x.sint}
% \frac{i}{\sqrt{\tau}}\int_{-i\sqrt{\tau}(\BR+i\ve)}
% \frac{(-q^{1/2}\e(x+w);q)_{\infty}(-q^{1/2}\e(x-w);q)_{\infty}}{(-\tq^{1/2}
%   \e(\frac{x+w}{\tau});\tq)_{\infty}(-\tq^{1/2}\e(\frac{x-w}{\tau});\tq)_{\infty}}
% %\Phi_{\sqrt{\tau}}\left(\frac{ix}{\sqrt{\tau}}\right)^2
% \frac{\e\left(\frac{1}{2\tau}x^{2}-\frac{zx}{\tau}\right)}{1
%   -\e\left(x-\frac{1}{2}-\frac{\tau}{2}\right)}dx \,.
% \ee
\be\label{41x.sint}
\int_{\BR+i\ve}
\Phi_{\sfb}(x+i\sfb^{-1}u)\Phi_{\sfb}(x-i\sfb^{-1}u)
\frac{\exp\left(-\pi i x^{2}-2\pi\frac{zx}{\sfb}\right)}{1
  +\tq^{1/2}\exp\left(-\frac{2\pi x}{\sfb}\right)}dx\,.
\ee
The factorisation of this integral was done in \cite[Equ.(149)]{GGMW:trivial}.
This module is again not self dual (see Proposition~\ref{41inhom.not.sd}) and
means that additional functions arise in the factorisation. It was shown, using
Equation~\eqref{thetaS}, that Equation~\eqref{41x.sint} factors as combinations
of elementary functions holomorphic in $\BC'$ times
\be
\calI(z,w,\tau)
=
g^{(0,1)}(\tilde{t},\ti x,\tq)+\tau g^{(0,x)}(\tilde{t},\ti x,\tq)
L^{(x^{-1})}(t,x,q)-\tau g^{(0,x^{-1})}(\tilde{t},\ti x,\tq)L^{(x)}(t,x,q)
\ee
where
\be
  L^{(x^{\mp})}(t,x,q)
  =
  \frac{
    \theta(t;q)(1-x^{\mp})(q;q)_{\infty}^2}{\theta(q^{-1}x^{\pm};q)
    \theta(tx^{\mp};q)}(qx^{\pm2};q)_{\infty}\sum_{k=0}^{\infty}(-1)^{k}
  \frac{q^{k(k+1)/2}x^{\pm k}}{(q;q)_{k}(qx^{\pm2};q)_{k}(1-q^{k}x^{\pm})}t^{-k}.
\ee
These functions $L^{(x^\mp)}$ can then be shown to satisfy
\be
\begin{aligned}
  L^{(x^{\mp})}(t,x,q)-L^{(x^{\mp})}(qt,x,q)&=\frac{\theta(x^{-2};q)
    \theta(t;q)^2(q;q)_{\infty}^3}{\theta(q^{-1}x;q)^2\theta(tx;q)
    \theta(tx^{-1};q)}\frac{(1-x)^2}{x}qtg^{(0,x^{\pm})}(qt,x,q).
\end{aligned}
\ee
Again this determines $L^{(x^{\mp})}$ up to the addition of an elliptic function
so checking the principal parts
\be
{\small
V(t,x,q)^{-1}
\begin{pmatrix}
0 & 1 & 0\\
0 & 0 & 1\\
q^{-2}t^{-1} & -q^{-2} & q^{-2}t^{-1}-(x+x^{-1})q^{-1}
\end{pmatrix}
=
\begin{pmatrix}
  -L^{(x)}(t,x,q) & * & *\\
  L^{(x^{-1})}(t,x,q)& * & *\\
  1 & 0 & 0
\end{pmatrix}.
}
\ee
where the $*$ are given by $\pm\det(V(t,x,q))^{-1}g^{(0,x^{\pm})}(q^{m}t,x,q)$
where $m=0,1$. Finally, one can use the functional equations to take $w$ away
from $0$ which gives the analytic continuation or can alter the contour
depending on $w$. Noting that the entries of $\Omega_{V,S}$ are combinations of
elementary functions times $\calI(z+n+m\tau,w,\tau)$, we see that $\Omega_{V,S}$
extends for $\tau\in\BC'$ and using the modularity of the monodromy and part (c)
of Theorem~\ref{thm.ST} completes the proof.
\end{proof}

\subsection{An analytic lift of the colored Jones polynomial}
\label{sub.CJ41lift}

We finish this section by giving a proof of Theorem~\ref{thm.CJ41lift}. The
main observation is that when $x=q^N$, the series $\hat{f}^{(0)}(t,q^{N},q)$
terminates to a polynomial of $t$, in which case the $q$-Borel transfrom, followed
by a $q$-Laplace transfrom is the identity. Explicitly, when $x=q^{N}$ for
$N\in\BZ_{\geq1}$ we have
  \be
  \begin{aligned}
\hat{f}^{(0)}(t,q^{N},q)
&=
\sum_{k=0}^{\infty}(-1)^{k}q^{-k(k+1)/2}
(q^{1+N};q)_{k}(q^{1-N};q)_{k}t^{k}\\
&=
\sum_{k=0}^{N-1}(-1)^{k}q^{-k(k+1)/2}
(q^{N+1};q)_{k}(q^{1-N};q)_{k}t^{k}\,.
  \end{aligned}
  \ee
  The theorem then follows from Equation~\eqref{qlapbort}.
\qed

\subsection{Specialisation $t=q^m$}
\label{sub.tqm}

In this short subsection, included for completeness, we briefly comment how
our analytic functions of $t$, specialised to $t=q^m$, become the known sequences
of $q$-series and $(x,q)$-series that have appeared in the literature. 
In~\cite{GZ:qseries} and~\cite{GGM:peacock,GGM,GGMW:trivial} the $q$-holonomic
modules are discrete versions of what we have considered in Section~\ref{sec.41}.
Here we describe how the solutions can be constructed from the ones presented here.
Consider a solution $f(t,q)$ of a $q$-difference equation
$$
a_r(t,q) f(q^r t, q) + a_{r-1}(t,q) f(q^{r-1}t,q) + \dots + a_0(t,q) f(t,q) =0 
$$
corresponding to an edge of slope $\kappa$ on the Newton polygon, and let
\be
\label{fm}
f_{m}(q)=(-1)^{\kappa m}q^{-\kappa m(m+1)/2}
\Res_{z=0}\theta(q^{m}\e(z),q)^{-\kappa}f(q^{m}\e(z),q)\frac{dz}{2\pi iz} \,.
\ee
Then we find that $f_m(q)$ satisfies the linear $q$-difference equation
$$
a_r(q^m,q) f_{r+m}(q) + a_{r-1}(q^m,q) f_{r+m-1}(q) + \dots
  + a_0(q^m,q) f_{m}(q) =0 \,.
$$
This follows from   
\be
\begin{small}
\begin{aligned}
  0=&\Res_{z=0}\left(a_r(q^m\e(z),q) f(q^{r+m}\e(z), q)
    + a_{r-1}(q^m\e(z),q) f(q^{r+m-1}\e(z),q) + \dots \right. \\ & \left.
    + a_0(q^m\e(z),q) f(q^m\e(z),q)\right)
  \frac{\e(\kappa z)dz}{\theta(\e(z);q)^{\kappa}2\pi iz}\\
  =&a_r(q^m,q) f_{r+m}(q) + a_{r-1}(q^m,q) f_{r+m-1}(q) + \dots
  + a_0(q^m,q) f_{m}(q) 
\end{aligned}
\end{small}
\ee
where the equality follows from holomorphicity of $\theta^{-\kappa}f$ or the
higher order of vanishing of the indicial polynomial.

For example, for the Equation~\eqref{41x1}, and its solution
solution $g^{(0,1)}$ defined in Equation~\eqref{ex41.g00}, we obtain that 
\be
    \Res_{z=0}g^{(0,1)}(q^{m}\e(z),q)\frac{dz}{2\pi iz}
    =
    \sum_{k=0}^{\infty}\left(k+\frac{1}{2}-m
  -2E^{(k)}_{1}(q)\right)(-1)^{k}\frac{q^{k(k+1)/2-km-m}}{(q;q)_{k}^{2}} \,,
\ee
a $q$-series that appears in \cite[Equ.13b]{GGM} and \cite[Equ.6]{GGMW:trivial}.

\subsection*{Acknowledgements} 

The authors wish to thank Thomas Dreyfus, Jie Gu, Rinat Kashaev, Marcos
Mari\~{n}o, Mikhail Kapranov, Matthias Storzer and Don Zagier for enlightening
conversations. The work of C.W. has been supported by the Max-Planck-Gesellschaft.

%%%%%%%%%%%%%%%%%%%%%%%%%%%%%%%%%%%%%%%%%%%%%%%%%%%%%%%%%%%%%%%%%%%%%%%%%%%% 
%%%%%%%%%%%%%%%%%%%%%%%%%%%%%%%%%%%%%%%%%%%%%%%%%%%%%%%%%%%%%%%%%%%%%%%%%%%%

%\bibliographystyle{hamsalpha}
\bibliographystyle{plain}
\bibliography{biblio}
\end{document}